\documentclass[11pt,a4paper]{article}
\usepackage[lmargin=1.0in,rmargin=1.0in,bottom=1.0in,top=1.0in,twoside=False]{geometry}
\usepackage{hchang}
\usepackage{mathtools}
\usepackage{graphicx}
\usepackage{enumerate}
\usepackage{tikz}
\usetikzlibrary{calc}
\usetikzlibrary{shapes}
\usetikzlibrary{decorations.pathmorphing}
\usetikzlibrary{decorations.pathreplacing, calligraphy}
\usetikzlibrary{backgrounds}
\usetikzlibrary{arrows.meta,calc}

\colorlet{vertexhl}{orange!75}   
\colorlet{edgehl}{orange!75}     
\colorlet{bsA}{red!55}
\colorlet{bsB}{blue!45}
\colorlet{bsC}{green!75!black}
\colorlet{bsD}{orange!80}
\colorlet{bsE}{violet!50}

\colorlet{fatedge}{red!60}      
\colorlet{fatvertex}{blue!40}   
\newcommand{\fatvertexopacity}{1} 
\newcommand{\fatedgeopacity}{1}     
\newcommand{\fatvertexrad}{5.8pt}   
\newcommand{\fatvertexwidth}{14pt} 
\newcommand{\fatedgerad}{2.3pt}     
\newcommand{\fatedgewidth}{5.5pt}  

\newcommand{\subclique}[4][1.5]{
\begin{tikzpicture}[scale=#1]
  \foreach \i in {1,...,#3} {
    \node[circle, draw, fill=black, inner sep=0.8pt] (N\i) at ({360/#3 * (\i - 1)}:2) {};
  }
  \foreach \i in {1,...,#3} {
    \foreach \j in {\i,...,#3} {
      \draw (N\i.center) -- (N\j.center)
      \foreach \k in {1,...,#2} {
      node[circle, draw, fill=black, inner sep=1pt, pos=\k/(#2+1)] (M\i\j\k) {}
      };
    }
  }
  \begin{scope}[on background layer]
    #4
  \end{scope}
\end{tikzpicture}
}
\newcommand{\spiderweb}[4][1.5]{
\begin{tikzpicture}[scale=#1]
  \node (A) at (0,0) {};
  \foreach \i in {1,...,#3} {
    \node[circle, draw, fill=black, inner sep=0.8pt] (N\i) at ({360/#3 * (\i - 1)}:2) {};
  }
  \foreach \i in {1,...,#3} {
    \foreach \j in {\i,...,#3} {
      \draw (N\i.center) -- (N\j.center)
      \foreach \k in {1,...,#2} {
      node[circle, draw, fill=black, inner sep=1pt, pos=\k/(#2+1)] (M\i\j\k) {}
      };
    }
  }
  \foreach \i in {1,...,#3} {
    \foreach \j in {\i,...,#3} {
        \foreach \k in {\i,...,#3} {
            \draw (M\i\j1.center) -- (M\i\k1.center);
        }
        \foreach \k in {1,...,\i} {
            \draw (M\i\j1.center) -- (M\k\i#2.center);
        }
    }
    \foreach \j in {1,...,\i} {
        \foreach \k in {1,...,\i} {
            \draw (M\j\i#2.center) -- (M\k\i#2.center);
        }
    }
  }
  \begin{scope}[on background layer]
    #4
  \end{scope}
\end{tikzpicture}
}

\tikzset{snake it/.style={decorate, decoration=snake}}

\usepackage{caption}   
\usepackage{subcaption} 
\usepackage{thmtools}
\usepackage{thm-restate}
\usepackage{float}
\usepackage[T1]{fontenc}

\usepackage{enumitem}

\usepackage{microtype}
\usepackage{comment}
\usepackage[capitalize,nameinlink]{cleveref}
\usepackage[normalem]{ulem}

\usepackage{stmaryrd}
\usepackage{marvosym}
\usepackage{booktabs}
\usepackage{multirow}
\usepackage{makecell}
\usepackage{nicefrac}

\newif\ifanonymous

\ifanonymous%
\else%
\nolinenumbers%
\fi%

\newtheorem{lemma}{Lemma}[section]
\newtheorem{theorem}[lemma]{Theorem}

\newtheorem{claim}[lemma]{Claim}

\newtheorem{definition}[lemma]{Definition}

\newtheorem{observation}{Observation}

\newcommand{\Oh}{\mathcal{O}}
\newcommand{\dist}{\mathsf{dist}}
\newcommand{\prof}[2]{\mathsf{prof}_{#1}[#2]}

\newcommand{\eps}{\varepsilon}

\newcommand{\len}{\mathsf{len}}
\setuptodonotes{size=\normalsize}

\newcommand{\CC}{\mathscr{C}}
\newcommand{\Cc}{\CC}

\newcommand{\wh}[1]{#1'}

\newcommand{\val}{\mathsf{cost}}
\newcommand{\OPT}{\mathsf{OPT}}

\newcommand{\Ball}{\mathrm{Ball}}
\renewcommand{\setminus}{-}

\newcommand{\Rnn}{\R_{\geq 0}}
\newcommand{\Rp}{\R_{>0}}

\renewcommand{\leq}{\leqslant}
\renewcommand{\geq}{\geqslant}
\renewcommand{\le}{\leqslant}
\renewcommand{\ge}{\geqslant}

\newcommand{\patternfont}[1]{\mathbf{#1}}
\newcommand{\clique}[2]{\patternfont{K}(#1,#2)}
\newcommand{\web}[2]{\patternfont{W}(#1,#2)}
\newcommand{\biclique}[3]{\patternfont{K}(#1,#2,#3)}
\newcommand{\biweb}[3]{\patternfont{W}(#1,#2,#3)}
\newcommand{\mixweb}[3]{\patternfont{M}(#1,#2,#3)}

\newif\ifclaimproofouterqed
\newenvironment{claimproof}{%
 \ifneedqed
  \claimproofouterqedtrue
 \else
  \claimproofouterqedfalse
 \fi
 \renewcommand{\QED}{\ensuremath{\lrcorner}}%
 \begin{proof}%
}{%
 \end{proof}%
 \ifclaimproofouterqed
  \global\needqedtrue
 \else
  \global\needqedfalse
 \fi
}

\newcommand{\niko}[1]{\todo[color=green!40,textcolor=black,bordercolor=black]{Niko: #1}}
\newcommand{\nikoin}[1]{\todo[inline,size=\normalsize,color=green!40,textcolor=black,bordercolor=black,caption={}]{Niko: #1}}

\newcommand{\affi}[1]{\textcolor{black!50}{#1}}
\usepackage{hyperref}

\begin{document}

\ifanonymous
\author{Anonymous}
\else
\author{
	Arnold Filtser\\{\small \affi{Bar-Ilan University}}\\\href{mailto:arnold.filtser@biu.ac.il}{\small arnold.filtser@biu.ac.il}\vspace{0.5cm}
    \and
	Hung Le\\{\small \affi{University of Massachusetts, Amherst}}\\\href{mailto:hungle@cs.umass.edu}{\small hungle@cs.umass.edu}
    \and
	Nikolas Mählmann\\{\small \affi{University of Warsaw}}\\\href{mailto:maehlmann@mimuw.edu.pl}{\small maehlmann@mimuw.edu.pl}
    \and
	Marcin Pilipczuk\\{\small \affi{University of Warsaw}}\\\href{mailto:malcin@mimuw.edu.pl}{\small malcin@mimuw.edu.pl}
    \and
	Michał Pilipczuk\\{\small \affi{University of Warsaw}}\\\href{mailto:michal.pilipczuk@mimuw.edu.pl}{\small michal.pilipczuk@mimuw.edu.pl}
}

\fi
\renewcommand\footnotemark{}
\title{\huge{Fatness and Flatness}\thanks{%
\begin{minipage}[t]{0.30\textwidth}
\raisebox{0.3cm}{\includegraphics[scale=0.15]{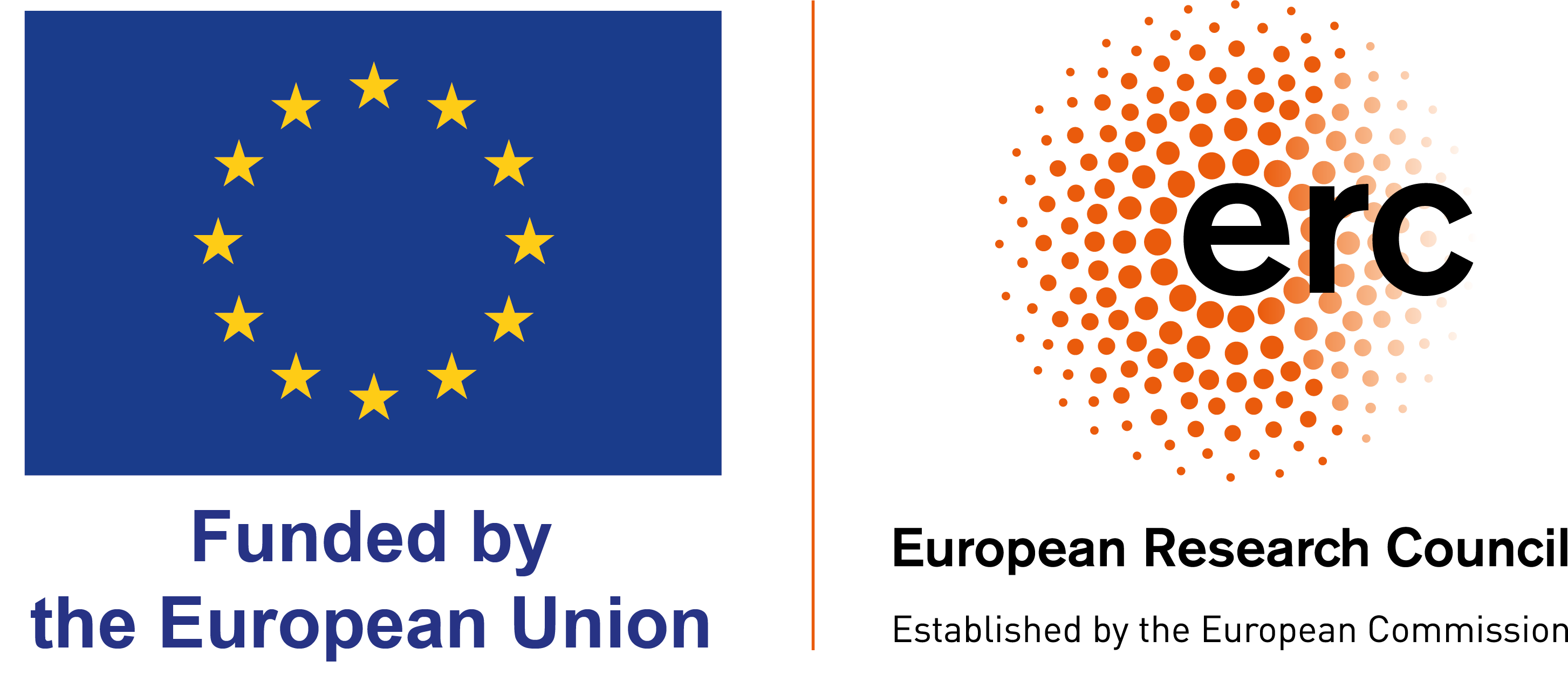}}
\end{minipage}\hfill
\begin{minipage}[b]{0.66\textwidth}\footnotesize
NM and MaP received funding from the ERC grant BUKA (No.~101126229).
NM also received funding from the European Union through an ERA Fellowship with grant agreement No.\ 101334340 -- LoCoMoDe.
MiP was supported by the project BOBR that has received funding from ERC
under the European Union's Horizon 2020 research and innovation programme,
grant agreement No.~948057. HL was supported by NSF grant CCF-2517033 and NSF CAREER Award CCF-2237288.
AF was supported by the Israel Science Foundation (grant No. 1042/22).
\end{minipage}
\begin{minipage}[b]{\textwidth}
\medskip
Substantial part of the research has been completed at the workshop `Global Structure and Geometry of Graphs' at MATRIX in Creswick, Australia, April 2026.
We thank MATRIX for an extremely comfortable and productive research environment.
\end{minipage}}}

\date{}

\maketitle

\begin{abstract}
 \emph{Fat minors} are the metric analog of graph minors that are tailored to the analysis of structure in metric (edge-weighted) graphs and, more generally, metric spaces equipped with a suitable notion of shortest paths. Despite a large interest in this notion, not much is known about the structure of metric graphs excluding a fixed fat minor.

 In this work we prove that if a metric graph $G$ excludes a fixed graph $H$ as a $\delta$-fat minor, for some $\delta>0$, then $G$ enjoys the metric analog of \emph{flatness} (also known as \emph{uniform quasi-wideness}) --- a structural property playing a fundamental role in the field of Sparsity. In essence, our flatness result says that for any $\alpha\geq \beta$ that are large enough compared to $\delta$, in every large enough set $A$ in $G$ one can find a sizable subset $B$ that becomes $\alpha$-scattered after removing from $G$ a bounded number of balls of radius at most $\beta$. We call this property \emph{drill-flatness}. Notably, the proof only relies on excluding fat minors that are \emph{shallow}: every branch set has radius at most $2\alpha$.

 As a corollary, we prove that metric graphs that exclude a fixed $\delta$-fat minor have bounded $\eps$-scatter dimension if we consider only $\eps$-scatters at distances that are large enough compared to~$\delta$. By combining this with the results of Abbasi et al.~[FOCS 2023], we conclude that the \textsc{$k$-Center} problem admits an approximation algorithm that, when applied to an edge-weighted graph excluding a fixed graph $H$ as a $\delta$-fat minor, finds a solution of cost at most $(1+\eps)\cdot\mathsf{OPT}+\Oh(\delta/\eps^2)$
 in time $\Oh_{H,\eps}(n^{\Oh(1)})$. This 
 statement can be extended to the general \textsc{Norm $k$-Clustering} problem studied by Abbasi et al., and thus also applies to \textsc{$k$-Median} and \textsc{$k$-Means}, and 
 constitutes one of the first algorithmic results achieved for general fat-minor-free metrics.


 We continue our study of drill-flatness in hereditary classes of (unweighted) graphs, where we obtain the following structural characterization in terms of induced minors. For every hereditary graph class $\CC$ the following are equivalent: (1) $\CC$ is induced-shallow-minor-free, (2) $\CC$ is fat-shallow-minor-free, (3) $\CC$ is drill-flat, (4) $\CC$ excludes certain variations of subdivided cliques as induced subraphs.
 This is an induced analog of the equivalence between the notions of flatness and of nowhere denseness --- one of the foundational results of Sparsity.

\end{abstract}

\vfill



\thispagestyle{empty}

\newpage
\thispagestyle{empty}
\tableofcontents
\thispagestyle{empty}
\newpage

\clearpage
\setcounter{page}{1}

\section{Introduction}\label{sec:intro}

\noindent\textbf{Clustering in structured metric spaces.}
Clustering forms a vast subarea of modern theoretical and applied research on algorithms.
Decades of research uncovered a curious interplay of various clustering problems
and assumptions on the properties of the input metric space that allow for efficient algorithms. 

In this work, we are interested in \emph{efficient parameterized approximation schemes}
(EPASes): $(1+\eps)$-approximation algorithms that run in time $h(k,\eps)\cdot \mathrm{poly}(n)$
for input consisting of $n$ points, asked to be divided into $k$ clusters. 
(Here, $h$ is an arbitrary computable function, usually exponential.)
A recent work of Abbasi, Banerjee, Byrka, Chalermsook, Gadekar, Khodamoradi, Marx, Sharma, and Spoerhase~\cite{ABBCGKMSS23} uncovered an invariant of a metric space,
dubbed \emph{$\eps$-scatter dimension}, whose boundedness allows for EPASes for a large
family of $k$-clustering problems under a wide umbrella of 
the \textsc{Norm $k$-Clustering} problem.

An \EMPH{$(\eps,r)$-scatter} of order $\ell$ in a metric space $G$ is a pair of sequences
of points $(a_1,\ldots,a_\ell)$ and $(b_1,\ldots,b_\ell)$ such that
\begin{itemize}
    \item $\dist_G(a_i,b_i)>(1+\eps)r$ for all $i\in \{1,\ldots,\ell\}$, and
    \item $\dist_G(a_i,b_j)\leq r$ for all $1\leq i<j\leq \ell$.
\end{itemize}
In~\cite{ABBCGKMSS23}, it is shown that edge-weighted planar graphs have bounded
\EMPH{$\eps$-scattered dimension} in the following sense:
the largest order of an $(\eps,r)$-scatter, for any $r\in \Rp$,
in an edge-weighted planar graph is bounded by a function of $1/\eps$.
Later, Bourneuf and Ma.~Pilipczuk~\cite{BP25} generalized this result to edge-weighted
$H$-minor-free graphs, for any fixed graph~$H$,
using metric analogs of tools from the graph-theoretic field of Sparsity~\cite{sparsity}, such as the weak coloring numbers
and flatness.

The \textsc{Norm $k$-Clustering} problem is defined as follows.
On input, we are given an integer $k\in \N$, a metric space $G$,
a set of potential facilities $F\subseteq G$, and a multiset of clients $C\subseteq G$.
There is also a fixed norm $\|\cdot\|$ on $\R^C$ that is polynomial-time computable and \EMPH{monotone} in the following sense:
if for $\bar x,\bar y\in \R_{\geq 0}^C$ we have that $\bar x$ is not larger than $\bar y$ on every coordinate, then $\|\bar x\|\leq \|\bar y\|$.
A \EMPH{solution} to the problem is a subset $S\subseteq F$ of at most $k$ facilities to be opened, and the \EMPH{cost} of such a solution is $\|(\dist_G(c,S)\colon c\in C)\|$, where $\dist_G(c,S)$ denotes the distance from $c$ to the closest facility in $S$. The goal is to find a solution with the smallest possible cost. By using norms $\ell_\infty$, $\ell_1$, and $\ell_2$, this general formulation models problems \textsc{$k$-Center}, \textsc{$k$-Median}, and \textsc{$k$-Means}.
In~\cite{ABBCGKMSS23}, it is proved that if the metric space $G$ has $\eps$-scatter dimension bounded by some function $h(1/\eps)$,
then the \textsc{Norm $k$-Clustering} on $G$ admits an EPAS.

The starting point of this work is the challenge to extend the applicability of the results of~\cite{ABBCGKMSS23}
to metric spaces excluding \EMPH{fat minors}.

\medskip

\noindent\textbf{Fat minors} are a generalization of the classic concept of graph minors that in addition incorporate metric constraints. Proposed by Georgakopoulos and Papasoglu~\cite{DBLP:journals/combinatorica/GeorgakopoulosP25}, they lie at the foundations of \emph{coarse graph theory}: a relatively young, but very dynamic branch of structural graph theory that investigates metric-related structure in graphs and more generally, metric spaces.

Let $G$ be a length space\footnote{A \EMPH{length space} is a metric space where for any two points $x,y$, the distance between $x$ and $y$ is equal to the infimum of the lengths of paths connecting $x$ and $y$. This notion can be understood as a robust generalization of edge-weighted graphs to a more continuous setting, where the metric space can be endowed with a well-behaved concept of paths that witness the distance. Readers unfamiliar with this notion may think of edge-weighted graphs instead, where every edge is modeled by a segment of respective length connecting the two endpoints.} and $H$ be an (unweighted, undirected) graph. For a distance parameter $\delta>0$, we say that $G$ contains $H$ as a \EMPH{$\delta$-fat minor} if there is a mapping $\phi$ from vertices and edges of $H$ to arc-connected subsets of $G$ so that:
\begin{itemize}
    \item whenever vertex $u$ is an endpoint of an edge $e$, the sets $\phi(u)$ and $\phi(e)$ intersect; and
    \item except for the above, for any two objects $o_1,o_2\in V(H)\cup E(H)$ we have $\dist_G(\phi(o_1),\phi(o_2))\geq \delta$.
\end{itemize}
Such a mapping $\phi$ is called a \EMPH{$\delta$-fat model} of $H$ in $G$, and sets $\{\phi(o)\colon o\in V(H)\cup E(H)\}$ are called the \EMPH{branch sets}. Thus, the disjointness requirement from the classic definition of a graph minor is replaced by the requirement of \EMPH{farness}: the branch sets of any two non-incident features of $H$ are not only required to be disjoint, but actually far from each other.

Georgakopoulos and Papasoglu~\cite{DBLP:journals/combinatorica/GeorgakopoulosP25} conjectured that if a length space $G$ excludes $H$ as a $\delta$-fat minor, for some fixed graph $H$ and $\delta>0$, then $G$ is quasi-isometric to an (unweighted) graph $G$ that is $H$-minor-free. Here, a \EMPH{quasi-isometry} between two metric spaces is a mapping that preserves the distances up to a constant multiplicative and additive distortion (depending on $H$ and $\delta$). If this so-called Fat Minor Conjecture was true, then the huge body of results and techniques developed over the decades for minor-free graphs could be lifted through the quasi-isometry to the setting of fat-minor-free length spaces.

Unfortunately, the Fat Minor Conjecture of Georgakopoulos and Papasoglu turned out to be false~\cite{DBLP:journals/corr/abs-2405-09383}, even in a strong sense~\cite{DBLP:journals/corr/abs-2508-15342}: one cannot replace the conclusion of $H$-minor-freeness with $H'$-minor-freeness, for any other graph $H'$ depending only on $H$. See also~\cite{DBLP:journals/corr/abs-2601-05761} for more recent works on smaller and more specific counterexamples.

However, even if the Fat Minor Conjecture is false, there is still hope that many of its supposed corollaries --- metric analogs of structural results for minor-free graphs --- could be still true. One such result was reported recently by Bonnet, Le, Pilipczuk, and Pilipczuk~\cite{bonnet2026coarsebalancedseparatorsfatminorfree}, who gave a coarse analog of the Alon-Seymour-Thomas Balanced Separator Theorem~\cite{AlonST90} for fat-minor-free graphs. However, besides this, the structure of fat-minor-free graphs and length spaces is essentially completely unexplored, and it is not clear which structural and algorithmic properties of minor-free graphs can be lifted, and which will fail in the fat-minor-free setting.

From the algorithmic point of view, in this work we are primarily interested in lifting results 
providing $(1+\eps)$-approximation algorithms in minor-free graphs. 
Note that the assumption of excluding a graph as a $\delta$-fat-minor implies structure
only at the ``scale'' larger than $\delta$, while below $\delta$, the metric can
be arbitrarily complicated. Thus, one should expect a correction, such as an additional
$+\Oh(\delta)$ additive error term in the approximation guarantee.

\bigskip

\noindent\textbf{Flatness,} also known as \EMPH{uniform quasi-wideness}, is a fundamental structural property of graph classes studied in the theory of Sparsity of Ne\v{s}et\v{r}il and Ossona de Mendez. We say that a graph class $\Cc$ is \EMPH{flat} if for every distance parameter $r\in \N$ there exists a constant $s_r\in \N$, called the \EMPH{budget}, and a function $N_r\colon \N\to \N$, called the \EMPH{overhead}, such that the following holds:
\begin{quote}
    For all $r,m\in \N$, graph $G\in \Cc$, and a vertex subset $A\subseteq V(G)$ with $|A|\geq N_r(m)$, there exist $B\subseteq A$ with $|B|\geq m$ and $X\subseteq V(G)$ with $|X|\leq s_r$ such that every path of length at most $r$ connecting two distinct vertices of $B$ must pass through $X$.
\end{quote}
The intuition behind this notion is the following: In any huge enough set of vertices $A$ one can find a sizable subset $B$ consisting of vertices that are pairwise far from each other, possibly after removing a bounded number of ``hub'' vertices (set $X$).

One of foundational results of Sparsity, proved by Ne\v{s}et\v{r}il and Ossona de Mendez~\cite{NesetrilM10}, is that a class of graphs is flat if and only if it is nowhere dense. Here, \EMPH{nowhere denseness} is the central notion of the theory: a class of graphs $\Cc$ is nowhere dense if for every $r\in \N$, one cannot obtain arbitrarily large cliques by contracting mutually disjoint subgraphs of radius at most $r$ in graphs from~$\Cc$. Note that for every fixed graph $H$, the class of $H$-minor-free graphs is nowhere dense, hence also flat. The connection between flatness and nowhere denseness underlies many key developments in Sparsity~\cite{sparsity}, and prominently features in its recent extensions to model-theoretically-well-behaved classes of dense graphs~\cite{DBLP:conf/icalp/DreierMST23,DreierMT24,MahlmannSiebertz26E, eleftheriadis2025extension, ghasemi2026weakly, mahlmann2025forbidden}.

\bigskip

\noindent\textbf{Fatness versus flatness.} 
A crucial realization made by Bourneuf and Ma.\ Pilipczuk in~\cite[Corollary~1.2]{BP25} is that there
is a useful metric analog of the notion of flatness. 
We formalize this (unnamed in~\cite{BP25}) concept under the name \EMPH{drill-flatness}, and prove that fat-minor-freeness implies drill-flatness. To state the result, we need some definitions.

Let $G$ be a length space and $\alpha\in \Rp$. A set $B\subseteq G$ shall be called \EMPH{$\alpha$-scattered}%
\footnote{We cannot avoid the name clash between an $\eps$-scatter and a set being $\alpha$-scattered,
both are established notions in prior literature. Luckily, in this work, $\eps$-scatters will appear in proofs only in \cref{sec:ksupplier}.}
if $\dist_G(a,b)>\alpha$ for all distinct $a,b\in B$. More generally, for $X\subseteq G$ we say that $B$ is \EMPH{$(\alpha,X)$-scattered} if $\dist_{G-X}(a,b)>\alpha$ for all distinct $a,b\in B\setminus X$. Here, $G-X$ is the length space derived from $G$ by removing all the points of $X$ and redefining the metric using paths that are disjoint from $X$.

Then, the definition of drill-flatness in length spaces reads as follows.
Note that, contrary to flatness, whose definition includes a universal quantifier over $r$,
here we focus only on a single distance scale $\alpha$. This is necessary to speak
about connections between drill-flatness and fatness, as in 
the latter the choice of $\delta$ fixes the scale.

\begin{definition}[Drill-flatness]\normalshape
    Let $\alpha,\beta\in \Rp$. We say that a length space $G$ is \EMPH{$(\alpha,\beta)$-drill-flat} with \EMPH{overhead} $N\colon \N\to \N$ and \EMPH{budget} $s\in \N$ if the following condition holds: 
    \begin{quote}
    For every $m\in \N$ and set $A\subseteq G$ with $|A|\geq N(m)$, there exists $B\subseteq A$ with $|B|=m$ and $X\subseteq G$ with $|X|\leq s$ such that $B$ is $(\alpha,\Ball(X,\beta))$-scattered.
    \end{quote}
    We call $\alpha$ the \EMPH{scatter radius} and $\beta$ the \EMPH{drill radius}.
\end{definition}

Intuitively speaking, in the definition of drill-flatness we replace deletion of a few ``hub'' vertices (set $X$) with deletion of a few balls of small radius $\beta$ (set $\Ball(X,\beta)$).

Next, our result will only assume exclusion of shallow fat minors: For $\alpha\in \R_{\geq 0}$, we say that a $\delta$-fat minor model $\phi$ is \EMPH{$\alpha$-shallow} if every branch set induces a subspace of radius at most $\alpha$.
Note that since shallow fat minors are more difficult to exhibit than fat minors, the shallow-fat-minor-free setting in which we will work generalizes the fat-minor-free one.

With all these definitions, our main structural result in the setting of metric spaces is the following:

\begin{restatable}[Drill-flatness in fat-minor-free spaces]{theorem}{metricuqwnew}\label{thm:metricuqwnew}
  For every $h\in \N$ and constants $\alpha,\delta\in \Rp$ and
  $\beta \geq 25 \cdot \sqrt{\alpha \delta}$,
  there exists $N \colon \N \to \N$ and $s \in \N$ such that
  every length space $G$ that excludes $K_h$ as a $2\alpha$-shallow $\delta$-fat minor
  is $(\alpha,\beta)$-drill-flat with overhead $N$ and budget $s$.

  Moreover, we have the following bounds, where the constants hidden in the $\Oh(\cdot)$ notation are universal, not depending on $h,\alpha,\beta,\delta$:
    \[N(m)\leq m^{2^{\Oh(h\cdot \alpha/\beta)}},\qquad \textrm{and} \qquad s\leq \Oh(h\cdot \alpha/\beta).\]
  In particular, $N$ and $s$ depend only on $h$ and the ratio $\alpha/\beta$.
\end{restatable}


On a high level, the proof of \cref{thm:metricuqwnew}, presented in \cref{sec:flatness-metric}, follows the general scheme of the proof of the equivalence of flatness and nowhere denseness. First, using a simple greedy argument we may assume that $A$ is $\beta$-scattered. Then, we perform $\Theta(\alpha/\beta)$ rounds of improvement, where in each round the scatteredness of $A$ increases by $c\beta$ for some small constant $c$, at the cost of drilling a few ``holes'' of radius $\beta$ in the space and replacing $A$ with a sizable subset of $A$.
For technical reasons we have to also slightly enlarge the holes that have been drilled previously, which is a feature that is not present in the original proof, and leads to a dependence of $\beta$ on $\alpha$.
For the proof of the improvement step, we mostly rely on the approach of Mi.\ Pilipczuk, Siebertz, and Toru\'nczyk~\cite{DBLP:conf/lics/PilipczukST18a}, but we actually simplify it: we show that instead of an elaborate reasoning using ideas from model theory such as the branching index, a simple extremal argument is sufficient. So at the end, from our reasoning one can in fact extract a simpler proof of the classic connection between flatness and nowhere denseness.

We leave as an open question, whether \cref{thm:metricuqwnew} can be strengthened to have $\beta$ only depend on $\delta$ and not on $\alpha$ (i.e., of the form $\beta = \Omega_h(\delta)$ or even $\beta = \Omega(\delta)$).

\bigskip

\noindent\textbf{Algorithmic consequences: $\eps$-scatter dimension and clustering in fat-minor-free spaces.}
Having established drill-flatness of length spaces excluding a fixed graph $H$ as a $\delta$-fat minor,
we turn back to $\eps$-scatter dimension. 
In~\cite{BP25}, the bound on $\eps$-scatter dimension is obtained from drill-flatness,
based on a realization that 
$\eps$-scatters are metric analogs of structures called \EMPH{semi-ladders}, introduced and studied in Sparsity by Fabia\'nski, Mi.~Pilipczuk, Siebertz, and Toru\'nczyk~\cite{FPST18}.
In particular, Fabia\'nski et al. proved a bound on the maximum orders of semi-ladders in nowhere dense classes that relies only on flatness.
This proof readily lifts to the metric setting with flatness replaced by drill-flatness as shown in~\cite[Lemma~4.1]{BP25}.
To bound the order of an $(\eps,r)$-scatter, the proof uses drill-flatness for scatter radius $\alpha \coloneqq 2r$.
Hence, we obtain the following.

\begin{restatable}{theorem}{scatter}
    \label{cor:scatter-bound}
	For every $h\in \N$ and $\eps\in (0,1)$, there exist constants $c,\ell\in \N$ such that the following holds:
    For every $\delta \in \Rp$ and every $r \geq c \delta$, 
    every length space $G$ that excludes $K_h$ as a $4r$-shallow $\delta$-fat-minor
    does not contain any $(\eps,r)$-scatter of order larger than $\ell$.
	
	Moreover, we have
	\[c\leq \Oh(1/\eps^2)\qquad\textrm{and}\qquad \ell\leq 2^{2^{\Oh(h/\eps)}}.\]
\end{restatable}

For comparison, Bourneuf and Ma.~Pilipczuk \cite{BP25} gave a bound of $\ell \leq 2^{(h/\eps)^{\Oh(h)}}$ in  $K_h$-minor-free graphs. Improving the bound of~\cref{cor:scatter-bound} in terms
of dependency of $\ell$ on $\eps$ remains an interesting open problem. 
As the bound on $c$ is concerned, 
we remark that improving \Cref{thm:metricuqwnew} to requiring only
$\beta = \Omega(\delta)$ automatically improves the bound to $c \leq \Oh(1/\eps)$. 
We conjecture that this should be the optimal bound on $c$ in \cref{cor:scatter-bound},
as then $r \geq c \delta$ is equivalent to $\delta = \Oh(\eps r)$, that is,
$\delta$ is smaller than the difference between ``short'' and ``long'' distances
in the definition of an $(\eps,r)$-scatter.

By combining the main algorithmic result of~\cite{ABBCGKMSS23} with \cref{cor:scatter-bound} and a standard ``coarsening'' argument that disposes of distances smaller than $c\delta$, we obtain the following consequence.

\begin{restatable}{theorem}{apxClust}
    \label{thm:apx}
    Fix $h\in \N$. Then given $\eps\in \Rp$ and an instance $(k,G,F,C,\|\cdot\|)$ of~ \textsc{Norm $k$-Clustering}, where $G$ excludes $K_h$ as a $\delta$-fat minor, one can in time $\Oh_{h,k,\eps}((|F|+|C|)^{\Oh(1)})$ compute a solution of cost at most $(1+\eps)\cdot \OPT+\Oh(\delta/\eps^2)\cdot \|\mathbf{1}\|$, where $\OPT$ is the minimum possible cost of a solution and $\mathbf{1}$ denotes the all-one vector in $\R^C$.
\end{restatable}

Note that the approximation guarantee contains an additive term $\Oh(\delta/\eps^2)\cdot \|\mathbf{1}\|$ that depends on the fatness parameter $\delta$ (but not on $h$, the constant hidden in the $\Oh(\cdot)$ notation is universal). 
This is the corresponds to the discussed fact that excluding $K_h$ as a $\delta$-fat minor
implies structure only at the ``scale'' larger than $\delta$.

\bigskip

\noindent\textbf{Drill-flatness in hereditary graph classes.} 
Finally, in the more specialized setting of \EMPH{hereditary}\footnote{A graph class is hereditary if it is closed under taking induced subgraphs} graph classes we uncover strong ties between drill-flatness and (shallow) induced minors.
Recall that an \EMPH{induced model} of a graph $H$ in a graph $G$ is a mapping $\phi$ that maps the vertices of $H$ to pairwise disjoint, connected subsets of the vertices of $G$ (called \EMPH{branch sets}) so that for any two vertices $u,v\in V(H)$, $u$ and $v$ are adjacent in $H$ if and only if their branch sets $\phi(u)$ and $\phi(v)$ are adjacent in $G$. Compared to the standard definition of a (minor) model, this notion additionally requires that branch sets of non-adjacent vertices of $H$ are non-adjacent in $G$; or equivalently, they are at distance at least $2$. Thus, induced minors are conceptually close to $2$-fat minors, and in particular excluding $H$ as an induced minor implies excluding $H$ as a $2$-fat minor. It appears, however, that many statements in coarse graph theory that fail in the fat-minor-free regime, do hold, or at least are not known to fail, in the induced-minor-free regime; see e.g.~\cite{ChudnovskyCK26,ChudnovskyCAL26,ChudnovskyH25,DBLP:journals/corr/abs-2309-08169,HendreyNST24,DBLP:conf/soda/KorhonenL24,DBLP:journals/jctb/Korhonen23} and references~therein.%

We adjust the definition of drill-flatness from length spaces to unweighted graphs by demanding that $A,B,X \subseteq V(G)$. Similarly, an induced minor model is \EMPH{$r$-shallow} if every branch set induces a subgraph of radius at most $r$.
With these definitions at hand, our main result in the induced-minor-free setting reads as follows.

\begin{restatable}
    [Drill-flatness in induced-minor-free graphs]{theorem}{inducedflatness}\label{thm:indminoruqw}
    For all positive integers $h$ and $d$, there exist a function $N\colon \N\to \N$ such that every graph $G$ that does not contain the $1$-subdivision of $K_h$ as a $3d$-shallow induced minor is $(2d,37)$-drill-flat with overhead $N$ and budget~$hd$. Moreover,
    \[N(m)\leq m^{2^{\Oh(hd)}}.\]
\end{restatable}

Note here that $1$-subdivisions of cliques --- graphs obtained from cliques by replacing every edge by a path of length $2$ --- are universal for induced minors in the following sense: Every graph is an induced minor of a large enough $1$-subdivision of a clique. 


The proof of \cref{thm:indminoruqw} relies on the same approach as that of \cref{thm:metricuqwnew}, just adjusted to the setting of induced minors. The main difference is that in the induction step we can simply restrict attention to an induced subgraph, which allows us to avoid increasing the radii of the holes drilled so far. This results in an important improvement: in contrast to \Cref{thm:metricuqwnew}, the drill radius becomes a universal constant ($37$), independent of the target scatter radius $d$. 
We leave as an open question, whether the drill radius in \Cref{thm:indminoruqw} can be improved to $1$.

In fact, we can extend \cref{thm:indminoruqw} to an elegant characterization of hereditary classes that enjoy drill-flatness. Precisely, we prove the following statement.
(The equivalence \ref{c:ismf}$\Leftrightarrow$\ref{c:fsmf} easily lifts from the non-shallow setting~\cite[Lemma 17]{hickingbotham2025induced}, but is included for completeness.)

\begin{restatable}{theorem}{indeq}\label{thm:equivalences}
    For every hereditary graph class $\CC$, the following are equivalent.
    \begin{enumerate}[label=(\arabic*)]
        \item\label{c:ismf} $\CC$ is induced-shallow-minor-free.
        \item\label{c:fsmf} $\CC$ is fat-shallow-minor-free.
        \item\label{c:df}   $\CC$ is drill-flat.
        \item\label{c:patt} For every $r \geq 3$ there is $k \in \N$ such that $\CC$ excludes $\clique{r}{k}$ and $\web{r}{k}$ as induced subgraphs.
    \end{enumerate}
\end{restatable}

\noindent
Here, a class $\Cc$ is
\begin{itemize}
    \item \EMPH{induced-shallow-minor-free} if for every $d\in \N$ there exists a graph $H_d$ such that no graph in $\Cc$ contains $H_d$ as a $d$-shallow induced minor;
    \item \EMPH{fat-shallow-minor-free} if there exists $\delta\in \N$ such that for every $d\in \N$ there exists a graph $H_d$ that is excluded in every graph from $\Cc$ as a $\delta$-fat minor;
    \item \EMPH{drill-flat} if there exists a drill radius $\beta\in \N$ such that for every $\alpha\in \N$, $\Cc$ is $(\alpha,\beta)$-drill-flat with some overhead and budget (depending on $\alpha$).
\end{itemize}
For the last condition fix $r \geq 3$ and $n\in\N$, and we define:
  \begin{itemize}
    \item $\clique{r}{n}$ to be the \EMPH{$r$-subdivided clique} of order $n$, which is obtained by subdividing each edge of the clique  $K_n$ exactly $r$ many times. We refer to the original vertices of the $K_n$ as \EMPH{principal} vertices.
    \item $\web{r}{n}$ to be the \EMPH{$r$-web} of order $n$, which is obtained from $\clique{r}{n}$ by turning the neighborhood of every principal vertex into a clique.
  \end{itemize}
These forbidden induced subgraphs previously appeared also in the study monadically stable graph classes~\cite{dreier2024first}. See \Cref{fig:intro-patterns} for an illustration.
\begin{figure}[h!]
  \centering\small
  $\vcenter{\hbox{\subclique[1]{3}{5}{}}}$
  \qquad\qquad
  $\vcenter{\hbox{\spiderweb[1]{3}{5}{}}}$
  \caption{$\clique{3}{5}$ and $\web{3}{5}$. Replicated from \cite{dreier2024first}.}\label{fig:intro-patterns}
\end{figure}

We note that the restriction to hereditary graph classes in \Cref{thm:equivalences} is crucial:
the (non-hereditary) class of graphs with a universal vertex is both fat-shallow-minor-free and drill-flat, but not induced-shallow-minor free and contains all graphs as induced subgraphs.

We prove \cref{thm:equivalences} by a cycle of implications, where the key implication \ref{c:ismf}\,$\Rightarrow$\,\ref{c:df} is provided by \cref{thm:indminoruqw}. Implication \ref{c:df}\,$\Rightarrow$\,\ref{c:fsmf} follows by observing that a fat minor of a large clique is an obstruction to flatness. Implication \ref{c:fsmf}\,$\Rightarrow$\,\ref{c:patt} is a basic observation that assuming $\Cc$ is hereditary, sufficiently large patterns contain fat minors of all $n$-vertex graphs. Finally, implication \ref{c:patt}\,$\Rightarrow$\,\ref{c:ismf} relies on finding large patterns within induced models of large subdivided bicliques using Ramsey's~Theorem.

\bigskip

\noindent\textbf{Examples of induced-shallow-minor-free (ISMF) classes.} 
ISMF classes seem to fill a natural gap in the theory of hereditary graph classes, as they unify two otherwise separate lines of research.
\begin{enumerate}
    \item By definition, every induced-minor-free class is ISMF.

    (This includes geometric graph classes such as string graphs.)
    \item By definition, every shallow-minor-free class is ISMF.
    
    (Shallow-minor-free classes are precisely the nowhere dense classes.
    They include all classes of bounded degree and all classes of bounded expansion.)

    More generally, by \Cref{thm:equivalences}.\ref{c:patt}, every monadically dependent graph class is ISMF.
    
    (Every nowhere dense class is monadically dependent~\cite{stable_graphs,adler2014interpreting}, as is every class of bounded twin-, merge-, or flip-width and every monadically stable class~\cite{bonnet2021twin, dreier2025merge, torunczyk2023flip, dreier2024first}.)
\end{enumerate}

We show that there also exist natural graph classes that are ISMF yet neither induced-minor-free (for $d\geq 3$) nor monadically dependent (for $d\geq 2$):

\begin{restatable}{theorem}{ballsismf}\label{thm:balls-ismf}
    For every $d\in \N$, the class of intersection graphs of balls in $\R^d$ is ISMF.
\end{restatable}

The proof of \cref{thm:balls-ismf} easily extends to intersection graphs of families of objects in $\R^d$ with universally bounded aspect ratio; we omit the details.

\paragraph{AI Disclosure.}
ChatGPT 5.5 Pro and Claude Code Opus 4.8 were used to turn proof sketches for the Ramsey arguments in \Cref{lem:fan-ramsey} and \Cref{lem:extract-bipatterns} into first drafts, which were further polished by the authors.
The authors verified the correctness and originality of all AI generated text.
All mathematical results are due to the authors.
\section{Preliminaries}\label{sec:prelims}

We denote $\N=\{0,1,2,\ldots\}$. For a real $\alpha$, by $\R_{>\alpha}$ and $\R_{\geq \alpha}$ we denote the sets of reals that are larger, respectively not smaller than $\alpha$.

\subsection{Graphs}

We use standard graph notation and terminology. For a vertex $v$ in a graph $G$, by $N_G[v]$ we denote the \emph{close neighborhood} of $v$ in $G$: the set consisting of $v$ and all its neighbors.

Recall that a graph $H$ is a \emph{minor} of a graph $G$ if there exists a \emph{model} of $H$ in $G$, that is, a mapping $\phi$ that assigns every vertex $v$ of $H$ a subset $\phi(v)$ of $V(G)$ that induces a connected subgraph so that the subsets $\{\phi(v)\colon v\in V(H)\}$ (called often \emph{branch sets}) are pairwise disjoint and for every edge $uv$ of $H$, there is an edge in $G$ connecting a vertex of $\phi(u)$ and a vertex of $\phi(v)$. The model $\phi$ is \emph{induced} if this latter condition is in fact an equivalence: for every pair of distinct vertices $u,v\in V(H)$, $u$ and $v$ are adjacent in $H$ if and only if there is an edge connecting $\phi(u)$ and $\phi(v)$ in $G$. Further, we say that $\phi$ is \emph{$r$-shallow}, for $r\in \N$, if $G[\phi(v)]$ has radius at most $r$ for every $v\in V(H)$. Finally, $G$ contains $H$ as an \emph{($r$-shallow, induced) minor} if $G$ contains an ($r$-shallow, induced, respectively) model of $H$.

We will need the following extremal statement, which is a corollary of a result due to Alon, Krivelevich, and Sudakov~\cite[Theorem~2.2]{DBLP:journals/cpc/AlonKS03}.
This exact formulation appears as~\cite[Corollary~2.3]{DBLP:conf/lics/PilipczukST18a}.

\begin{theorem}\label{thm:alon-density}
    For all positive integers $h, n$, every $n$-vertex graph with more than $\binom{h+1}{2} n^{3/2}$ edges
    contains $K_h$ as a $1$-shallow minor. 
\end{theorem}

\subsection{Metric graphs and length spaces}

By a \emph{metric graph} we mean the simplicial complex endowed with a metric
formed from an edge-weighted graph $G$ by taking the vertex set of $G$ and gluing, for every edge $uv$ of $G$, an interval of length equal to the weight of $uv$ to the vertices $u$ and $v$.
Note that thus, the point set of a metric graph $G$ consists not only of the original vertices of $G$, but also of all the points of all the intervals modeling the edges.
We also note that this construction can be applied when the graph $G$ is infinite.

Following previous works in coarse graph theory, in this work
we will rely on the notion of a \emph{length space}, a generalization of metric graphs,
which is a metric space equipped with a robust notion of (shortest) paths. Since we are predominantly interested in metric spaces derived from graphs, we typically denote a metric space by $G$, where $G$ is understood as the set of points in the space and $\dist_G$ denotes the metric; the subscript can be omitted if clear from the context. A \emph{path} in $G$ is a homeomorphic image of an interval. We may also speak about an \emph{$a$-$b$-path} to indicate that $a$ and $b$ are the endpoints. The \emph{length} of a path $P$ is defined as the supremum of the values $\sum_{i=1}^n \dist_G(x_{i-1},x_i)$, where $x_0,x_1,\ldots,x_n$ are points lying in this order on $P$ with $x_0$ and $x_n$ being the endpoints. Then $G$ is a \emph{length space} if 
\[\dist_G(a,b)=\inf \{\,\mathrm{length}(P)\colon P\textrm{ is an }a\textrm{-}b\textrm{-path}\,\},\qquad \textrm{for all }a,b\in G.\]
Note that this value may be infinite if there is no $a$-$b$-path.

It is immediate that a metric graph is a length space and, if the metric graph comes
from a finite edge-weighted graph, the infimum in the definition of the distance
is actually attained by some path. However, as our proofs deal with general length spaces,
there will be a number (very peculiar for readers working in graph algorithms)
moments where from $\dist_G(a,b) \leq d$ for some $a,b \in G$
and $d \in \Rp$, we will only be able to get an $a$-$b$-path of length at most $d+\eps$ for any $\eps> 0$,
but not necessarily of length at most $d$.


Let $G$ be a length space.
For two subsets of points $A,B\subseteq G$, we denote
\[\dist_G(A,B)\coloneqq \inf_{a\in A, b\in B} \dist_G(a,b).\]
For a subset $Z\subseteq G$ and $\beta\in \Rp$, we denote
\[\Ball_G(Z,\beta)\coloneqq \{x\in G\mid \dist_G(x,Z)\leq \beta\}.\]
In a number of places, we will ``drill holes'' in a length space. 
That is, for a set $B \subseteq G$, we may define a metric space $G-B$ by endowing this set with the metric
\[\dist_{G-B}(a,b)\coloneqq \inf \{\,\textrm{length}(P)\colon P\textrm{ is an }a\textrm{-}b\textrm{-path disjoint with }B\,\}.\]
It is straightforward to check that $G-B$ is again a length space. For a set $A\subseteq G$, the length space \emph{induced} by $A$ is the length space $G[A]\coloneqq G-(G-A)$.

The weak radius of a set $A\subseteq G$ is the infimum of $\rho\in \Rp$ such that $A\subseteq \Ball_G(x,\rho)$ for some $x\in G$, and the weak diameter of $A$ is the supremum of $\dist_G(a,b)$ for $a,b\in A$. The \emph{strong radius} and the \emph{strong diameter} of $A$ are defined as the weak radius and the weak diameter of $A$ in $G[A]$. Note that the weak radius/diameter of a set $A\subseteq G$ is always a lower bound on its strong radius/diameter.


Let $G$ be a length space and $A\subseteq G$.
Recall that for $\alpha\in \Rp$, we say that $A$ is \emph{$\alpha$-scattered} if $\dist_G(a,b)>\alpha$ for all $a,b\in A$. More generally, for $Z\subseteq G$ we say that $A$ is \emph{$(\alpha,Z)$-scattered} if $A\setminus Z$ is $\alpha$-scattered in $G-Z$.



We have the following immediate observation.
\begin{lemma}\label{lem:dist-GZ-to-G}
    Let $G$ be a length space, $\alpha\in \Rp$, and $A,B,Z \subseteq G$ be
    such that $\dist_G(A,Z) \geq \alpha/2$, $\dist_G(B,Z) \geq \alpha/2$, and  $\dist_{G-Z}(A,B) \geq \alpha$. Then also $\dist_G(A,B) \geq \alpha$. 
\end{lemma}
\begin{proof}
Let $P$ be an $a$-$b$-path in $G$, for any $a\in A$ and $b\in B$.
If $P$ is disjoint with $Z$,
then the length of $P$ is at least $\alpha$ due to the assumption $\dist_{G-Z}(A,B) \geq \alpha$.
Otherwise, consider any $x \in P \cap Z$. Then the subpath of $P$ from $a$ to $x$
is of length at least $\alpha/2$, due to $\dist_G(A,Z) \geq \alpha/2$, and similarly
the subpath of $P$ from $x$ to $b$ is of length at least $\alpha/2$. Hence $P$ is of length at least $\alpha$. 
\end{proof}

\subsection{Fat minors}
Let $H$ be a graph, $G$ a length space or a graph, and $\delta\in \Rp$.
A \emph{$\delta$-fat model} of $H$ in $G$ is a map $\phi$ that assigns to every vertex and every edge of $H$
an arc-connected subset of $G$ such that the following holds:
\begin{itemize}
\item For every $e \in E(H)$ and an endpoint $v$ of $e$, we have $\phi(e) \cap \phi(v) \neq \emptyset$.
\item For every two objects $o_1, o_2 \in V(H) \cup E(H)$, if $\{o_1, o_2\}$ is not a pair consisting of an edge and one of its endpoints,
then $\dist_G(\phi(o_1),\phi(o_2)) \geq \delta$.
\end{itemize}
Again, the sets $\{\phi(o)\colon o\in V(H)\cup E(H)\}$ will be called \emph{branch sets}.
We then say that $H$ is a \emph{$\delta$-fat minor} of $G$ if there is a $\delta$-fat model of $H$ in $G$.

For $\rho_V,\rho_E\in \Rnn$, we say that a $\delta$-fat model $\phi$ of $H$ is \emph{$(\rho_V,\rho_E)$-shallow} if
for every $v \in V(H)$, the strong radius of $\phi(v)$ is at most $\rho_V$, and
for every $e \in E(H)$, the strong radius of $\phi(e)$ is at most~$\rho_E$.
By a \emph{$\rho$-shallow model} we mean a $(\rho,\rho)$-shallow model. 

Containing shallow fat minors is defined naturally through the existence of suitable models.

We have the following observation about the transitivity of the (fat) minor relation.
\begin{lemma}\label{lem:minor-of-fat-minor}
    Let $H_1$ and $H_2$ be graphs, $G$ be a length space, and $\delta\in \Rp$.
    Suppose $H_1$ is a minor of $H_2$ and $H_2$ is a $\delta$-fat minor of $G$. Then $H_1$ is a $\delta$-fat minor of $G$.
    
    Furthermore, if $H_1$ is an $r$-shallow minor of $H_2$ and $H_2$ is a $(\rho_V, \rho_E)$-shallow $\delta$-fat minor of $G$, for some $r\in \N$ and $\rho_V,\rho_E\in \Rnn$, then 
    $H_1$ is a $((2r+1)\rho_V+2r\rho_E, \rho_E)$-shallow $\delta$-fat minor of $G$.
\end{lemma}
\begin{proof}
    It is easy to check that the composition of an $r$-shallow model of $H_1$ in $H_2$ and a $(\rho_V,\rho_E)$-shallow $\delta$-fat model of $H_2$ in $G$ yields a $((2r+1)\rho_V+2r\rho_E, \rho_E)$-shallow $\delta$-fat model of $H_1$ in~$G$.
\end{proof}

\part{Drill-flatness in length spaces}
\section{Constructing a fat minor}

In this section we prove two auxiliary lemmas, which will later allow us to construct fat minors whenever a certain configuration of distance between points in a length spaced is uncovered.

\begin{figure}[htbp]
    \centering
    \begin{tikzpicture}[
    >=Latex,
    every node/.style={font=\small},
    hvertex/.style={circle,fill=black,inner sep=1.6pt},
    vimage/.style={circle,fill=blue!70!black,draw=white,line width=.4pt,inner sep=2.2pt},
    eimage/.style={circle,fill=orange!90!black,draw=white,line width=.4pt,inner sep=2.4pt},
    branch/.style={draw=blue!70!black,densely dashed,line width=.95pt},
    edgeball/.style={draw=orange!90!black,densely dashed,line width=.95pt},
    goodpath/.style={blue!65!black,line width=1.15pt,line cap=round},
    avoid/.style={draw=red!70!black,dashed,line width=.9pt},
    zset/.style={fill=red!70!black,draw=red!80!black,line width=.5pt},
    smalllabel/.style={font=\scriptsize}
]

\def\BranchR{1.15}
\def\EdgeR{0.35}

\pgfmathsetmacro{\TouchDist}{\BranchR+\EdgeR}
\pgfmathsetmacro{\Side}{2*\TouchDist}
\pgfmathsetmacro{\TriHeight}{sqrt(3)*\TouchDist}

\coordinate (Hu) at (0.35,1.00);
\coordinate (Hv) at (2.05,1.00);
\coordinate (Hw) at (1.20,2.35);

\node[font=\large] at (1.2,0.10) {$H$};

\draw[black,line width=.75pt] (Hu)--(Hv)--(Hw)--cycle;

\node[hvertex,label=below:$u$] at (Hu) {};
\node[hvertex,label=below:$v$] at (Hv) {};
\node[hvertex,label=above:$w$] at (Hw) {};

\node[smalllabel] at (1.20,.73) {$e=uv$};

\draw[->,thick] (2.45,1.65) -- (3.1,1.65)
    node[midway,above] {$f$};

\path[fill=gray!5,draw=black!35,line width=.9pt]
    plot[smooth cycle,tension=.85] coordinates {
        (3.45,.00) (3.75,5.70) (6.95,6.05)
        (10.65,5.25) (11.05,1.20) (8.80,-.10) (5.05,-.25)
    };

\node[font=\large] at (6.42,0.1) {$G$};

\coordinate (Fu) at (4.85,1.35);
\coordinate (Fv) at ($(Fu)+(\Side,0)$);
\coordinate (Fw) at ($(Fu)+(\TouchDist,\TriHeight)$);

\coordinate (Fuv) at ($(Fu)!0.5!(Fv)$);
\coordinate (Fuw) at ($(Fu)!0.5!(Fw)$);
\coordinate (Fvw) at ($(Fv)!0.5!(Fw)$);

\draw[goodpath] (Fuv) .. controls ($(Fuv)+(-.45,-.28)$) and ($(Fu)+(.55,-.20)$) .. (Fu);
\draw[goodpath] (Fuv) .. controls ($(Fuv)+(.45,-.28)$) and ($(Fv)+(-.55,-.20)$) .. (Fv);

\draw[goodpath] (Fuw) .. controls ($(Fuw)+(-.55,-.12)$) and ($(Fu)+(-.20,.65)$) .. (Fu);
\draw[goodpath] (Fuw) .. controls ($(Fuw)+(.05,.62)$) and ($(Fw)+(-.65,-.10)$) .. (Fw);

\draw[goodpath] (Fvw) .. controls ($(Fvw)+(.55,-.12)$) and ($(Fv)+(.20,.65)$) .. (Fv);
\draw[goodpath] (Fvw) .. controls ($(Fvw)+(-.05,.62)$) and ($(Fw)+(.65,-.10)$) .. (Fw);

\foreach \p in {Fu,Fv,Fw}{
    \draw[branch] (\p) circle (\BranchR);
}

\foreach \p in {Fuv,Fuw,Fvw}{
    \draw[edgeball] (\p) circle (\EdgeR);
}

\draw[avoid] (9.25,1.45) ellipse (.72 and .47);
\filldraw[zset] (9.25,1.45) ellipse (.25 and .14);
\node[white,font=\scriptsize] at (9.25,1.45) {$Z$};
\node[red!70!black,smalllabel,anchor=west] at (9.95,1.45)
    {$N_{\delta/2}(Z)$};

\node[vimage] at (Fu) {};
\node[vimage] at (Fv) {};
\node[vimage] at (Fw) {};

\node[eimage] at (Fuv) {};
\node[eimage] at (Fuw) {};
\node[eimage] at (Fvw) {};

\node[smalllabel] at ($(Fu)+(-0.32,-.32)$) {$f(u)$};
\node[smalllabel] at ($(Fv)+(0.02,-.32)$) {$f(v)$};
\node[smalllabel] at ($(Fw)+(0,0.38)$) {$f(w)$};

\node[smalllabel] at ($(Fuv)+(0,-.62)$) {$f(uv)$};
\node[smalllabel] at ($(Fuw)+(-.72,.18)$) {$f(uw)$};
\node[smalllabel] at ($(Fvw)+(.72,.18)$) {$f(vw)$};


\end{tikzpicture}
    \caption{a $(\frac{1}{2}-\xi-\delta,2\xi+\delta)$-shallow $\delta$-fat minor model of $G$ constructed from $f: V(H)\cup E(H)\rightarrow G-Z$ in \Cref{lem:points-to-fat-minor}. The small balls (centered at $f(e)$) have radius $\frac12-\xi-\delta$, while the large balls (centered at $f(v)$) have radii at most $\frac12-\xi-\delta$.}
    \label{fig:shallow-fat-1}
\end{figure}
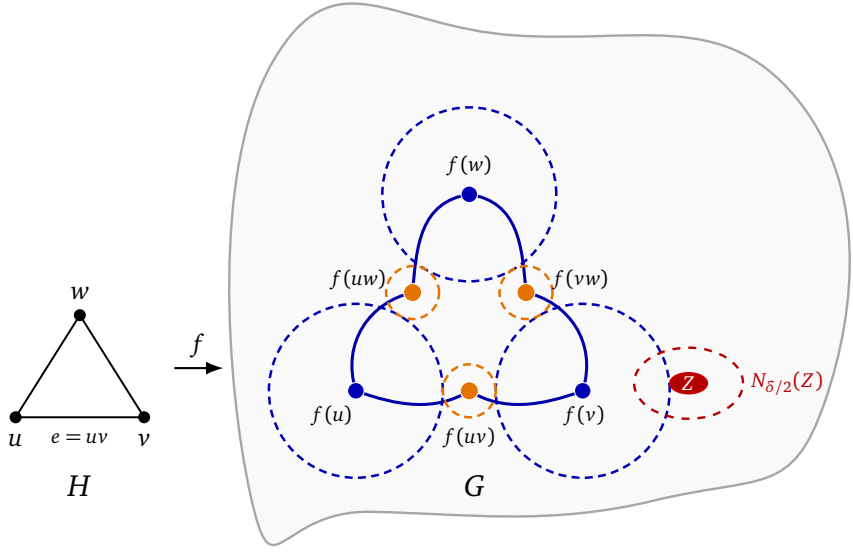

\begin{lemma}\label{lem:points-to-fat-minor}
    Let $H$ be a graph with no isolated vertices, $G$ a length space, $Z \subseteq G$, 
    and $0 < \delta \leq \xi \leq \frac{1}{8}$ be reals.
    Assume there exists a function $f \colon V(H) \cup E(H) \to G - Z$ with the following properties:
    \begin{itemize}
    \item For every two distinct vertices $u,v \in V(H)$, we have $\dist_{G-Z}(f(u),f(v)) > 1$.
    \item For every two distinct edges $e_1,e_2 \in E(H)$, we have $\dist_{G-Z}(f(e_1), f(e_2)) > 4\xi + 3\delta$.
    \item For every $e \in E(H)$ and an endpoint $v$ of $e$, 
     there exists  in $G$ an $f(e)$-$f(v)$-path $P_{e,v}$
     of length at most $\frac{1}{2}+\xi$ that does not contain any point within distance
     less than $\delta/2$ from $Z$.
    \end{itemize}
    Then $G$ contains $H$ as a $(\frac{1}{2}-\xi-\delta,2\xi+\delta)$-shallow $\delta$-fat minor.
\end{lemma}
\begin{proof} See \Cref{fig:shallow-fat-1} for an illustration. Note that the third asserted condition implies that for every $e \in E(H)$ and an endpoint $v$ of $e$,
    \begin{equation*}
        \dist_{G-Z}(f(e),f(v)) \leq \tfrac{1}{2}+\xi.
    \end{equation*}
    
    By the triangle inequality and the first condition of the lemma, applied w.r.t.\ $u$, the other endpoint of $e$, we have:
    \begin{equation}\label{eq:ptfm:ev-dist}
        \dist_{G-Z}(f(e),f(v))\ge\dist_{G-Z}(f(u),f(v))-\dist_{G-Z}(f(e),f(u))>1-(1/2+\xi)=\tfrac{1}{2}-\xi.
    \end{equation}
        
    Pick any $e \in E(H)$ and an endpoint $v$ of $e$. 
    Let $p_{e,v}$ be the point on $P_{e,v}$ within distance (on $P_{e,v}$) exactly $2\xi+\delta$ from $f(e)$; note that such a point
    exists and is distinct from $f(v)$ due to~\eqref{eq:ptfm:ev-dist} and the assumption $\delta \leq \xi \leq \frac{1}{8}$. 
    Point $p_{e,v}$ splits $P_{e,v}$ into two paths; let us denote them by $P_{e,v}^e$ and $P_{e,v}^v$ 
    so that $o$ is an endpoint of $P_{e,v}^o$ for $o \in \{e,v\}$.
    
    For an edge $e = uv \in E(H)$, we define \[\phi(e) \coloneqq P_{e,v}^e \cup P_{e,u}^e,\]
    and for a vertex $v \in V(H)$, we define \[\phi(v) \coloneqq \bigcup \{P_{e,v}^v\colon e\in E(H) \textrm{ is incident to }v\}.\]
    We claim that $\phi$ is the desired model of $H$ in $G$.
    Clearly, every $\phi(e)$ and $\phi(v)$ is arc-connected, and whenever $v$ is an endpoint of $e$, the sets $\phi(v)$ and $\phi(e)$ intersect at~$p_{e,v}$.

    Since the length of every path $P_{e,v}^e$ is exactly $2\xi+\delta$, we have that $\phi(e)$ has strong radius at most $2\xi+\delta$ for every $e\in E(H)$, with $f(e)$ being the witnessing center.
    Furthermore, as for distinct $e_1,e_2 \in E(H)$ we have $\dist_{G-Z}(f(e_1),f(e_2)) > 4\xi+3\delta$. Thus, by the triangle inequality, we infer
    \begin{equation}\label{eq:ptfm:ee} \dist_{G-Z}(\phi(e_1), \phi(e_2)) > (4\xi + 3\delta) - 2(2\xi + \delta) =  \delta. \end{equation}

    Observe that the length of every path $P_{e,v}^v$ is at most $\frac{1}{2}+\xi-(2\xi+\delta) =  \frac{1}{2}-\xi-\delta$.
    Hence, for every $v\in V(H)$ we have that $\phi(v)$ has strong radius at most $\frac{1}{2}-\xi-\delta$, with $f(v)$ being the witnessing center.
    In particular, as for distinct $u,v \in V(H)$ we have $\dist_{G-Z}(f(u),f(v)) > 1$,  we infer by the triangle inequality that
    \begin{equation}\label{eq:ptfm:vv}
     \dist_{G-Z}(\phi(u),\phi(v)) > 1-2(1/2-\xi - \delta) =  2\xi+2\delta > \delta. 
    \end{equation}

    Consider now $v \in V(H)$, $e \in E(H)$, and an endpoint $u$ of  $e$ such that $u \neq v$. 
    As the length of $P_{e,u}$ is at most $\frac{1}{2}+\xi$ while $\dist_{G-Z}(f(u),f(v)) > 1$, we have
    \[ \dist_{G-Z}(f(v), P_{e,u}) > \tfrac{1}{2}-\xi. \]
    Consequently, since $\phi(v)\subseteq \Ball_{G-Z}(f(v),\frac{1}{2}-\xi-\delta)$, we have
    \[ \dist_{G-Z}(\phi(v),P_{e,u}) > \delta. \]
    This implies that if $v$ is not incident with $e$, then
    \begin{equation}\label{eq:ptfm:ve}
     \dist_{G-Z}(\phi(v),\phi(e)) > \delta. 
    \end{equation}
    Since the sets $\phi(\ast)$ are constructed from paths $P_{\ast,\ast}$,
    no point of any set $\phi(o)$ is within distance less than $\delta/2$ from $Z$.
    We infer that \cref{lem:dist-GZ-to-G} applies
    and the inequalities of~\eqref{eq:ptfm:ee}, \eqref{eq:ptfm:vv}, and~\eqref{eq:ptfm:ve}
    hold also for distances in $G$, not only in $G-Z$. 
    This completes the proof that $\phi$ is a $(\frac{1}{2}-\xi-\delta,2\xi+\delta)$-shallow $\delta$-fat model of $H$ in $G$.
\end{proof}

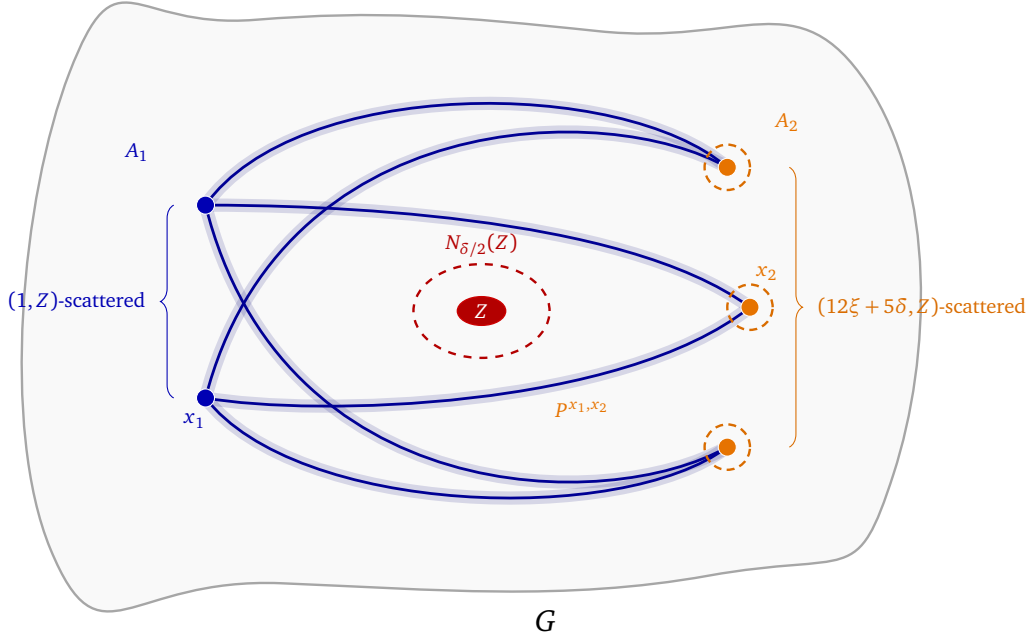
\begin{figure}[htbp]
    \centering
    \begin{tikzpicture}[
    >=Latex,
    every node/.style={font=\small},
    aone/.style={
        circle,
        fill=blue!70!black,
        draw=white,
        line width=.4pt,
        inner sep=2.4pt
    },
    atwo/.style={
        circle,
        fill=orange!90!black,
        draw=white,
        line width=.4pt,
        inner sep=2.4pt
    },
    atwoball/.style={
        draw=orange!85!black,
        densely dashed,
        line width=.9pt
    },
    pathcore/.style={
        blue!58!black,
        line width=1.05pt,
        line cap=round
    },
    pathtube/.style={
        blue!45!black,
        line width=5.0pt,
        line cap=round,
        opacity=.14
    },
    forbidden/.style={
        draw=red!70!black,
        dashed,
        line width=.9pt
    },
    zset/.style={
        fill=red!70!black,
        draw=red!80!black,
        line width=.5pt
    },
    smalllabel/.style={font=\scriptsize},
    braceblue/.style={
        decorate,
        decoration={brace,amplitude=5pt},
        blue!70!black
    },
    braceorange/.style={
        decorate,
        decoration={brace,amplitude=5pt,mirror},
        orange!85!black
    }
]

\path[fill=gray!5,draw=black!35,line width=.9pt]
    plot[smooth cycle,tension=.85] coordinates {
        (0.00,0.00) (0.25,6.30) (2.70,7.00)
        (7.25,6.85) (10.80,6.15) (11.15,1.00)
        (8.70,-.35) (3.05,-.45)
    };

\node[font=\large] at (6.55,-0.95) {$G$};

\coordinate (a1) at (2.05,2.00);
\coordinate (a2) at (2.05,4.55);

\coordinate (b1) at (8.95,1.35);
\coordinate (b2) at (9.25,3.20);
\coordinate (b3) at (8.95,5.05);

\draw[forbidden] (5.70,3.15) ellipse (.90 and .62);
\filldraw[zset] (5.70,3.15) ellipse (.32 and .19);
\node[white,font=\scriptsize] at (5.70,3.15) {$Z$};

\node[red!70!black,smalllabel] at (5.70,4.02)
    {$N_{\delta/2}(Z)$};


\draw[pathtube] (a1) .. controls (3.25,.45) and (7.55,.30) .. (b1);
\draw[pathtube] (a1) .. controls (3.50,1.75) and (7.55,1.85) .. (b2);
\draw[pathtube] (a1) .. controls (3.05,5.85) and (7.35,5.95) .. (b3);

\draw[pathtube] (a2) .. controls (3.05,.55) and (7.35,.45) .. (b1);
\draw[pathtube] (a2) .. controls (3.50,4.55) and (7.55,4.45) .. (b2);
\draw[pathtube] (a2) .. controls (3.25,6.25) and (7.55,6.25) .. (b3);

\draw[pathcore] (a1) .. controls (3.25,.45) and (7.55,.30) .. (b1);
\draw[pathcore] (a1) .. controls (3.50,1.75) and (7.55,1.85) .. (b2);
\draw[pathcore] (a1) .. controls (3.05,5.85) and (7.35,5.95) .. (b3);

\draw[pathcore] (a2) .. controls (3.05,.55) and (7.35,.45) .. (b1);
\draw[pathcore] (a2) .. controls (3.50,4.55) and (7.55,4.45) .. (b2);
\draw[pathcore] (a2) .. controls (3.25,6.25) and (7.55,6.25) .. (b3);

\foreach \p in {b1,b2,b3}{
    \draw[atwoball] (\p) circle (.30);
}

\node[aone] at (a1) {};
\node[aone] at (a2) {};

\node[atwo] at (b1) {};
\node[atwo] at (b2) {};
\node[atwo] at (b3) {};

\node[smalllabel,blue!70!black] at ($(a1)+(-.15,-.30)$) {$x_1$};

\node[smalllabel,orange!90!black] at ($(b2)+(.22,0.45)$) {$x_2$};
\node[smalllabel,orange!90!black] at ($(b2)+(-2.22,-1.35)$) {$P^{x_1,x_2}$};

\node[smalllabel,blue!70!black] at (1.15,5.25) {$A_1$};
\node[smalllabel,orange!90!black] at (9.75,5.65) {$A_2$};

\draw[braceblue] ($(a1)+(-.42,0)$) -- ($(a2)+(-.42,0)$)
    node[midway,left=7pt,smalllabel]
    {$(1,Z)$-scattered};

\draw[braceorange] ($(b1)+(.82,0)$) -- ($(b3)+(.82,0)$)
    node[midway,right=7pt,smalllabel,align=left]
    {$(12\xi+5\delta,Z)$-scattered};


\end{tikzpicture}
    \caption{Two scattered sets $A_1$ and $A_2$ in \Cref{lem:biclique-fat-minor}. Each highlighted blue path has length at most $1/2+\xi$.}
    \label{fig:shallow-fat-2}
\end{figure}

\begin{lemma}\label{lem:biclique-fat-minor}
    Let $G$ be a length space, $Z \subseteq G$,
    and $0 < \delta \leq \xi \leq \frac{1}{10}$ be reals.
    Let $A_1,A_2 \subseteq G-Z$ be finite sets with the following properties:
    \begin{itemize}
    \item $A_1$ is $(1,Z)$-scattered and $|A_1|\geq 2$.
    \item $A_2$ is $(12\xi+5\delta,Z)$-scattered.
    \item For every $x_1 \in A_1$ and $x_2 \in A_2$, 
    there exists an $x_1$-$x_2$-path $P^{x_1,x_2}$ in $G$ of length at most $\tfrac{1}{2} + \xi$
    that does not contain any point within a distance less than $\delta/2$ from $Z$.
    \end{itemize}
    Then $G$ contains $K_{|A_1|,|A_2|}$ as a $(\tfrac{1}{2}+\xi,\xi+\delta/2)$-shallow $\delta$-fat minor.
\end{lemma}
\begin{proof} See \Cref{fig:shallow-fat-2} for an illustration.  Let $x_1 \in A_1$ and $x_2 \in A_2$ be arbitrary.
    Note that since every path $P^{\cdot,\cdot}$ is disjoint with $Z$ and
    of length at most $1/2+\xi$, while $A_1$ is $(1,Z)$-scattered and contains at least one point other that $x_1$, we have
    $\dist_{G-Z}(x_1,x_2) > 1 - (1/2+\xi) =  1/2-\xi$.
    In particular, the path $P^{x_1,x_2}$ is of length more than $1/2-\xi$. 

    Recall that $P^{x_1,x_2}$ is not required to be a shortest path.
    Pick on $P^{x_1,x_2}$ two points $p^{x_1,x_2}_1$ and $p^{x_1,x_2}_2$ as follows:
    \begin{itemize}
    \item $p^{x_1,x_2}_1$ is the first point encountered on $P^{x_1,x_2}$ when traversing from $x_1$ to $x_2$ that satisfies
    \[ \dist_{G-Z}(x_1,p^{x_1,x_2}_1) = \tfrac{1}{2} - 3\xi-2\delta. \]
    \item $p^{x_1,x_2}_2$ is the last point encountered on $P^{x_1,x_2}$ when traversing from $x_1$ to $x_2$ that
    satisfies
    \[ \dist_{G-Z}(x_1,p^{x_1,x_2}_2) = \tfrac{1}{2}-3\xi-\delta. \]
    \end{itemize}
    Such points always exist, because $\dist_{G-Z}(x_1,x_2)>\tfrac{1}{2}-\xi$ and $\delta \leq \xi \leq \frac{1}{10}$ ensures that $\tfrac{1}{2}-3\xi-2\delta\geq 0$.
    
    By the construction, points $x_1$, $p^{x_1,x_2}_1$, $p^{x_1,x_2}_2$, $x_2$ lie on $P^{x_1,x_2}$ in this order.
    Thus, points $p^{x_1,x_2}_1$ and $p^{x_1,x_2}_2$ split $P^{x_1,x_2}$ into three paths $P^{x_1,x_2}_1$, $P^{x_1,x_2}_E$, and $P^{x_1,x_2}_2$, ordered in this way when traversing $P^{x_1,x_2}$ from $x_1$ to $x_2$,
    see Figure~\ref{fig:split-pxx}.

    \begin{figure}[h]
    \centering{\includegraphics[scale=1]{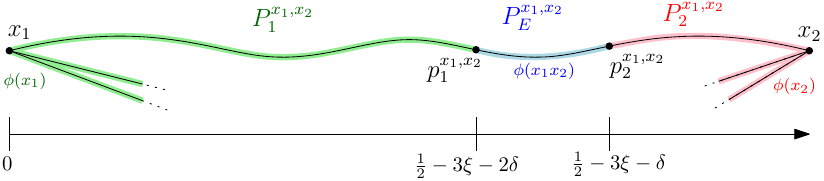}\caption{Splitting $P^{x_1,x_2}$ into subpaths $P_1^{x_1,x_2}$, $P_E^{x_1,x_2}$, and $P_2^{x_1,x_2}$.}\label{fig:split-pxx}}
    \end{figure}
    
    By the choice of $p^{x_1,x_2}_1$, all points of $P_1^{x_1,x_2}$ are within
    distance $\tfrac{1}{2}-3\xi-2\delta$ from $x_1$ in $G-Z$, that is, 
    \begin{equation}\label{eq:wiewiorka}P_1^{x_1,x_2}\subseteq \Ball_{G-Z}(x_1,\tfrac{1}{2}-3\xi-2\delta).
    \end{equation}
    By the choice of $p^{x_1,x_2}_2$, no vertex of $P_1^{x_1,x_2}$ lies within
    distance less than $\frac{1}{2}-3\xi-\delta$ of $x_1$. 
    Hence, for every $x_1 \in A_1$ and every (not necessarily distinct) $x_2,x_2' \in A_2$, we have
    \begin{equation}\label{eq:delta-gap}
        \dist_{G-Z}(P_1^{x_1,x_2}, P_2^{x_1,x_2'}) \geq \delta.
    \end{equation}
    Again by the choice of $p^{x_1,x_2}_2$ together with $\dist_{G-Z}(x_1,x_2)>\tfrac{1}{2}-\xi$ implies that the length of $P^{x_1,x_2}_2$ is larger than $(\tfrac{1}{2}-\xi)-(\tfrac{1}{2}-3\xi-\delta)=2\xi+\delta$. On the other hand, the length of $P_{x_1,x_2}$ is at most $\frac{1}{2}+\xi$, so the length of $P^{x_1,x_2}_2$ is at most $(\tfrac{1}{2}+\xi)-(\tfrac{1}{2}-3\xi-\delta)=4\xi+\delta$, implying that
    \begin{equation}\label{eq:siebenschlaefer}
    P_2^{x_1,x_2}\subseteq \Ball_{G-Z}(x_2,4\xi+\delta).
    \end{equation}
    Finally,
    the length of $P^{x_1,x_2}_E$ is at most 
    \[ 
    \len(P^{x_1,x_2}_E)
    \leq
    \underbrace{(\tfrac{1}{2}+\xi)}
    _{\geq \len(P^{x_1,x_2})}
    - 
    \underbrace{(\tfrac{1}{2} - 3\xi-2\delta)}
    _{\leq \len (P^{x_1,x_2}_1)} 
    -
    \underbrace{ 
    (2\xi+\delta)}
    _{\leq \len (P^{x_1,x_2}_2)}
    =
    2\xi+\delta. \]

    Let $H$ be the biclique with sides $A_1$ and $A_2$. 
    For every $x_1 \in A_1$ and $x_2 \in A_2$, define \[\phi(x_1x_2) \coloneqq P^{x_1,x_2}_E.\]
    Note that this is a path of length at most $2\xi+\delta$,
    so of strong radius at most $\xi+\delta/2$.
    For any $x_1 \in A_1$ and $x_2 \in A_2$, define
    \[\phi(x_1) = \bigcup_{y \in A_{2}} P^{x_1,y}_1,
    \qquad
    \phi(x_2) = \bigcup_{y \in A_{1}} P^{y,x_2}_2.
    \]
    Note that since $P^{x_1,y}_1$, as a subpath of $P^{x_1,y}$, has length at most $\frac{1}{2}+\xi$, hence $\phi(x_1)$ has strong radius at most $\frac{1}{2}+\xi$, with $x_1$ being the witnessing center.
    Analogously, this also holds for $\phi(x_2)$. 
    
    We claim that $\phi$ is the desired model of $H$ in $G$.
    Clearly, $\phi(x_1x_2)$ connects a point of $\phi(x_1)$ with a point of $\phi(x_2)$.
    The promised strong radius bounds for shallowness have already been established. It remains
    to verify fatness.

    By \eqref{eq:wiewiorka} and \eqref{eq:siebenschlaefer}, we have for every $x_1 \in A_1$ and $x_2 \in A_2$:
    \begin{align} 
        \phi(x_1) &\subseteq \mathrm{Ball}_{G-Z}(x_1, \tfrac{1}{2} - 3\xi-2\delta),\label{eq:bfm:1}\\
        \phi(x_2) &\subseteq \mathrm{Ball}_{G-Z}(x_2, 4\xi+\delta).\label{eq:bfm:4}
    \end{align}
    Since the length of $P_E^{x_1,x_2}=\phi(x_1x_2)$ is at most $2\xi+\delta$, this implies
    \begin{align}
        \phi(x_1x_2) &\subseteq 
        \Ball_{G-Z}(p^{x_1,x_2}_1, 2\xi + \delta)
        \subseteq
        \mathrm{Ball}_{G-Z}(x_1,\tfrac{1}{2}-\xi-\delta),\label{eq:bfm:2}\\
        \phi(x_1x_2) &
        \subseteq 
        \Ball_{G-Z}(p^{x_1,x_2}_2, 2\xi + \delta)\subseteq \mathrm{Ball}_{G-Z}(x_2,6\xi+2\delta).\label{eq:bfm:3}
    \end{align}
    As $A_1$ is $(1,Z)$-scattered, from~\eqref{eq:bfm:1} we have that for all distinct
    $x_1,x_1' \in A_1$,
    \begin{equation}\label{eq:bfm:11}
    \dist_{G-Z}(\phi(x_1), \phi(x_1')) > 1 - 2(\tfrac{1}{2}-3\xi-2\delta) =  6\xi+4\delta >\delta. 
    \end{equation}
    Similarly, as $A_2$ is $(12\xi+5\delta,Z)$-scattered, by~\eqref{eq:bfm:4} we have for all distinct $x_2,x_2' \in A_2$, 
    \begin{equation}\label{eq:bfm:22}
        \dist_{G-Z}(\phi(x_2), \phi(x_2')) >
        (12\xi+5\delta) - 2(4\xi + \delta)
        =
        4\xi+3\delta > \delta.
    \end{equation}
    As $A_1$ is $(1,Z)$-scattered and every path $P^{\cdot,\cdot}$ is disjoint with $Z$
    and of length at most $\tfrac{1}{2}+\xi$, from \eqref{eq:bfm:1} and \eqref{eq:bfm:2}
    we have that for all distinct $x_1,x_1' \in A_1$ 
    and all (not necessarily distinct) $x_2, x_2' \in A_2$,
    \begin{equation}\label{eq:bfm:dist}
        \dist_{G-Z}(\phi(x_1) \cup \phi(x_1x_2), P^{x_1',x_2'}) > 1-(\tfrac{1}{2}-\xi-\delta)-(\tfrac{1}{2}+\xi)=\delta.
    \end{equation}
    Then we have that for all $x_1\in A_1$ and $x_2 \in A_2$
    \begin{equation}\label{eq:bfm:12}
    \begin{split}
      \dist_{G-Z}(\phi(x_1),\phi(x_2))
        &= \dist_{G-Z}\left(\bigcup_{y \in A_2} P^{x_1,y}_1, \bigcup_{y \in A_1} P^{y,x_2}_2\right) \\
        &= \min_{x_1' \in A_1,\, x_2' \in A_2} \dist_{G-Z}\left(P^{x_1,x_2'}_1, P^{x_1',x_2}_2\right) \\
        &\geq \delta,
    \end{split}
    \end{equation}
    where we conclude $\dist_{G-Z}(P^{x_1,x_2'}_1, P^{x_1',x_2}_2) \geq \delta$ from \eqref{eq:bfm:dist} if $x_1 \neq x_1'$ and from \eqref{eq:delta-gap} if $x_1 = x_1'$.
    
    Let $x_1,x_1' \in A_1$ and $x_2,x_2' \in A_2$.
    If $x_1 \neq x_1'$, then by~\eqref{eq:bfm:dist}
    we have $\dist_{G-Z}(\phi(x_1x_2),\phi(x_1') \cup \phi(x_1'x_2')) > \delta$. 
    On the other hand, if $x_2 \neq x_2'$, then,
    as $A_2$ is $(12\xi+5\delta,Z)$-scattered, by~\eqref{eq:bfm:4} and~\eqref{eq:bfm:3}
    we have \[\dist_{G-Z}(\phi(x_1x_2),\phi(x_2') \cup \phi(x_1'x_2')) > (12\xi+5\delta)-(6\xi+2\delta)-(6\xi+2\delta)=\delta.\]
    These inequalities together imply that if $o_1,o_2\in V(H)\cup E(H)$ are two distinct edges, or a vertex and an edge non-incident to it, then $\dist_{G-Z}(\phi(o_1),\phi(o_2)\geq \delta$. Together with \eqref{eq:bfm:11}, \eqref{eq:bfm:22}, and~\eqref{eq:bfm:12}, this establishes that the model $\phi$ is $\delta$-fat in $G-Z$.

    Finally, note that all the sets $\phi(\cdot)$ are built from parts of paths
    $P^{\cdot,\cdot}$, and hence are within distance at least $\delta/2$ from $Z$.
    So \cref{lem:dist-GZ-to-G} applies and the lower bounds on the distances
    between different sets $\phi(\cdot)$ obtained above hold in $G$, not only in $G-Z$. So $\phi$ is $\delta$-fat in $G$ as well, and this finishes the proof.
\end{proof}

\section{Drill-flatness in the absence of a fat minor}\label{sec:flatness-metric}

In this section we prove the main result of this part, \cref{thm:metricuqwnew}. As explained in \cref{sec:intro}, the proof is a careful lift of the classic reasoning from the non-metric setting, of the equivalence of nowhere denseness and flatness. The main point of inspiration is the proof due to Pilipczuk, Siebertz, and Toru\'nczyk~\cite{DBLP:conf/lics/PilipczukST18a} that provided explicit, polynomial bounds for flatness.

Specifically, in \cref{sec:improve} we present a procedure to increase the scatteredness of a given set $A$ a bit, at the cost of restricting $A$ to its sizable subset and drilling a bounded number of ``holes'' in the given length space. Then, in \cref{sec:iterate} we show that this procedure can be iterated any bounded number of times, which amounts to the proof of \cref{thm:metricuqwnew}.

\subsection{Improving the scatteredness}\label{sec:improve}

In this section we show a procedure that improves the scatteredness of a given set by a small multiplicative factor. We first show the following technical lemma.

\begin{lemma}\label{lem:unit-step}
    Let $h\geq 2$ be an integer, $G$ be a length space, $Z \subseteq G$ 
    and $\delta$, $\varepsilon$ be reals satisfying $0<\delta \leq \varepsilon\leq \tfrac{1}{8}$.
    Assume $A \subseteq G$ is a finite $(1,Z)$-scattered set of size at least $4^5(h+1)^{10}$.
    For $a\in \N$, let $Z^a \coloneqq \Ball_G(Z, a\delta/2)$. 
    Then, one of the following cases holds:
    \begin{itemize}
    \item There exists $v \in G-Z^1$ such 
     \[ \left| A \cap \Ball_{G-Z^1}(v, \tfrac{1}{2} + 5\varepsilon + 3\delta) \right| > |A|^{\tfrac{1}{10}}. \]
    \item There exists $B \subseteq A$ of size at least $|A|^{\tfrac{1}{10}}$ that is $(1+\varepsilon,Z^2)$-scattered.
    \item $G$ contains a $(\tfrac{3}{2}+\varepsilon-\delta,2\varepsilon+\delta)$-shallow $\delta$-fat minor model of $K_h$.
    \end{itemize}
\end{lemma}
\begin{proof}
    We assume that the first two outcomes do not occur, and we aim to construct a model  of $K_h$ witnessing the third outcome.

    Let $H$ be an auxiliary graph with vertex set $A$ where any distinct $x,y \in A-Z^2$ are adjacent
    whenever 
    $\dist_{G-Z^2}(x,y)\leq 1+\varepsilon$.
    (Thus, vertices of $A\cap Z^2$ are isolated in $H$.)
    We claim that 
    \begin{equation}\label{eq:unit:EH-large}
        |E(H)| > \tfrac{1}{16}|A|^{\tfrac{19}{10}}.
    \end{equation}
    Indeed, otherwise the sum of the degrees in $H$ would be at most $\tfrac{1}{8}|A|^{\tfrac{19}{10}}$, implying that
    at least $|A|/2$ vertices of $H$ have degree at most $\frac{1}{4}|A|^{\tfrac{9}{10}}$.
    Note that $\frac{1}{4}|A|^{\tfrac{9}{10}}\leq \frac{1}{2}|A|^{\tfrac{9}{10}} - 1$, as $|A| \geq 4^5(h+1)^{10} \geq 2^{20}$.
    Consequently, $H$ contains an independent set of size at least
    \[ \frac{|A|/2}{\frac{1}{2} |A|^{\tfrac{9}{10}} - 1 + 1} = |A|^{\tfrac{1}{10}}. \]
    Such an independent set would be the second outcome of the lemma. This proves~\eqref{eq:unit:EH-large}.

    For any edge $e \in E(H)$, we may fix a path $P^e$ in $G-Z^2$ of length\footnote{The extra $\varepsilon$ in the bound comes from the fact that we work with length spaces, so for $\dist(x,y) = 1+\varepsilon$ we may not have an $x$-$y$ path of length exactly $1+\varepsilon$, but only of any length larger than $1+\varepsilon$.}
     at most $1+2\varepsilon$
    that connects the endpoints of $e$. 
    Let $m_e$ be the midpoint of $P^e$.
    Consider the multiset $M \coloneqq \{m_e~\colon~e \in E(H)\}$; then $|M| = |E(H)|$.
    Let $K$ be an auxiliary graph with vertex set $M$ where two midpoints $m_e$ and $m_f$ are adjacent whenever $\dist_{G-Z^1}(m_e,m_f) \leq 4\varepsilon+3\delta$. 

    We claim that the maximum degree of $K$ cannot be too large.
    Let $m \in M$ be a vertex of $K$, say of degree $k$. 
    Let $B \subseteq A$ be defined as follows:
    \[
        B \coloneqq \bigcup_{m_e \in N_K[m]} \{v \in A~\mid~ \text{$v$ is an endpoint of $e$}\}.
    \]
    That is, $B$ consists of the endpoints of those edges $e$ that satisfy $m_e \in N_K[m]$. 
    As every point in $M$ is a midpoint of a path of length at most $1+2\varepsilon$, we have
    \[ B \subseteq \Ball_{G-Z^1}(m, \tfrac{1}{2} + 5\varepsilon + 3\delta). \]
    Hence, noting the $B \cap Z^1 = \emptyset$, unless the first outcome occurs we have $|B| \leq |A|^{\tfrac{1}{10}}$. 
    On the other hand, all the edges of $H$ whose midpoints belong to $N_K[m]$ have both endpoints in $B$, hence
    \[ k + 1 = |N_K[m]| \leq \binom{|B|}{2}. \]
    We infer that
    \[ k + 1\leq \tfrac{1}{2}|A|^{\tfrac{1}{5}} .\]
    Since the choice of $m$ is arbitrary, by \eqref{eq:unit:EH-large} we conclude that $K$ contains an independent set $J$ of size 
    \[|J| \geq \frac{|M|}{\frac{1}{2}|A|^{\tfrac{1}{5}}} = \frac{2|E(H)|}{|A|^{\tfrac{1}{5}}} \geq \tfrac{1}{8}|A|^{\tfrac{17}{10}}. \]
    Since $|A| \geq 4^5(h+1)^{10}$, we have
    \begin{equation}\label{eq:unit:J}
     |J| \geq  \tfrac{1}{8}|A|^{\tfrac{1}{5}}|A|^{\tfrac{3}{2}}\geq \tfrac{1}{8}\cdot  4(h+1)^2 |A|^{\tfrac{3}{2}} > \binom{h+1}{2} |A|^{\tfrac{3}{2}}. 
    \end{equation}
    
    Let $H'$ be the subgraph of $H$ consisting of edges $\{e\in E(H)~\mid~m_e \in J\}$ and vertices incident to them. 
    By~\eqref{eq:unit:J} and \cref{thm:alon-density}, $H'$ contains $K_h$ as $1$-shallow minor. 
    
    Let $f(x) \coloneqq x$ for every $x \in V(H')$ and $f(e) \coloneqq m_e$ for every $e \in E(H') = J$. 
    We verify that $f$ satisfies the prerequisites of \cref{lem:points-to-fat-minor}
    for $H'$, $G$, $Z^1$, $\delta$, and $\xi \coloneqq \varepsilon$.

    \begin{itemize}
    \item For every two distinct vertices $u,v \in V(H') \subseteq A$, we have \[\dist_{G-Z^1}(f(u) = u,f(v) = v) > 1,\]
    
    \begin{itemize}
        \item This is because we assumed $A$ to be $(1,Z)$-scattered, hence also $(1,Z^1)$-scattered.
    \end{itemize}

    \item For every two distinct edges $e,e' \in E(H')$, we have 
    \[
        \dist_{G-Z^1}(f(e) = m_{e}, f(e') = m_{e'}) > 4\eps + 3\delta.    
    \]
    \begin{itemize}
        \item This is because $J \supseteq \{m_{e},m_{e'}\}$ is an  independent set in $K$.
    \end{itemize}

    \item For every $e \in E(H')$ and an endpoint $v$ of $e$, 
    there exists  in $G$ an $f(e)$-$f(v)$-path $P_{e,v}$
    of length at most $\frac{1}{2}+\eps$ that does not contain any point within distance
    less than $\delta/2$ from $Z^1$.

    \begin{itemize}
        \item This is because $f(e) = m_e$ is the midpoint of the path $P^e$ of length at most $1+2\eps$ which has $f(v) = v$ as an endpoint. 
    This path is contained in $G - Z^2$, so it contains no point within distance
    less than $\delta/2$ from $Z^1$.
    \end{itemize}
    
    \end{itemize}

    Consequently, by \cref{lem:points-to-fat-minor}, $G$ contains $H'$ as a $(\tfrac{1}{2}-\varepsilon-\delta,2\varepsilon+\delta)$-shallow
    $\delta$-fat minor. 
    As $H'$ contains $K_h$ as $1$-shallow minor, by \cref{lem:minor-of-fat-minor} we infer that
    $G$ contains $K_h$ as a $(\tfrac{3}{2}+\varepsilon-\delta,2\varepsilon+\delta)$-shallow $\delta$-fat minor.
    This is the third outcome, so the proof of the lemma is complete.
\end{proof}

In the next lemma, we iteratively apply \cref{lem:unit-step} and, in case of obtaining the first outcome,
insert into $Z$ a small ball around the output vertex $v$. We show that this process cannot continue for too many
steps, as otherwise we would be able to construct a fat minor of a large biclique using \cref{lem:biclique-fat-minor}.

\begin{restatable}{lemma}{smallstep}\label{lem:small-step}
    Let $h, m$ be positive integers with $h\geq 2$, $G$ be a length space, $Z \subseteq G$, and
    $\delta, \varepsilon$ be reals such that
    $0<\delta \leq \varepsilon\leq \tfrac{1}{200}$.
    Assume that $G$ does not contain $K_h$ as a $(2,10\varepsilon+\delta/2)$-shallow $\delta$-fat minor.
    Let $A \subseteq G$ be a finite $(1,Z)$-scattered set of size at least
    \begin{equation}\label{eq:small:A}
        |A| \geq \left(\max\left(4^5 (h+1)^{10}, m\right)\right)^{10^h}.
    \end{equation}
    Then, there exists a set $X \subseteq G$ of size at most $h-1$
    and a set $B \subseteq A$ of size at least $m$
    that is $(1+\varepsilon, \Ball_G(Z,\delta) \cup \Ball_G(X, 120\varepsilon+6\delta))$-scattered in $G$.
\end{restatable}
\begin{proof}
    We compute sets $A\eqqcolon A_0 \supseteq A_1 \supseteq A_2\supseteq \ldots$
    and points $v_1,v_2,v_3,\ldots \in G$ as follows.
    
    Let $i \geq 0$ and assume $A_i$ and $v_1,\ldots,v_i$ are computed. 
    If $|A_i| < 4^5 (h+1)^{10}$, we stop. Otherwise, we apply \cref{lem:unit-step}
    to $h$, $\delta$, $\varepsilon$, $Z_i \coloneqq Z \cup \bigcup_{j=1}^i \Ball_G(v_j, 120\varepsilon+5\delta)$,
    and $A_i$. ($Z_i$ plays the role of $Z$ in \cref{lem:unit-step}, not to be confused with $Z^a$ in the same lemma.)
    Let $Z^1\coloneqq \Ball_G(Z,\delta/2)$ and $Z_i^1 = \Ball_G(Z_i, \delta/2)$. 
    In case of the first outcome, say witnessed by a vertex $v\in G-Z^1_i$, we set \[v_{i+1} \coloneqq v\qquad \textrm{and}\qquad A_{i+1} \coloneqq  A_i \cap \Ball_{G-Z_i^1}(v, \tfrac{1}{2} + 5\varepsilon+3\delta),\] and increase $i$ by one.
    In case of the second outcome, say witnessed by a set $B$, we stop the procedure.
    The third outcome cannot happen due to the assumption on the excluded shallow fat minor of $G$.

    Observe that by construction, $|A_i| \geq |A|^{\left(1/10\right)^i}$ and $\{v_1,v_2,\ldots\}$ is $(120\varepsilon+5\delta)$-scattered in $G-Z$. 
    Furthermore, for every $1 \leq j \leq i$ and $x \in A_i$, we also have $x\in A_j$ and hence $\dist_{G-Z_{j}^1}(v_j,x)\leq \tfrac{1}{2}+5\varepsilon+3\delta$.
    Consequently, in $G-Z_j^1$ there exists an $x$-$v_j$-path $P^{j,x}$
    of length\footnote{Again, we add the extra $\eps$ summand is due to the fact that $\dist_{G-Z_{j}^1}(v_j,x)\leq \tfrac{1}{2}+5\varepsilon+3\delta$ means that the infimum of the lengths of the $v_j$-$x$-paths in $G-Z_j^1$ is $\tfrac{1}{2}+5\varepsilon+3\delta$, and not that there exists a path of this length.}  at most $\tfrac{1}{2}+6\varepsilon+3\delta$, which is also a path of this length in $G-Z^1$.

    So, if the procedure runs for $h$ steps, it computes the set $A_h$ satisfying $|A_h|\geq h$ as well as the set $\{v_1,\ldots,v_h\}$, to which we may apply \cref{lem:biclique-fat-minor} for $Z$, $\delta$, $\xi \coloneqq 10\varepsilon$; note here that $\varepsilon\leq \tfrac{1}{200}$ implies $\xi\leq \tfrac{1}{20}$. As a result, we conclude that
    $G$ contains a $(\tfrac{1}{2}+10\varepsilon, 10\varepsilon+\delta/2)$-shallow $\delta$-fat model of~$K_{h,h}$.
    Since $K_{h,h}$ contains $K_h$ as $1$-shallow minor, \cref{lem:minor-of-fat-minor}
    implies $G$ contains $K_h$ as a $(\tfrac{3}{2}+50\varepsilon+\delta, 10\varepsilon+\delta/2)$-shallow $\delta$-fat minor, which is a contradiction because $\tfrac{3}{2}+50\varepsilon+\delta<2$.
    
    From~\eqref{eq:small:A}, we infer that for every $0 \leq i \leq h$ 
    we have 
    \[ |A_i| \geq \left(\max\left(4^5(h+1)^{10}, m\right)\right)^{10^{h-i}}. \]
    This implies that the procedure needs to stop by detecting the second outcome of \cref{lem:unit-step}
    at some step $i < h$. Let $B$ be the set output by this outcome.
    Clearly, $|B| \geq |A_i|^{\tfrac{1}{10}} \geq m$. Since \cref{lem:unit-step} is invoked at step $i$
    with $Z_i = Z \cup \bigcup_{j=1}^i \Ball(v_j, 120\varepsilon+5\delta)$, 
    we have that $B$ is $(1+\eps, Z_i^2)$-scattered with
    \begin{align*}
    Z_i^2
        &= \Ball_G\left(Z \cup \bigcup_{j=1}^i \Ball_G(v_j, 120\varepsilon+5\delta),\, \delta\right) \\
        &= \Ball_G\left(Z \cup \Ball_G(X, 120\varepsilon+5\delta),\, \delta\right) \\
        &= \Ball_G(Z, \delta) \cup \Ball_G(X, 120\varepsilon+6\delta)
    \end{align*}
    and $X \coloneqq \{v_1,v_2,\ldots,v_i\}$. This finishes the proof of the lemma.
\end{proof}

\subsection{Iteration}\label{sec:iterate}

We are now ready to finish the main proof.

\metricuqwnew*
\begin{proof}
    We may assume that $h\geq 2$. We set
    \begin{eqnarray*}
        k &\coloneqq &  \max\left(0, \lceil300\cdot (\alpha/\beta-1)\rceil\right),\\
        N(m) &\coloneqq & m \cdot \left(\max\left(4^5(h+1)^{10}, m\right)\right)^{10^{hk}},\\
        s &\coloneqq & 1+(h-1)k.
    \end{eqnarray*}
    Thus we have $N(m)\leq m^{2^{\Oh(h\cdot \alpha/\beta)}}$ and $s\leq \Oh(h\cdot \alpha/\beta)$, as claimed.
    For the rest of the proof, we assume
    without loss of generality that $m \geq 4^5(h+1)^{10}$, as otherwise we can just reset $m \coloneqq 4^5(h+1)^{10}$.
    
    We proceed with the proof.
    First, let $A_0$ be an inclusion-wise maximal $\beta$-scattered set in $A$. Note that by the maximality, every element of $A$ is at distance at most $\beta$ from some element $A_0$. Hence, if we have $|A_0|<|A|/m$, then by the pigeonhole principle there exists some $u\in A_0$ such that $|A\cap \Ball_G(u,\beta)|\geq m$. Consequently, if we set $X\coloneqq \{u\}$ and $B\coloneqq A\cap \Ball_G(u,\beta)$, we find that $B$ is $(\alpha,\Ball_G(X,\beta))$-scattered (or even $(\alpha',\Ball_G(X,\beta))$-scattered for any $\alpha'\in \Rp$ since in this case $B\setminus \Ball_G(X,\beta) = \varnothing$), as required. Hence, from now on, we assume that
    \[|A_0|\geq |A|/m\geq N(m)/m=m^{10^{hk}}.\]
    We will also assume that $\alpha>\beta$, for otherwise $B\coloneqq A_0$ and $X\coloneqq \emptyset$ satisfy the required~conditions.

    Observe that $\alpha > \beta$ implies $k \geq 1$. 
    We also note that the assumption $\beta\geq 25\cdot\sqrt{\alpha\delta}$ implies that $\delta\leq \tfrac{\beta^2}{625\alpha}$, and hence
    \begin{equation}\label{eq:waran}
        \frac{\beta}{2k}=\frac{\beta}{2\cdot \lceil 300\cdot (\alpha/\beta-1)\rceil}\geq \frac{\beta}{600\cdot \alpha/\beta}=\frac{\beta^2}{600\alpha}\geq \delta.
    \end{equation}
    Also, as $\alpha > \beta$, we have
    \begin{equation}\label{eq:waran2}
        \delta \leq \frac{\beta^2}{625\alpha} < \frac{\beta}{625}.
    \end{equation}


    
    Our goal now is to compute sets $A_0 \supseteq A_1 \supseteq \ldots \supseteq A_k$ and
    $\emptyset \eqqcolon X_0 \subseteq X_1 \subseteq \ldots \subseteq X_k$.
    We will maintain the following invariants: 
    \begin{itemize}
    \item $|A_i| \geq m^{10^{h(k-i)}}$,
    \item $|X_i| \leq (h-1)i$,
    \item $A_i$ is $((1+\frac{i}{300})\beta, \Ball(X_i, (\frac{1}{2} + \frac{i}{2k})\beta))$-scattered in $G$.
    \end{itemize}
    Clearly, these invariants are satisfied for $i=0$. 
    Observe that if we construct $A_k$ and $X_k$, then we are done with $B \coloneqq A_k$ and $X \coloneqq X_k$:
    the first two invariants imply that $|A_k| \geq m$ and $|X_k| \leq (h-1)k\leq s$, while the third invariant
    implies that $A_k$ is $(\alpha, \Ball(X,\beta))$-scattered in $G$ by the choice of $k$.
    It remains to show how to construct $A_{i+1}$ and $X_{i+1}$ from $A_i$ and $X_i$.

    Assume then that for some $0 \leq i < k$ we have computed $A_i$ and $X_i$ satisfying these invariants.
    Let \[\ell \coloneqq (1+\tfrac{i}{300})\beta.\] As $i < k$, we have $\ell < \alpha$.
    Let $G'$ be the length space obtained from $G$ by rescaling the metric by a multiplicative factor of $1/\ell$ (i.e., $\dist_{G'}(x,y)=\dist_G(x,y)/\ell$ for all $x,y\in G$). 
    Note that $A_i$ is a $(1,\Ball(X_i,(\frac{1}{2}+\frac{i}{2k})\cdot \beta/\ell))$-scattered set in $G'$. 
    By rescaling, $G'$ does not contain $K_h$ as a
    $2\alpha/\ell$-shallow
    $\delta/\ell$-fat minor. 
    Let \[\delta' \coloneqq \delta/\ell\qquad\textrm{and}\qquad\varepsilon' \coloneqq \tfrac{1}{300}\cdot \beta/\ell.\]
    As $\alpha > \ell\geq \beta$, 
    it follows that $G'$ does not contain $K_h$ as $2$-shallow $\delta'$-fat minor. 
    Furthermore, 
    \[ \delta' = \delta/\ell \leq \tfrac{1}{625}\cdot \beta/\ell \leq \varepsilon' \qquad \mathrm{and}\qquad \varepsilon' = \tfrac{1}{300}\cdot \beta/\ell \leq \tfrac{1}{300},\]
    where the first inequality follows~\eqref{eq:waran2}.

    Let \[m' \coloneqq m^{10^{h(k-i-1)}}.\] Note that $|A_i| \geq (m')^{10h}$. 
    We can now apply \cref{lem:small-step} to 
    $h$, $m'$, $G'$, $\delta'$, $\eps'$, $Z' \coloneqq \Ball(X_i,(\frac{1}{2}+\frac{i}{2k})\cdot \beta/\ell)$
    and $A_i$,
    thus obtaining a set $X \subseteq G'$ of size at most $h-1$ 
    and $B \subseteq A_i$ of size at least $m'$ that is 
    \begin{center}
    $(1+\varepsilon',\Ball_{G'}(Z', \delta') \cup \Ball_{G'}(X, 120\varepsilon'+6\delta'))$-scattered    
    \end{center}
    in $G'$.
    After rescaling back, we have that 
    $B$ is 
    \begin{center}
    $\left(\ell + \varepsilon'\ell, \Ball_G\left(X_i, \left(\tfrac{1}{2} + \tfrac{i}{2k}\right)\beta + \delta'\ell \right) \cup \Ball_G(X, 120\varepsilon'\ell + 6\delta'\ell)\right)$-scattered
    \end{center}
    in $G$. 
    By~\eqref{eq:waran}, \eqref{eq:waran2}, and the choice of $\ell$, $\varepsilon'$, and $\delta'$, we have
    \begin{align*}
    \ell + \varepsilon'\ell &= \left(1 + \frac{i}{300}\right)\beta + \frac{1}{300}\beta = \left(1 + \frac{i+1}{300}\right)\beta,\\
    \left(\frac{1}{2}+\frac{i}{2k}\right)\beta+ \delta'\ell &= \left(\frac{1}{2}+\frac{i}{2k}\right)\beta + \delta \leq \left(\frac{1}{2}+\frac{i}{2k}\right)\beta + \frac{\beta}{2k} = \left(\frac{1}{2}+\frac{i+1}{2k}\right)\beta ,\\
    120\varepsilon'\ell + 6\delta'\ell &= \frac{120}{300}\beta + 6\delta \leq \frac{126}{300}\beta < \frac{1}{2}\beta. 
    \end{align*}
    Hence, $A_{i+1} \coloneqq B$ is $(1+\frac{i+1}{300}\beta, \Ball_G(X_{i+1}, (\frac{1}{2}+\frac{i+1}{2k})\beta))$-scattered
    in $G$, where $X_{i+1} \coloneqq X_i \cup X$. 
    This finishes the proof.
\end{proof}

\section{Approximation of \textsc{Norm $k$-Clustering} in fat-minor-free metrics}\label{sec:ksupplier}

In this section we give an approximation scheme for \textsc{Norm $k$-Clustering}, a general clustering problem proposed by Abbasi et al.~\cite{ABBCGKMSS23} as a common generalization of \textsc{$k$-Center}, \textsc{$k$-Median}, \textsc{$k$-Means}, and many others. Let us recall the definition and terminology.

An instance of \textsc{Norm $k$-Clustering} consists of:
\begin{itemize}
    \item an integer $k\in \N$;
    \item a metric space $G$;
    \item a set of potential facilities $F\subseteq G$;
    \item a multiset of clients $C\subseteq G$; and
    \item a polynomial-time computable monotone norm $\|\cdot\|$ on $\R^C$ (the real vector space of dimension $|C|$, indexed by the elements of $C$).
\end{itemize}
That $\|\cdot\|$ is \EMPH{monotone} means that if $\bar x,\bar y\in \R_{\geq 0}^C$ satisfy that $\bar x$ is not larger than $\bar y$ on each coordinate, then $\|\bar x\|\leq \|\bar y\|$. 

A few words of caution regarding the format of the input are necessary. First, the metric space $G$ is not given explicitly, as it may be infinite, but it will be convenient to think that $F$ and $C$ reside in some ambient metric space $G$ (which will be a length space in our case). For the problem itself the only relevant distances are those between the elements of $F\cup C$, hence we simply assume that on input we are given a matrix of distances between the elements of $F\cup C$. Consequently, it should not be surprising that the time complexity of the algorithms is expressed only in terms of $|F|+|C|$. Second, the norm $\|\cdot\|$ is given by a polynomial-time oracle to compute it.

A \EMPH{solution} to an instance $(k,G,F,C,\|\cdot\|)$ of \textsc{Norm $k$-Clustering} is a set of facilities $S\subseteq F$ with $|S|\leq k$. As usual, for a solution $S$ and a client $c\in C$, $\dist_G(c,S)$ is the distance between $c$ and the closest facility of the solution. Then, we define the \EMPH{cost} of the solution as
\[\val(S)\coloneqq \|(\dist_G(c,S)\colon c\in C)\|.\]
In other words, we take the vector of distances from the clients to the solution, and compute its $\|\cdot\|$-norm. By $\OPT$ we denote the optimum value of a solution. Then \textsc{Norm $k$-Clustering} is the optimization problem of finding a solution with as small cost as possible.

In the context of an instance of \textsc{Norm $k$-Clustering}, $\mathbf{1}$ denotes the all-ones vector in $\R^C$.

\subsection{Scatters}

We now recall the definition of \EMPH{scatters} and \EMPH{$\eps$-scatter dimension}, introduced by Abbasi et al.~\cite{ABBCGKMSS23} in the context of approximation schemes for \textsc{Norm $k$-Clustering}. As noted by Bourneuf and Ma.~Pilipczuk~\cite{BP25}, it is a metric variant of \EMPH{semi-ladders} of Fabia\'nski et al.~\cite{FPST18}.

\begin{definition}[$(\eps,r)$-scatter]
	Let $G$ be a metric space and $\eps,r\in \Rp$. An \EMPH{$(\eps,r)$-scatter} of order $\ell$ in $G$ is a pair of sequences $(a_1,\ldots,a_\ell)$ and $(b_1,\ldots,b_\ell)$ of points in $G$ such that
	\begin{itemize}
		\item $\dist_G(a_i,b_i)>(1+\eps)r$ for all $i\in [\ell]$, and
		\item $\dist_G(a_i,b_j)\leq r$ for all $i,j\in [\ell]$ with $i<j$.
	\end{itemize}
\end{definition}

\begin{definition}[$\eps$-scatter dimension]
    A class of metric spaces $\Cc$ has \EMPH{bounded $\eps$-scatter dimension} if there exists a function $\ell\colon (0,1)\to \N$ such that for every metric space $G\in \Cc$, $\eps\in (0,1)$, and $r\in \Rp$, $G$ does not contain any $(\eps,r)$-scatter of order larger than $\ell(1/\eps)$.
\end{definition}

Note that in this definition, the bound on the order of scatters must be independent of the distance~$r$. In other words, the orders of $(\eps,r)$-scatters are bounded uniformly for every possible $r$.

As proved by Bourneuf and Ma.\ Pilipczuk~\cite{BP25}, for every fixed graph $H$, the class of edge-weighted $H$-minor-free graphs has bounded $\eps$-scatter dimension. This generalized earlier results of Abbasi et al.~\cite{ABBCGKMSS23} for edge-weighted planar graphs and bounded-treewidth graphs. Besides this, Abbasi et al. observed that metrics of bounded doubling dimension also have bounded $\eps$-scatter dimension; this in particular applies to Euclidean spaces of fixed dimension.

The following statement summarizes the algorithmic results of Abbasi et al.

\begin{theorem}[\cite{ABBCGKMSS23}]\label{thm:abbasi}
    Suppose $\ell\in \N$ and $I=(k,G,F,C,\|\cdot\|)$ is an instance of \textsc{Norm $k$-Clustering} such that in $G$ there are no $(\eps,r)$-scatters of order larger than $\ell$, for any $r\in \Rp$. Then one can in time $\Oh_{k,\ell}((|F|+|C|)^{\Oh(1)})$ compute a solution to $I$ of cost at most $(1+\eps)\OPT$.

    In particular, if $\Cc$ is a class of metric spaces with bounded $\eps$-scatter dimension, then for every $\eps>0$, there exists a $(1+\eps)$-approximation algorithm for \textsc{Norm $k$-Clustering} on $\Cc$ with running time $\Oh_{\Cc,k,\eps}((|F|+|C|)^{\Oh(1)})$.
\end{theorem}

Here, we use drill-flatness to prove that in fat-minor-free length spaces, there is a bound on the order of $(\eps,r)$-scatters for distances $r$ that are large compared to the fatness of the excluded minor.
The proof of the following lemma is essentially the
same as~\cite[Lemma~4.1]{BP25} and also can be seen as an adaptation of the proof of \cite[Lemma~29]{FPST18} to the metric setting.

\begin{lemma}\label{lem:flat-no-scatter}
	Let $\eps\in (0,1)$ and $r\in \Rp$. Suppose $G$ is a length space that is $(2r,\eps r/3)$-metric-flat with wideness $N\colon \N\to \N$ and margin $s\in \N$. Then $G$ does not contain an $(\eps,r)$-scatter of order $N(2\cdot(2+\lfloor 6/\eps\rfloor)^s+1)$.
\end{lemma}
\begin{proof}
	Let $\rho\coloneqq \eps r/3$. For contradiction suppose that in $G$ there is an $(\eps,r)$-scatter $a_1,\ldots,a_\ell,b_1,\ldots,b_\ell$ of some order $\ell\geq N(2\cdot(2+\lfloor 6/\eps\rfloor)^s+1)$. Note that the points $a_1,\ldots,a_\ell$ must be pairwise different, as they have pairwise different vectors of distances to $b_1,\ldots,b_\ell$. By applying the assumed metric-flatness to $\{b_1,\ldots,b_\ell\}$ and passing to subsequences of $a_1,\ldots,a_\ell$ and $b_1,\ldots,b_\ell$, at the cost of reducing the lower bound on $\ell$ to $\ell\geq 2\cdot(2+\lfloor 6/\eps\rfloor)^s+1$ we may assume the following: 
	\begin{equation}\label{c:flat}
		\textrm{
		There is as set $X$ with $|X|\leq s$ such that $\{b_1,\ldots,b_\ell\}$ is $(2r,\Ball(X,\rho))$-scattered.
		}
	\end{equation}

	\newcommand{\profN}[1]{\mathsf{prof}_{#1}}

	For every $i\in [\ell]$, we define the \emph{profile} of $b_i$ on $X$ as a function $\profN{b_i}\colon X\to \{0,1,\ldots,\lfloor 6/\eps\rfloor,+\infty\}$ follows: for $x\in X$,	
	\[\prof{b_i}{x}\coloneqq \begin{cases} \textrm{largest }j\in \N \textrm{ such that } j\cdot\rho\leq \dist_{G}(b_i,x) & \quad\textrm{if }\dist_G(b_i,x)\leq 2r,\\
	+\infty & \quad\textrm{otherwise.}
	\end{cases}\]
	Note that there are at most $(2+\lfloor 6/\eps\rfloor)^s$ distinct profiles. Hence, by applying the pigeonhole principle and passing to subsequences of $a_1,\ldots,a_\ell$ and $b_1,\ldots,b_\ell$, at the cost of reducing the lower bound on $\ell$ to $\ell\geq 3$, we may assume the following.
	\begin{equation}\label{c:profiles}
		\textrm{All the vertices $b_1,\ldots,b_\ell$ have the same profile on $X$.}
	\end{equation}

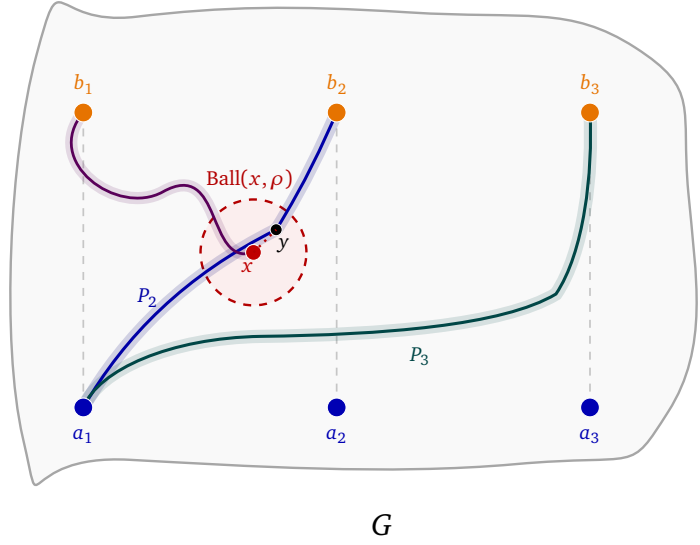
\begin{figure}[htbp]
    \centering
    \begin{tikzpicture}[
    >=Latex,
    every node/.style={font=\small},
    apt/.style={
        circle,
        fill=blue!70!black,
        draw=white,
        line width=.4pt,
        inner sep=2.5pt
    },
    bpt/.style={
        circle,
        fill=orange!90!black,
        draw=white,
        line width=.4pt,
        inner sep=2.5pt
    },
    xpt/.style={
        circle,
        fill=red!75!black,
        draw=white,
        line width=.4pt,
        inner sep=2.1pt
    },
    ypt/.style={
        circle,
        fill=black,
        draw=white,
        line width=.35pt,
        inner sep=1.6pt
    },
    p2tube/.style={
        blue!45!black,
        line width=5pt,
        opacity=.16,
        line cap=round
    },
    p3tube/.style={
        teal!55!black,
        line width=5pt,
        opacity=.14,
        line cap=round
    },
    bxtube/.style={
        violet!65!black,
        line width=4.6pt,
        opacity=.13,
        line cap=round
    },
    p2/.style={
        blue!65!black,
        line width=1.15pt,
        line cap=round
    },
    p3/.style={
        teal!55!black,
        line width=1.15pt,
        line cap=round
    },
    bxpath/.style={
        violet!70!black,
        line width=1.05pt,
        line cap=round
    },
    ball/.style={
        draw=red!70!black,
        dashed,
        fill=red!15,
        fill opacity=.34,
        line width=.9pt
    },
    alignline/.style={
        gray!45,
        dashed,
        line width=.65pt
    },
    smalllabel/.style={font=\scriptsize},
    note/.style={
        font=\scriptsize,
        align=center,
        text width=10.6cm
    }
]

\def\Gwidth{8.80}   
\def\Gheight{5.80}  

\coordinate (Gcenter) at (3.60,.35);

\begin{scope}[
    shift={(Gcenter)},
    x={(\Gwidth cm,0cm)},
    y={(0cm,\Gheight cm)}
]
    \path[fill=gray!5,draw=black!35,line width=.9pt]
        plot[smooth cycle,tension=.85] coordinates {
            (-.50,-.44)
            (-.48,.43)
            (-.33,.52)
            (.08,.53)
            (.42,.45)
            (.50,.12)
            (.47,-.38)
            (.24,-.50)
            (-.34,-.49)
        };

\end{scope}

\node[font=\large] at (3.95,-3.35) {$G$};

\coordinate (b1) at (0,2.10);
\coordinate (b2) at (3.35,2.10);
\coordinate (b3) at (6.70,2.10);

\coordinate (a1) at (0,-1.80);
\coordinate (a2) at (3.35,-1.80);
\coordinate (a3) at (6.70,-1.80);

\foreach \i in {1,2,3}{
    \draw[alignline] (b\i) -- (a\i);
}

\coordinate (x) at (2.25,.25);
\coordinate (y) at (2.55,.55);

\draw[ball] (x) circle (.70);
\node[smalllabel,red!70!black] at ($(x)+(-.05,.96)$) {$\mathrm{Ball}(x,\rho)$};


\draw[p2tube]
    (a1)
    .. controls (.55,-.85) and (1.35,-.05) .. (y)
    .. controls (2.95,1.20) and (3.15,1.65) .. (b2);

\draw[p3tube]
    (a1)
    .. controls (.20,-1.30) and (1.15,-.88) .. (2.35,-.86)
    .. controls (3.95,-.84) and (5.55,-.72) .. (6.25,-.30)
    .. controls (6.62,.25) and (6.74,1.30) .. (b3);

\draw[bxtube]
    (b1)
    .. controls (-.55,1.45) and (.45,.72) .. (1.05,1.05)
    .. controls (1.85,1.50) and (1.62,.05) .. (x);

\draw[p2]
    (a1)
    .. controls (.55,-.85) and (1.35,-.05) .. (y)
    .. controls (2.95,1.20) and (3.15,1.65) .. (b2);

\draw[p3]
    (a1)
    .. controls (.20,-1.30) and (1.15,-.88) .. (2.35,-.86)
    .. controls (3.95,-.84) and (5.55,-.72) .. (6.25,-.30)
    .. controls (6.62,.25) and (6.74,1.30) .. (b3);

\draw[bxpath]
    (b1)
    .. controls (-.55,1.45) and (.45,.72) .. (1.05,1.05)
    .. controls (1.85,1.50) and (1.62,.05) .. (x);

\node[bpt] at (b1) {};
\node[bpt] at (b2) {};
\node[bpt] at (b3) {};

\node[apt] at (a1) {};
\node[apt] at (a2) {};
\node[apt] at (a3) {};

\node[xpt] at (x) {};
\node[ypt] at (y) {};

\node[smalllabel,orange!90!black] at ($(b1)+(0,.38)$) {$b_1$};
\node[smalllabel,orange!90!black] at ($(b2)+(0,.38)$) {$b_2$};
\node[smalllabel,orange!90!black] at ($(b3)+(0,.38)$) {$b_3$};

\node[smalllabel,blue!70!black] at ($(a1)+(0,-.38)$) {$a_1$};
\node[smalllabel,blue!70!black] at ($(a2)+(0,-.38)$) {$a_2$};
\node[smalllabel,blue!70!black] at ($(a3)+(0,-.38)$) {$a_3$};

\node[smalllabel,red!75!black] at ($(x)+(-.08,-.20)$) {$x$};
\node[smalllabel] at ($(y)+(.10,-.2)$) {$y$};

\draw[red!70!black,dotted,line width=.9pt] (x) -- (y);

\node[smalllabel,blue!65!black] at (0.85,-.35)
    {$P_2$};

\node[smalllabel,teal!55!black] at (4.45,-1.12)
    {$P_3$};







\end{tikzpicture}
    \vspace{-2em}
    \caption{ $P_2$ intersects $\Ball(x,\rho)$ at a point $y$.}
    \label{fig:3path}
\end{figure}

	As $a_1,\ldots,a_\ell$ and $b_1,\ldots,b_\ell$ form an $(\eps,r)$-scatter, we have $\dist_G(a_1,b_2)\leq r$ and $\dist_G(a_1,b_3)\leq r$, witnessed by an $a_1$-$b_2$-path $P_2$ of length at most $r$ and an  $a_1$-$b_3$-path $P_3$ of length at most $r$ in $G$. The union of those two paths contains a $b_2$-$b_3$-path of length at most $2r$, which, by \eqref{c:flat}, must intersect $\Ball(x,\rho)$ for some $x\in X$. 
	Therefore, $P_2$ or $P_3$ intersects $\Ball(x,\rho)$; we assume that $P_2$ does, the proof in the other case is entirely symmetric. See \Cref{fig:3path}. By triangle inequality we infer that
	\begin{equation}\label{eq:kookaburra}
		\dist_G(b_2,x)+\dist_G(x,a_1)\leq \len(P_2)+2\rho\leq r+2\rho.
	\end{equation}
	In particular, $\dist_G(b_2,x)\leq r+2\rho\leq 2r$.
	
	As $\prof{b_1}{x}=\prof{b_2}{x}$ and this value is finite due to $\dist_G(b_2,x)\leq 2r$, it follows that 
	\begin{equation}\label{eq:sroka}
	\dist_G(b_1,x)\leq \dist_G(b_2,x)+\rho.
	\end{equation}
	By combining \eqref{eq:kookaburra} and \eqref{eq:sroka}, we conclude that
	\[\dist_G(a_1,b_1)\leq \dist_G(b_1,x)+\dist_G(x,a_1)\leq \dist_G(b_2,x)+\rho+\dist_G(x,a_1)\leq r+3\rho = (1+\eps)r.\]
	This is a contradiction with $a_1,\ldots,a_\ell$ and $b_1,\ldots,b_\ell$ being an $(\eps,r)$-scatter.
\end{proof}


By combining \cref{lem:flat-no-scatter,thm:metricuqwnew}, we easily obtain the statement promised in \cref{sec:intro}.

\scatter*
\begin{proof}
    We set
    \[c\coloneqq 25^2\cdot 3^2\cdot 2\cdot \eps^{-2}=11250\cdot \eps^{-2}.\]
    Consider any $r\geq c\delta$ and let 
    \[\alpha\coloneqq 2r\qquad \text{and}\qquad \beta\coloneqq \eps r/3.\]
    The assumption that $r\geq c\delta$ implies that
    \[25\sqrt{\alpha\delta}\leq 25\sqrt{2r^2/c}=\eps r/3=\beta.\]
    Hence, we may apply \cref{thm:metricuqwnew} and conclude that $G$ is $(\alpha,\beta)=(2r,\eps r/3)$-drill-flat with overhead $N(m)\leq m^{2^{\Oh(h/\eps)}}$ and budget $s\leq \Oh(h/\eps)$.
	We may now apply \cref{lem:flat-no-scatter} to conclude that $G$ does not contain any $(\eps,r)$-scatter of order $\ell$, for some $\ell \leq 2^{2^{\Oh(h/\eps)}}$.
\end{proof}

\newcommand{\snap}{\mathsf{snap}}

\subsection{Approximation in fat-minor-free metrics}

We have now all the tools to prove \cref{thm:apx}. The proof relies on a quite standard ``coarsening'' argument that makes all the non-zero distances large.


\apxClust*
\begin{proof}
    Let $I\coloneqq (k,G,F,C,\|\cdot\|)$ be the input instance.
    Without loss of generality we may assume that $\eps<0$.
    Set $r^\circ\coloneqq c\delta$, where $c\leq \Oh(1/\eps^2)$ is the constant provided by \cref{cor:scatter-bound} for $\eps$; thus $r^\circ\leq \Oh(\delta/\eps^2)$. By \cref{cor:scatter-bound}, there exists $\ell\leq 2^{2^{\Oh(h/\eps)}}$ such that $G$ does not contain any $(\eps,r)$-scatters of order larger than $\ell$, for any $r\geq r^\circ$. 
    
    Let $W\subseteq G$ be the set of all the points of $G$ featured in $F$ or $C$; thus $|W|\leq |F|+|C|$. Next, let $\wh{W}\subseteq W$ be an inclusion-wise maximal $r^\circ$-scattered set in $W$, constructed greedily in time $|W|^{\Oh(1)}\leq (|F|+|C|)^{\Oh(1)}$. For every $w\in W$, let $\snap(w)\in \wh{W}$ be the point of $\wh{W}$ that is the closest to $w$ (ties broken arbitrarily). By the maximality of $\wh{W}$, we have
    \[\dist_G(w,\snap(w))\leq r^\circ\qquad\textrm{for all }w\in W.\]

    Next, define
    \begin{itemize}
        \item $\wh{G}$ to be the metric space obtained from $G$ by restricting the ground set to $\wh{W}$ and inheriting the metric from $G$ (note: this is no longer a length space, as the ground set of $\wh{G}$ is finite);
        \item $\wh{F}\coloneqq \snap(F)$; and
        \item $\wh{C}\coloneqq \snap(C)$ (as a multiset).
    \end{itemize}
    We note that this construction allows us to extend the assumption on the exclusions of large scatters to all the distances.

    \begin{claim}\label{cl:fullScattDim}
        For any $r\in \Rp$, $\wh{G}$ does not contain any $(\eps,r)$-scatter of order larger than $\max(\ell,2)$.
    \end{claim}
    \begin{claimproof}
        If $r>r^\circ$, then the claim follows from the non-existence of such scatters in $G$, because $\wh{G}$ inherits the metric from $G$. And if $r\leq r^\circ$, then $\wh{G}$ cannot have any $(\eps,r)$-scatter of order larger than~$2$, because any two distinct points in $\wh{G}$ are at distance larger than $r^\circ$. 
    \end{claimproof}
    
    We now construct a new instance $\wh{I}\coloneqq (k,\wh{G},\wh{F},\wh{S},\|\cdot\|)$ of \textsc{Norm $k$-Clustering}. In the sequel, by $\OPT$ and $\wh{\OPT}$ we respectively denote the minimum costs of solutions in $I$ and in $\wh{I}$, and by $\val_I(\cdot)$ and $\val_{\wh{I}}(\cdot)$ the costs of solutions in $I$ and in $\wh{I}$. 

    \begin{claim}\label{cl:compareOPT}
        We have
        \[\OPT-2r^\circ\cdot\|\mathbf{1}\|\leq \wh{\OPT}\leq \OPT+2r^\circ\cdot\|\mathbf{1}\|.\]
        Moreover, every solution $\wh{S}$ to $\wh{I}$ can be in time $(|F|+|C|)^{\Oh(1)}$ transformed to a solution $S$ to $I$ such that $\val_I(S)\leq \val_{\wh{I}}(\wh{S})+2r^\circ\cdot\|\mathbf{1}\|$. 
    \end{claim}
    \begin{claimproof}
        Consider any solution $\wh{S}\subseteq \wh{F}$ to $\wh{I}$. For every $\wh{s}\in \wh{S}$, pick any $s\in F$ such that $\snap(s)=\wh{s}$ and let $S$ be the set of those points $s$. Thus $|S|\leq |\wh{S}|\leq k$ and $S$ is a solution to $I$. Note that since $\dist_G(s,\snap(s)=\wh{s})\leq r^\circ$, we have $\dist_G(w,S)\leq \dist_G(w,\wh{S})+r^\circ$ for each $w\in W$. Further, since $\dist_G(c,\snap(c))\leq r^\circ$ for each $c\in C$, by triangle inequality we have
        \[\dist_G(c,S)\leq \dist_G(\snap(c),\wh{S})+2r^\circ.\]
        This means that the vector $(\dist_G(c,S)\colon c\in C)$ is coordinate-wise smaller or equal than the vector $(\dist_G(\snap(c),\wh{S})+2r^\circ\colon c\in C)$. By the monotonicity of the norm $\|\cdot\|$ and triangle inequality, we conclude that
        \[\val_I(S)\leq \val_{\wh{I}}(\wh{S})+2r^\circ\cdot \|\mathbf{1}\|.\]
        By applying this conclusion to $\wh{S}$ being an optimum solution to $\wh{I}$, we infer that $\OPT\leq \wh{\OPT}+2r^\circ\cdot \|\mathbf{1}\|$, or equivalently $\OPT-2r^\circ\cdot \|\mathbf{1}\|\leq \wh{\OPT}$.

        A symmetric reasoning yields that $\wh{\OPT}\leq \OPT+2r^\circ\cdot \|\mathbf{1}\|$ as well.
    \end{claimproof}

    By \cref{cl:fullScattDim}, we may apply the algorithm of \cref{thm:abbasi} to the instance $\wh{I}$ and thus, in time $\Oh_{k,\eps}((|F|+|C|)^{\Oh(1)})$, compute a solution $\wh{S}$ to $\wh{I}$ with $\val_{\wh{I}}(\wh{S})\leq (1+\eps)\wh{\OPT}$. By \cref{cl:compareOPT}, this solution can be transformed into a solution $S$ to $I$ with
    \begin{align*}\val_I(S)&\leq \val_{\wh{I}}(\wh{S})+2r^\circ\cdot\|\mathbf{1}\|\\ &\leq (1+\eps)\wh{\OPT}+2r^\circ\cdot \|\mathbf{1}\|\\
    &\leq (1+\eps)\OPT+(4+2\eps)r^\circ\cdot\|\mathbf{1}\|\\
    &\leq (1+\eps)\OPT+\Oh(\delta/\eps^2)\cdot \|\mathbf{1}\|.\end{align*}
    This concludes the proof.
\end{proof}

\part{Drill-flatness in hereditary graph classes}

The goal of this part is to prove the following.

\indeq*

\noindent \Cref{fig:implications} gives an overview of the proof of \Cref{thm:equivalences}.

\begin{figure}[ht]
\centering
\scalebox{0.9}{
\begin{tikzpicture}[%
  node distance=27mm,>=Latex,
  every edge/.style={draw=black,thin},
  block/.style={
    draw,
    fill=white,
    rectangle,
    minimum width={3.2cm},
    minimum height={1.4cm},
    align = center
  }
  ]

\node[block] (ismf)
{\ref{c:ismf} ISMF};

\node[block, right = 4cm of ismf] (df)
{\ref{c:df} drill-flat};

\node[block, below = 1.5cm of df] (fsmf)
{\ref{c:fsmf} FSMF};

\node[block, below = 1.5cm of ismf] (patt)
{\ref{c:patt} no $\clique{r}{k}$ \\ and $\web{r}{k}$};

\path[->]
(ismf) edge node[above] {\Cref{thm:indminoruqw}} (df);

\path[->]
(df) edge node[right] {\Cref{lem:drill-flat-implies-fsmf}} (fsmf);

\path[->]
(fsmf) edge node[below] {\Cref{lem:fsmf-implies-pattern-free}} (patt);

\path[->]
(patt) edge node[left] {\Cref{thm:patterns}} (ismf);

\end{tikzpicture}
}
\caption{The implications comprising \cref{thm:equivalences}.}
\label{fig:implications}
\end{figure}
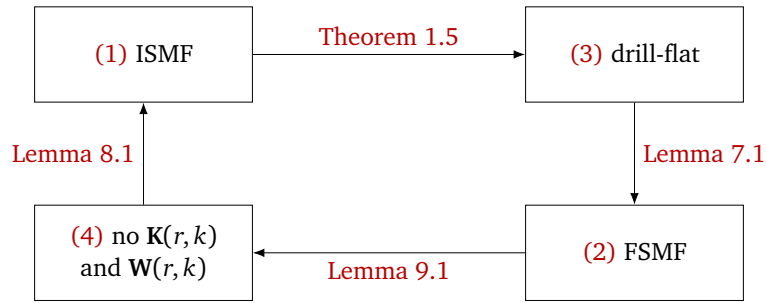

We collect the definitions of the notions comprising \Cref{thm:equivalences} here.
Induced-shallow-minor-freeness is a natural analog of nowhere denseness.

\begin{definition}[ISMF]\label{def:ismf}
  A graph class $\CC$ is \emph{induced-shallow-minor-free (ISMF)} if for every $d \in \N$ there exists a graph $H_d$ such that every $G \in \CC$ excludes $H_d$ as an $d$-shallow induced minor.
\end{definition}

\paragraph{Fat minors in unweighted graphs.}
For fat-shallow-minor-freeness, we need to note that, in literature, fat minors in unweighted graphs are defined slightly differently
than the length space definition we used so far.
For (unweighted) graphs $G$ and $H$ and $\delta > 0$, 
we say that $G$ contains $H$ as a $\delta$-fat minor if there is a mapping
$\phi$ from vertices and edges of $H$ to nonempty connected subgraphs of $G$
so that 
\begin{itemize}
\item whenever vertex $u$ is an endpoint of an edge $e$, the sets $\phi(u)$ and $\phi(e)$ intersect; and 
\item except for the above, for any two objects $o_1,o_2 \in V(H) \cup E(H)$, we have $\dist_G(\phi(o_1), \phi(o_2)) \geq \delta$.
\end{itemize}
The slight difference is that, if we treat $G$ as a metric graph, 
the length space definition of a fat minor allows additionally $\phi(u)$ and $\phi(e)$
to end (and meet in the first bullet) in the middle of an edge. 
This difference only slightly influences the constants, which is irrelevant for the results of this section.

\begin{lemma}\label{lem:equiv-fat-minors}
  Let $G$ and $H$ be unweighted graphs and let $d, \delta > 0$. 
  \begin{enumerate}
    \item If $G$ as an unweighted graph contains $H$ as a $d$-shallow $\delta$-fat minor,
    then $G$ as a metric graph contains $H$ as a $d$-shallow $\delta$-fat minor. 
    \item If $G$ as a metric graph contains $H$ as a $d$-shallow $(\delta+2)$-fat minor,
    then $G$ as an unweighted graph contains $H$ as a $(d+1)$-shallow $\delta$-fat minor. 
  \end{enumerate}
\end{lemma}
\begin{proof}
  The first point is immediate: a branch set being connected subgraph of $G$
  is also a connected subset of the metric graph $G$. 

  For the second point, let $\phi$ be the minor model of $H$
  in the metric graph $G$. Without loss of generality, assume that
  for every $e = uv \in E(H)$, the branch set $\phi(e)$ is a minimal path connecting
  a point of $\phi(u)$ with a point of $\phi(v)$. 
  For every $e \in E(H)$ and $u \in e$, if the endpoint of $\phi(e)$ belonging to $\phi(u)$
  is in the middle of an edge $f$ of $G$, move it (by prolonging $\phi(e)$ and removing the corresponding part from $\phi(u)$)
  to the endpoint of $f$ that belongs fo $\phi(u)$ (such an endpoint exists as $H$ is a $(\delta+2)$-fat-minor.
  This operation may increase the radius of $\phi(e)$ by at most $1$, decrease fatness by at most $2$,
  and does not violate the connectivity nor radius properties of $\phi(u)$. 
\end{proof}

With this discussion in mind, we can state the definition of fat-shallow-minor-freeness.
\begin{definition}[FSMF]\label{def:fsmf}
  A graph class $\CC$ is \emph{fat-shallow-minor-free (FSMF)} if there exists $\delta\in\N$ such that for every $\alpha \in \N$, there exists a graph $H_\alpha$ such that every $G \in \CC$ excludes $H_\alpha$ as an  $\alpha$-shallow $\delta$-fat minor.
\end{definition}

\paragraph{Drill-flatness in unweighted graphs.}
The definition of drill-flatness suffers from the same irrelevant minor differences 
between length spaces and unweighted graphs. 
In the definition of drill-flatness for an unweighted graph $G$,
we would require $A \subseteq V(G)$ and guarantee $X \subseteq V(G)$.
That is, the input points can be only on vertices of $G$ (as opposed to points
in the middle of an edge in the metric graph of $G$)
and the centers of the drilled holes can also be only vertices. 

Similarly, as in the case for the fatness definition, whether we think of an unweighted
graph $G$ as a graph or a metric graph does not matter (up to small additive change of the parameters).

\begin{lemma}\label{lem:equiv-flatness}
  Let $G$ be a graph and $\alpha,\beta \in \Rp$.
  \begin{enumerate}
  \item If $G$ as an unweighted graph is $(\alpha + 1,\beta)$-drill-flat
  with overhead $N$ and budget $s$, then $G$
  as a metric graph is $(\alpha, \beta + \frac{1}{2})$-drill-flat
  with overhead $m \mapsto (m-1) \cdot N(m)$ and budget $\max(1,s)$.
  \item If $G$ as a metric graph is $(\alpha,\beta)$-drill-flat with
  overhead $N$ and budget $s$, then $G$
  as an unweighted graph is $(\alpha, \beta+\frac{1}{2})$-drill-flat with overhead $N$ and budget $s$.
  \end{enumerate}
\end{lemma}
\begin{proof}
For every point $x$ of the metric graph $G$, let $f(x)$ be the vertex of $G$ closest to $x$, breaking ties arbitrarily if $x$ is the midpoint of an edge. 
Clearly, $\dist_G(x,f(x)) \leq \frac{1}{2}$.

We start with the first point. 
Let $m \in N$ and let $A$ be a set of points in the metric graph $G$
of size $|A| > (m-1) \cdot N(m)$.
If for some $v \in V(G)$, we have $|f^{-1}(v)| \geq m$, then we can return $X = \{v\}$
and $B \coloneqq A \cap f^{-1}(v)$ as then $B \subseteq \Ball_G(v, \frac{1}{2})$.

Otherwise, $A' \coloneqq f(A)$ is of size $|A'| > N(m)$. 
Apply the assumed $(\alpha+1,\beta)$-drill-flatness of $G$ as an unweighted graph, 
obtaining $B' \subseteq A'$ of size $m$ and $X \subseteq V(G)$ of size at most $s$.
For every $b \in B'$, pick one $g(b) \in A$ with $f(g(b)) = b$. 
Let $B = \{g(b)~|~b \in B'\}$. Clearly $|B| = m$; to finish
the proof of the first point, it suffices to prove that $B$ is $(\alpha,\Ball_G(X,\beta+\frac{1}{2}))$-scattered.

Consider a path $P$ of length at most $\alpha$ between two distinct $b_1,b_2 \in B$.
Prolong it to a walk $P'$ of length at most $\alpha+1$ between $f(b_1)$ and $f(b_2)$ in $B'$.
Since $B'$ is $(\alpha+1,\Ball_G(X,\beta))$-scattered, the walk $P'$ 
contains a point $p$ within distance at most $\beta$ from some $x \in X$.
If $p$ lies on $P$, we are done. Otherwise, $p$ lies on the shortest path between
$b_i$ and $f(b_i)$ for $i \in \{1,2\}$, so $\dist_G(p,b_i) \leq \frac{1}{2}$
and, consequently, $b_i \in \Ball_G(x,\beta+\frac{1}{2})$. This finishes the proof of the first point.

For the second point, let $A \subseteq V(G)$ be of size $|A| > N(m)$. 
Apply $(\alpha,\beta)$-drill-flatness of $G$ as a metric graph to $A$,
obtaining $B$ and $X$. Let $X' = \{f(x)~|~x \in X\}$. 
Clearly, $\Ball_G(x,\beta) \subseteq \Ball_G(f(x),\beta+\frac{1}{2})$. 
Hence, $B$ is $(\alpha,\Ball_G(X',\beta+\frac{1}{2}))$-scattered, as desired.
\end{proof}

With this discussion in mind, we can define drill-flatness of a class of graphs.

\begin{definition}[Drill-flatness for graph classes]\label{def:class-drill-flat}
    A graph class $\CC$ is \emph{drill-flat} if there exists a drill radius $\beta\in\N$ such that for every scatter radius $\alpha\in \N$, there exists $N\colon \N \to \N$ and $s \in \N$ such that every graph $G\in \CC$ is $(\alpha,\beta)$-drill-flat with overhead $N$ and budget $s$.
\end{definition}

\paragraph{Forbidden patterns.}
Finally, let us define the patterns mentioned in the last point of
\cref{thm:equivalences}. See \Cref{fig:patterns} for an illustration.

\begin{definition}[Patterns]\label{def:patterns}
  Fix $r \geq 3$ and $n\in\N$. We define:
  \begin{itemize}
    \item $\clique{r}{n}$ to be the \emph{$r$-subdivided clique} of order $n$, which is obtained by subdividing each edge of the clique  $K_n$ exactly $r$ many times. We refer to the original vertices of the $K_n$ as \emph{principal} vertices.
    \item $\web{r}{n}$ to be the \emph{$r$-web} of order $n$, which is obtained from $\clique{r}{n}$ by turning the neighborhood of every principal vertex into a clique.
  \end{itemize}
\end{definition}


\newcommand{\fatredge}[2]{%
  \draw[fatedge, line width=\fatedgewidth, line cap=round]
    (M#1#22.center) -- (M#1#23.center) -- (M#1#24.center) -- (M#1#25.center);%
  \fill[fatedge] (M#1#22) circle (\fatedgerad) (M#1#23) circle (\fatedgerad)
                 (M#1#24) circle (\fatedgerad) (M#1#25) circle (\fatedgerad);}
\newcommand{\fatarm}[3]{%
  \draw[fatvertex, line width=\fatvertexwidth, line cap=round]
    (#1.center) -- (#2.center) -- (#3.center);%
  \fill[fatvertex] (#2) circle (\fatvertexrad) (#3) circle (\fatvertexrad);}

\newcommand{\trianglepathred}{%
  \begin{scope}[transparency group, opacity=\fatedgeopacity]
    \fatredge{1}{2}\fatredge{1}{3}\fatredge{2}{3}%
    \fatredge{1}{5}\fatredge{4}{5}\fatredge{2}{5}%
  \end{scope}}
\newcommand{\trianglepathblue}[1]{%
  \begin{scope}[transparency group, opacity=\fatvertexopacity]
    \fill[fatvertex] (N1) circle (\fatvertexrad) (N2) circle (\fatvertexrad)
                     (N3) circle (\fatvertexrad) (N4) circle (\fatvertexrad)
                     (N5) circle (\fatvertexrad);%
    \fatarm{N1}{M121}{M122}\fatarm{N1}{M131}{M132}\fatarm{N1}{M151}{M152}%
    \fatarm{N2}{M126}{M125}\fatarm{N2}{M231}{M232}\fatarm{N2}{M251}{M252}%
    \fatarm{N3}{M136}{M135}\fatarm{N3}{M236}{M235}%
    \fatarm{N4}{M451}{M452}%
    \fatarm{N5}{M156}{M155}\fatarm{N5}{M456}{M455}\fatarm{N5}{M256}{M255}%
    #1%
  \end{scope}}
\newcommand{\trianglepathwebcliques}{%
  \fill[fatvertex]
    (N1.center) -- (M121.center) -- (M131.center) -- (M151.center) -- cycle  
    (N2.center) -- (M231.center) -- (M251.center) -- (M126.center) -- cycle  
    (N3.center) -- (M136.center) -- (M236.center) -- cycle                   
    (N5.center) -- (M156.center) -- (M256.center) -- (M456.center) -- cycle; 
}

\newcommand{\minorgraph}{%
  \begin{tikzpicture}[scale=.8]
    \foreach \i in {1,...,5} { \coordinate (P\i) at ({360/5 * (\i - 1)}:1); }
    \draw[fatedge, line width=2.5pt, line cap=round]
                 (P1) -- (P2) (P2) -- (P3) (P3) -- (P1)  
                 (P1) -- (P5) (P5) -- (P4)               
                 (P2) -- (P5);                           
    \foreach \i in {1,...,5} {
      \node[circle, draw, fill=fatvertex, minimum size=9pt, inner sep=0] at (P\i) {};
    }
  \end{tikzpicture}%
}

\begin{figure}[h!]
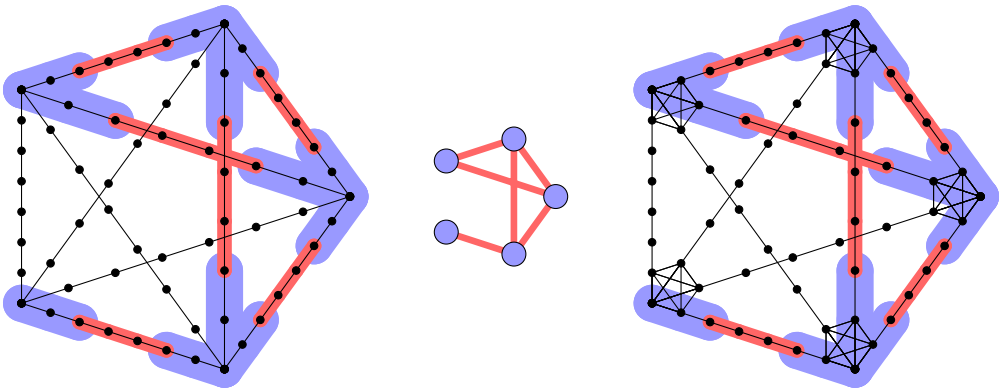

  \centering\small
  $\vcenter{\hbox{\subclique[1.2]{6}{5}{\trianglepathblue{}\trianglepathred}}}$
  \qquad
  $\vcenter{\hbox{\minorgraph}}$
  \qquad
  $\vcenter{\hbox{\spiderweb[1.2]{6}{5}{\trianglepathblue{\trianglepathwebcliques}\trianglepathred}}}$
  \caption{Left: $\clique{6}{5}$. Right:  $\web{6}{5}$.  Middle: the graph $H$ highlighted in both $\clique{6}{5}$ and $\web{6}{5}$ as a $2$-shallow $2$-fat minor. Red branch sets model the edges of $H$ (the four middle subdivision vertices of an edge), blue branch sets model the vertices of $H$ (stars centered at the principal vertices); branch sets of an incident edge and vertex share a single vertex.}\label{fig:patterns}
\end{figure}

\section{ISMF implies drill-flatness}\label{sec:indminor-flatness}

In this section, we will prove \cref{thm:indminoruqw}, which states that every ISMF graph class is drill-flat and gives precise witnessing bounds.
Similarly as in the case of flatness in nowhere dense graphs~\cite{NesetrilM10}, we prove \cref{thm:indminoruqw}
in $d$ steps, at each step increasing scatteredness by $2$. 
The main technical engine --- a single increase step --- is described in \cref{ss:indminor:inc-d}. This step is analogous to the metric one described in \cref{sec:improve}.
\cref{thm:indminoruqw} is restated bellow for convenience.
\inducedflatness*

\subsection{Improving the scatteredness}\label{ss:indminor:inc-d}

The following lemma encapsulates a single step of the argument.
The proof follows the general outline of the proof of \cref{lem:unit-step}, with a number of technical details
adjusted to the induced minor setting. Since excluding an induced minor is a hereditary property (contrary to excluding
a fat minor), one aspect will be significantly simpler: we do not need to care about a set $Z$ as in the statement
of \cref{lem:unit-step}, as here we will be able to just delete it and continue with an induced subgraph. 

In the proofs, we will be slightly more fine-grained in keeping track of the shallowness of the constructed induced minors. Let $H'$ be a $1$-subdivision of a graph $H$.
Then, we say that $H'$ is an $(a,b)$-shallow induced minor of a graph $G$
if there exists an induced minor model of $H'$ in $G$, where every
branch set corresponding to a vertex of $H$ induces a subgraph of radius at most $a$
while every branch set corresponding to a subdivision vertex of $H'$ (i.e., a vertex
of $V(H') \setminus V(H)$) induces a subgraph of radius at most $b$.
    
\begin{lemma}\label{lem:indminoruqw:unit-step}
    Let $h,d$ be positive integers with $d \geq 2$, let $G$ be a graph, and $A \subseteq V(G)$ be a $2d$-scattered set in $G$
    of size at least $4^5(h+1)^{10}$.
    Then, one of the following holds:
    \begin{itemize}
    \item There exists $v \in V(G)$ such that $|A \cap \Ball_G(v, d + 6)| \geq |A|^{\frac{1}{10}}$.
    \item There exists $B \subseteq A$ of size at least $|A|^{\frac{1}{10}}$ that is $2(d+1)$-scattered in $G$.
    \item There exists a $(3d,2)$-shallow induced model of the $1$-subdivision of $K_h$ in $G$.
    \end{itemize}
\end{lemma}
\begin{proof}    
    We assume that the first two outcomes do not happen and we produce the induced minor model as in the last outcome.

    Construct an auxiliary graph $H$ with $V(H) = A$, where two vertices $x,y\in A$ are adjacent if and only if $\dist(x,y) \leq 2d+2$. 
    We claim that 
    \begin{equation}\label{eq:indminor:EHlarge} |E(H)| > \tfrac{1}{16}|A|^{\frac{19}{10}}. \end{equation}
    Indeed, otherwise the sum of the degrees in $H$ would be at most $\tfrac{1}{8}|A|^{\tfrac{19}{10}}$, implying that
    at least $|A|/2$ vertices of $H$ have degree at most $\frac{1}{4}|A|^{\tfrac{9}{10}}$.
    Note that $\frac{1}{4}|A|^{\tfrac{9}{10}}\leq \frac{1}{2}|A|^{\tfrac{9}{10}} - 1$, as $|A| \geq 4^5(h+1)^{10} \geq 2^{20}$. Consequently, $H$ contains an independent set of size at least
    \[ \frac{|A|/2}{\frac{1}{2}|A|^{\frac{9}{10}}-1+1} = |A|^{\frac{1}{10}}. \]
    Such an independent set is the second outcome of the lemma.  This proves~\eqref{eq:indminor:EHlarge}.

    Let $e = xy \in E(H)$. Fix a shortest path $P_e$ between $x$ and $y$.
    Define $M_e \subseteq V(P_e)$ as follows. 
    Recall that $A$ is $2d$-scattered; hence $\dist(x,y) \in \{2d+1,2d+2\}$. 
    If $\dist(x,y) = 2d+1$, let $M_e$ be the middle four vertices of $P_e$, and if $\dist(x,y) = 2d+2$,
    let $M_e$ be the middle five vertices of $P_e$. Recall that $d \geq 2$, so the endpoints of $P_e$ are never in $M_e$.
    Furthermore, let $m_e \in M_e$ be one of the two middle vertices of $P_e$ if $\dist(x,y) = 2d+1$ and the midpoint of $P_e$ if $\dist(x,y) = 2d+2$.
    Construct now an auxiliary graph $K$ with vertex set $V(K) = E(H)$, where any $e,f\in V(K)=E(H)$ are adjacent if and only if
     $M_e \cap M_f \neq \emptyset$ or there is an edge in $G$ with one endpoint in $M_e$ and the other in $M_f$.
     
    We claim that the maximum degree in $K$ cannot be too large.
    Fix any $e \in E(H)$ and let $B \subseteq A$ be defined as follows:
    \[
        B \coloneqq \bigcup_{f \in N_K[e]} \{v \in A~\mid~ \text{$v$ is an endpoint of $f$}\}.
    \]
    That is, $B$ consists of the endpoints of those edges $f$ of $H$ that satisfy $f \in N_K[e]$.
    Consider any such $f \in N_K[e]$ 
    and observe that $M_f \subseteq \Ball(m_e, 7)$. (The worst case is when $|M_e| = |M_f| = 5$ and 
    only one endpoint of $M_e$ is adjacent to one endpoint of $M_f$.)
    Consequently, the endpoints of $f$ lie in $\Ball(m_e, d+6)$.
    Thus, unless the first outcome happens, we have
    $\left| B \right| \leq |A|^{\frac{1}{10}}$.
    Moreover,
    \[ |N_K[e]| \leq \binom{|B|}{2} \leq \binom{|A|^{\frac{1}{10}}}{2} \leq \tfrac{1}{2}|A|^{\frac{1}{5}}. \]
    This proves that the maximum degree in $K$ is at most $\tfrac{1}{2}|A|^{\frac{1}{5}}-1$. 

    Therefore, $K$ contains an independent set $J$ of size: 
    \[
    |J|\ge\frac{|E(H)|}{\frac{1}{2}|A|^{\frac{1}{5}}}\overset{(\ref{eq:indminor:EHlarge} )}{>}\frac{\tfrac{1}{16}|A|^{\frac{19}{10}}}{\frac{1}{2}|A|^{\frac{1}{5}}}>\tfrac{1}{8}|A|^{\frac{17}{10}}.
    \]
    
    Since $|A| \geq 4^5 (h+1)^{10}$, we have 
    \[ |J|>\tfrac{1}{8}|A|^{\frac{17}{10}}=\tfrac{1}{8}|A|^{\frac{1}{5}}\cdot|A|^{\frac{3}{2}}\ge\tfrac{1}{8}\cdot4\cdot(h+1)^{2}\cdot|A|^{\frac{3}{2}}>\binom{h+1}{2}|A|^{\frac{3}{2}}. \]
    
    Let $H'$ be the subgraph of $H$ with vertex set $A$ and edge set $J$. 
    
    We claim that $G$ contains a $1$-subdivision of $H'$ as a $(d-2,2)$-shallow induced minor. 
    For $v \in A$, let $\phi(v) \coloneqq \Ball(v, d-2)$ and for $e \in J$ let $\phi(e) \coloneqq M_e$. (Recall again $d \geq 2$.)
    We claim that $\phi$ is a $(d-2,2)$-shallow induced model of the $1$-subdivision of $H'$. Indeed, the branch sets of $\phi$ induced connected subgraphs
    and the bounds on their radii are immediate from the construction.
    Furthermore, for every $e = xy \in J$, we clearly have that $\phi(e)$ is adjacent to $\phi(x)$ and $\phi(y)$, and $\phi(e)$ is disjoint from $\phi(x)\cup \phi(y)$ because $P_e$ was chosen to be a shortest path.
    It remains to check vertex-disjointness and anti-adjacency of the other pairs of branch sets of $\phi$.

    Since $A$ is $2d$-scattered, $\phi(v)$ and $\phi(u)$ are disjoint and non-adjacent for every two distinct $v,u \in A$.
    Since $J$ is independent in $K$, $\phi(e)$ and $\phi(f)$ are disjoint and non-adjacent for any two distinct $e,f \in J$.
    Finally, consider $e = xy \in J$ and $z \in A \setminus \{x,y\}$. Then, $\phi(e) = M_e \subseteq \Ball(x, d+1) \cup \Ball(y,d+1)$.
    Since $A$ is $2d$-scattered and $\phi(z) = \Ball(z,d-2)$, $\phi(z)$ and $\phi(e)$ must be disjoint and non-adjacent as otherwise, either $\dist(z,x)\leq 2d$ or $\dist(z,y)\leq 2d$, contradicting that $A$ is $2d$-scattered.
    This finishes the proof that $\phi$ is a $(d-2,2)$-shallow induced model of the $1$-subdivision of $H'$. 

    Since $|J| > \binom{h+1}{2} |A|^{\frac{3}{2}}$, \cref{thm:alon-density} implies that $H'$ contains $K_h$ as $1$-shallow minor. 
    Hence, analogously to \cref{lem:minor-of-fat-minor}, $G$ contains the $1$-subdivision of $K_h$ as $(3d,2)$-shallow induced minor. This finishes the proof.
\end{proof}

The next lemma is an analog of \cref{lem:small-step}.
\begin{lemma}\label{lem:indminor:small-step}
    Let $h,d,m$ be positive integers with $d \geq 2$, $G$ be a graph that does not contain the $1$-subdivision of $K_h$ as a $(3d,6)$-shallow induced minor,
    and $A \subseteq V(G)$ be a $2d$-scattered set in $G$ of size at least 
    \[ |A| \geq \left( \max\left( 4^5(h+1)^{10}, m + h - 1\right)\right)^{10^h}. \]
    Then, there exists a set $X \subseteq V(G)$ of size at most $h$ and a set $B \subseteq A$ of size at least $m$
    that is $(2(d+1),\Ball(X,37))$-scattered in $G$.
\end{lemma}
\begin{proof}
We can assume $m \geq 4^5 (h+1)^{10} - h + 1$, as otherwise we can just reset $m \coloneqq 4^5 (h+1)^{10} -h + 1$.

Consider the following procedure. We will compute sets $A_0 \supseteq A_1 \supseteq \ldots$ and vertices $v_1,v_2,\ldots$
while maintaining the following invariants for every $0 \leq i < h$:
\begin{itemize}
    \item $A_i$ is large:
    \[ |A_i| \geq \left(m+h-1\right)^{10^{h-i}}. \]
    \item For every $j\in [i]$, we have
     \[ A_i \subseteq \Ball(v_j, d+6) \cup \Ball(\{v_1,\ldots,v_i\}, 37). \]
    \item The set $\{v_1,\ldots,v_i\}$ is $37$-scattered in $G$. 
\end{itemize}

We initiate $A_0 \coloneq A$. Clearly, the invariants above are satisfied for $i=0$.

Let $0 \leq i < h$ and assume that $A_i$ and $v_1,\ldots,v_i$ have already been defined.
Let \[B_i \coloneqq A_i \cap \Ball(\{v_1,\ldots,v_i\}, 37)\qquad \textrm{and} \qquad A_i' = A_i \setminus B_i.\]
If $|B_i| \geq m$, we return $B \coloneqq B_i$ and $X \coloneqq \{v_1,\ldots,v_i\}$.
Otherwise, as $|A_i| \geq (m+h-1)^{10^{h-i}} \geq m^{10}$, we have $|A_i'| \geq m^{10}-m \geq 4^5 (h+1)^{10}$. 
Let \[G_i \coloneqq G \setminus \Ball(\{v_1,\ldots,v_i\}, 37).\] Clearly, $A_i'$ is $2d$-scattered in $G_i$.

We apply \cref{lem:indminoruqw:unit-step} to $h$, $d$, $G_i$, and $A_i$.
The third outcome cannot happen, because $G$ is assumed not to contain such a shallow induced minor. 
In case the second outcome happens with $B \subseteq A_i'$ of size at least $|A_i'|^{\frac{1}{10}}$, we return $B \cup B_i$ and $X \coloneqq \{v_1,v_2,\ldots,v_i\}$. 
Note that this is a valid return as
\begin{equation}\label{eq:BiAi}
    |B_i\cup B|\geq |B_i| + |A_i'|^{\frac{1}{10}} = |B_i| + \left(|A_i| - |B_i|\right)^{\frac{1}{10}} \geq |A_i|^{\frac{1}{10}} \geq \left(m+h-1\right)^{10^{h-i-1}} \geq m.
\end{equation}
In case of the first outcome happens with a vertex $v$, we set $v_{i+1} \coloneqq v$ and 
\[ A_{i+1} \coloneqq B_i \cup \left(A_i' \cap \Ball_{G_i}(v, d + 6)\right). \]
The first invariant follows from the same calculation as in~\eqref{eq:BiAi}. The second is immediate due to the definition of $B_i$ and the
third follows from $v \in V(G_i)$. 

Assume that the procedure executed all the steps without termination, which resulted in finding $A_h$ and $v_h$. 
Again, let $B_h \coloneqq A_h \cap \Ball(\{v_1,\ldots,v_h\}, 37)$ and $A_h' \coloneqq A_h \setminus B_h$;
and if $|B_h| \geq m$, then $B \coloneqq B_h$ and $X \coloneqq \{v_1,\ldots,v_h\}$ constitute a valid output


Otherwise, we have $|A_h'| \geq h$. In particular, $A_h'$ is nonempty and by the second invariant we have that every vertex of $A_h'$ is simultaneously in $\bigcap_{j\in [h]}\Ball(v_j,d+6)$ and outside of $\Ball(\{v_1,\ldots,v_h\},37)$. So we must have $d+6>37$, or equivalently, $d>31$.

Let $H$ be an auxiliary biclique with sides $A_h'$ and $\{v_1,\ldots,v_h\}$.
We will construct a $(d-14,6)$-shallow 
induced model of the $1$-subdivision of $H$ in $G$. Note that then, by selecting a matching of size $h$ in $H$ and merging, for every edge $f$ of the matching, the branch sets of the endpoints of $f$ and the branch set of the vertex subdividing $f$, we may conclude that $G$ contains a $(2(d-14)+6+1=2d-21,6)$-shallow induced model of the $1$-subdivision of $K_h$; this will be a contradiction.

We define
\begin{align*}
    \phi(x) & \coloneqq  \Ball(x, d-13)&&\qquad \textrm{for }x \in A_h';\textrm{ and}\\
    \phi(v_i) & \coloneqq \Ball(v_i, 6) &&\qquad \textrm{for }i\in [h].
\end{align*}
Further, for $e = xv_i \in E(H)$ with  $x \in A_h'$ and $i\in [h]$, we fix a shortest path $P_e$
between $x$ and $v_i$ and let $\phi(e)\coloneqq V(P_e)\setminus (\phi(x)\cup \phi(v_i))$; note that thus, $\phi(e)$ is the vertex set of an infix of $P_e$. 
We claim that $\phi$ is the desired model of the $1$-subdivision of $H$ (where every edge $e$ of $H$ is implicitly identified with the vertex subdividing it).

Recall that $A_h' \subseteq A$ is $2d$-scattered in $G$, while $\dist(v_i, x) \leq d+6$ for every $i\in [h]$ and $x \in A_h'$.
Hence also $\dist(v_i,x) > d-6$. 
This implies that for every $x \in A_h'$ and $i\in [h]$, $\phi(x)$ and $\phi(v_i)$ are disjoint and non-adjacent.
In particular, $\phi(xv_i)$ is well-defined and nonempty, and it induces a path of length between $0$ and $11$ (so of radius at most $6$). 
Note that thus we have $\phi(xv_i) \subseteq \Ball(x, d-1) \cap \Ball(v_i, 18)$.

Let $x \in A_h'$ and $i\in [h]$. 
Since $A_h'$ is $2d$-scattered, for every $y \in A_h'$, $x \neq y$, and $j\in [h]$
(possibly $i=j$), the sets $\phi(x) \cup \phi(xv_i)$ and $\phi(y) \cup \phi(yv_j)$ 
are disjoint and anti-adjacent, as they are contained in $\Ball(x,d-1)$ and $\Ball(y,d-1)$, respectively.
Similarly, since $\{v_1,\ldots,v_h\}$ is 37-scattered, 
for every $y \in A_h'$ (possibly $x=y$), and $1 \leq j \leq h$, $i \neq j$, 
the sets $\phi(v_i) \cup \phi(xv_i)$ and $\phi(v_j) \cup \phi(yv_j)$ are disjoint and non-adjacent,
as they are contained in $\Ball(v_i, 18)$ and $\Ball(v_j, 18)$, respectively.
As $d-13\geq 6$, this finishes the proof that $\phi$ is a $(d-12, 6)$-shallow induced minor model of the $1$-subdivision of $H$ in $G$.
\end{proof}

\subsection{Iteration}

We are ready to prove the main result of this section, which also proves the \ref{c:ismf}$\,\Rightarrow$\,\ref{c:df} implication of \cref{thm:equivalences}.

\inducedflatness*

\begin{proof}
    We prove the statement for
    \[ N(m)\coloneqq m\cdot \left(\max\left(4^5(h+1)^{10}, m+h-1\right)\right)^{10^{h\cdot \max(d-18,1)}}\in m^{2^{\Oh(hd)}}. \]
    We can assume $m \geq 4^5 (h+1)^{10} - h + 1$, as otherwise we can just reset $m \coloneqq 4^5 (h+1)^{10} -h + 1$.

    So suppose $A\subseteq V(G)$ is such that $|A|\geq N(m)$.
    Suppose there exists $v \in V(G)$ satisfying $|A \cap \Ball(v,37)| \geq m$.
    Then, we may return $B \coloneqq A \cap \Ball(v,37)$ and $X \coloneqq \{v\}$. 
    Otherwise, by a greedy argument, there exists $A_{18} \subseteq A$ that is $37$-scattered and is of size at least
    \[ \frac{|A|}{m} \geq \left(m+h-1\right)^{10^{h \max(d-18,1)}}. \]
    If $d \leq 18$, then we are done with $B \coloneqq A_{18}$ and $X \coloneqq \emptyset$, so assume otherwise. 
    We will construct sets $A_{18} \supseteq A_{19} \supseteq \ldots \supseteq A_d$ and $X_{18} \subseteq X_{19} \subseteq \ldots\subseteq X_d$
    so that for every $18 \leq i \leq d$, 
    \begin{itemize}
        \item $A_i$ is $(2i,\Ball(X_i,37))$-scattered in $G$,
        \item $|X_i| \leq h(i-18)$, and
        \item $|A_i| \geq \left(m+h-1\right)^{10^{h(d+1-i)}}.$
    \end{itemize}
    These invariants are satisfied for $A_{18}$ and $X_{18} \coloneqq \emptyset$.

    Assume then that for some $18 \leq i < d$ we have constructed $A_i$ and $X_i$. 
    We set $B_i \coloneqq A_i \cap \Ball(X_i,37)$ and $A_i' \coloneqq A_i \setminus B_i$.
    If $|B_i| \geq m$, then we may return $B \coloneqq B_i$ and $X \coloneqq X_i$, so assume otherwise.
    Let 
    \[ m' \coloneqq \left(m+h-1\right)^{10^{h(d-i)}} - |B_i|. \]
    As $i < d$, we have $m' \geq m + h - 1$. Furthermore, by the lower bound on the size of $A_i$, we have 
    \[ |A_i'| \geq (m')^{10^h}. \]
    Let $G_i \coloneqq G \setminus \Ball(X_i,37)$.
    We apply \cref{lem:indminor:small-step} to $h$, $d' \coloneqq i$, $m'$, and $A_i'$,
    obtaining sets $B_i'$ and $X_i'$ such that $|B_i'| \geq m'$, $|X_i'| \leq h$,
    and $B_i'$ is $(2(i+1),\Ball_{G_i}(X_i', 37))$-scattered in $G_i$. 

    Let $A_{i+1} \coloneqq B_i \cup B_i'$ and $X_{i+1} \coloneqq X_i \cup X_i'$. 
    Clearly, $A_{i+1}$ is $(2(i+1), \Ball(X_{i+1},37))$-scattered in $G$.
    Furthermore, 
    \[ |A_{i+1}| \geq m' + |B_i| =\left(m+h-1\right)^{10^{h(d-i)}}. \]
    This finishes the description of one step of the construction.

    Once we obtain $A_d$ and $X_d$, we may return $B \coloneqq A_d$ and $X \coloneqq X_d$. This finishes the proof.
\end{proof}

\section{Drill-flatness implies FSMF}

Next, we prove the \ref{c:df}\,$\Rightarrow$\,\ref{c:fsmf} implication of \cref{thm:equivalences}.

\begin{lemma}\label{lem:drill-flat-implies-fsmf}
  Every drill-flat graph class is fat-shallow-minor-free.
\end{lemma}
\begin{proof}
  We will prove the lemma by contraposition.
  Let $\CC$ be a graph class that is not fat-shallow-minor-free and assume towards a contradiction that $\CC$ is drill-flat for some drill radius $\beta \in \N$. 

  As $\CC$ is not fat-shallow-minor-free, we have that for every $\delta$ there exists some $\alpha \in \N$ such that $\CC$ contains every graph as an $\alpha$-shallow  $\delta$-fat minor.
  Fix such an $\alpha$ for $\delta \coloneq 2\beta$.

  Let $s \in \N$ and $N \colon \N \to \N$ be the budget and the overhead for which $\CC$ is $(4\alpha,\beta)$-drill-flat. Set $n\coloneq N(2s+2)$.

  Let $K = (V_K,E_K)$ be the clique on $n$ vertices and $G = (V_G,E_G)$ be a graph in $\CC$ that contains $K$ as an $\alpha$-shallow $\delta$-fat  minor, witnessed by a model $\phi \colon V_K \cup E_K \to 2^{V_G}$.
  Let $c \colon V_K \to V_G$ be an injective function that assigns to each vertex $p$ from $K$, a \EMPH{center} vertex $c(p) \in \phi(p)$ from $G$ that witnesses the $\alpha$-shallowness: this means $G[\phi(p)]$ contains a path of length at most $\alpha$ from each $v \in \phi(p)$ to $c(p)$.
  We apply $(4\alpha, \beta)$-drill-flatness to the set of all centers $c(V_K)$, which yields a subset $B \subseteq V_K$ of size at least $2s + 2$ and a set $X \subseteq V_G$ of size at most $s$ such that $c(B)$ is $(4\alpha,\Ball(X,\beta))$-scattered.
  
  Our goal will be to prove the following.

  \begin{claim}\label{clm:twogood}
  There exist two distinct vertices $p,q \in B$ of $K$ satisfying
  \[
    \Ball(X,\beta) \cap \big( \phi(p) \cup \phi(pq) \cup \phi(q) \big) = \emptyset.
  \]  
  \end{claim}
  
  From the claim, the desired contradiction is derived as follows.

  \begin{itemize}
    \item On the one hand, $\Ball(X,\beta)$ intersects every path of length at most $4\alpha$ path between $c(p)$ and $c(q)$, as they are both part of the scattered set $c(B)$.
    \item On the other hand, $\Ball(X,\beta)$ does not intersect the branch sets $\phi(p)$, $\phi(pq)$, and $\phi(q)$, which together contain a path of length at most $4\alpha$ between $c(p)$ and $c(q)$, by $\alpha$-shallowness.
  \end{itemize}

  Towards proving the claim, choose $f: X \to V_K \cup E_K$ to be any function satisfying
  \[
    f(x) \in \argmin_{o \in V_K \cup E_K} \left[\dist \big(x,\phi(o)\big)\right].
  \]
  This means $f$ maps each vertex $x \in X$ to the vertex or edge of $K$ whose branch set is the closest to $x$ (with ties broken arbitrarily).
  We say a vertex $x \in X$ \EMPH{blocks} a vertex $p \in V_K$ if $f(x)$ is either 
  \begin{itemize}
    \item a vertex of $K$ and equal to $p$, or
    \item an edge of $K$ with endpoint $p$.
  \end{itemize}
  The following claim implies \Cref{clm:twogood}.

  \begin{claim}\label{clm:onegood}
    For all $p,q \in V_K$ and $x \in X$ such that neither $p$ nor $q$ is blocked by $x$:
    \[
      \Ball(x,\beta) \cap \big( \phi(p) \cup \phi(pq) \cup \phi(q) \big) = \emptyset.  
    \]
  \end{claim}

  In particular, \Cref{clm:onegood} allows us to choose the desired elements $p,q$ for \Cref{clm:twogood} as two non-blocked elements from $B$.
  This is possible, since each vertex from $X$ blocks at most two vertices from~$V_K$, so $B$ contains at least
  \[
    |B| - 2|X| \geq (2s+2) - 2s \geq 2
  \]
  non-blocked vertices.
  Hence, in order to prove the lemma, it only remains to prove \Cref{clm:onegood}.
  
  Fix $p,q,x$ accordingly and assume towards a contradiction that the conclusion of the claim fails.
  Then there is $o \in \{p,pq,q\}$ such that $\Ball(x,\beta)$ intersects $\phi(o)$.
  This means $x$ and $\phi(o)$ are at a distance at most $\beta$.
  By the choice of $f$, this upper bounds the distance between $x$ and $\phi(f(x))$ by~$\beta$, too.
  By the triangle inequality, $\phi(o)$ and $\phi(f(x))$ are distance at most $2\beta$.
  
  On the other hand, since $p$ and $q$ are not blocked by $x$, $p$ and $q$ are neither endpoints of $f(x)$ (if $f(x)$ is an edge) nor equal to $f(x)$ (if $f(x)$ is a vertex).
  In particular, $o$ is distinct from $f(x)$, and neither is an endpoint of the other. Thus, $\phi(o)$ and $\phi(f(x))$ are at a distance larger than $2\beta = \delta$ by $\delta$-fatness; a contradiction which proves the claim and therefore the lemma.
\end{proof}

\section{Pattern-free implies ISMF}
The goal of this section is to prove the following lemma.

\begin{lemma}\label{thm:patterns}
  Let $\CC$ be a graph class such that for every $r \geq 3$, there is $k \in \N$ such that $\CC$ excludes $\clique{r}{k}$ and $\web{r}{k}$ as induced subgraphs.
  Then $\CC$ is ISMF.
\end{lemma}

We will prove it by contraposition and show that if $\CC$ is not ISMF, i.e., $\CC$ contains all graphs as induced $d$-shallow minors for some fixed $d \in \N$, then $\CC$ also contains arbitrarily large $r$-subdivided cliques and $r$-webs for some fixed $r \geq 3$.
Hence, proving \Cref{thm:patterns} reduces to proving the following lemma.

\begin{lemma}\label{lem:extract-patterns}
  There is a function $R: \N \to \N$, such that for all $d,m \in \N$ there is a graph $H$ with the following property.
  If a graph $G$ contains $H$ as a $d$-shallow induced minor, then there is $3 \leq r \leq R(d)$ such that $G$ contains 
  $\clique{r}{m}$ or $\web{r}{m}$ as an induced subgraph.
\end{lemma}

The key idea to exhibit large subdivided cliques and webs inside shallow induced minor models is encapsulated by the following definition and lemma.

\begin{definition}[$\ell$-fan]\label{def:fan}
  Let \(G\) be a graph, let \(c\in V(G)\), let \(X\subseteq V(G)\), and let
  \(\ell\ge 1\) be an integer. An \emph{\(\ell\)-fan} with \emph{center} \(c\) and
  \emph{leaf set} \(X\) is a family of induced length-$\ell$ $c$-$x$ paths 
  \[
      P_x \quad = \quad c p_x^1 p_x^2 \cdots p_x^\ell=x
  \]
  for each $x\in X$, that are pairwise vertex-disjoint except for their
  common endpoint \(c\), and together satisfy one of the following two properties:
  \begin{enumerate}
    \item There are no edges between \(P_x-\{c\}\)
    and \(P_y-\{c\}\) for distinct \(x,y\in X\).
    \item The only edges between \(P_x-\{c\}\) and
    \(P_y-\{c\}\), for distinct \(x,y\in X\), are the edges \(p_x^1p_y^1\). In
    particular, the vertices \(p_x^1\), \(x\in X\), then form a clique.
  \end{enumerate}
    In the first case we call the fan \emph{independent}, in the second we call it \emph{linked}.
\end{definition}

The following lemma is a generalization of
\cite[Lemma 2.2]{SCOTT2020487}.
The latter additionally assumes bounded clique number, and thus always yields an independent fan.

\begin{lemma}[Fan lemma]\label{lem:fan-ramsey}
    For every $r\in \N$ there exists a function $F : \N \to \N$ with the following property.
  Let \(G\) be a graph, \(o\in V(G)\), $m\in\N$, and \(X\subseteq V(G)\) be a set
  of at least \(F(m)\) vertices, each at distance at most \(r\) from \(o\).
  Then are \(X'\subseteq X\) of size \(m\) and $1\leq \ell \leq r$,
  such that $G$ contains an $\ell$-fan with leaf set $X'$.
\end{lemma}

\begin{proof}
  We prove the lemma by induction on \(r\). The case \(r=0\) is trivial. Assume
  \(r\ge 1\), and choose for every \(x\in X\) a shortest \(o\)-\(x\) path in a
  fixed breadth-first search tree rooted at \(o\).

  For every child \(s\) of \(o\) in this tree, let \(X_s\) be the set of terminals
  \(x\in X\) whose chosen path starts with \(s\). If some \(X_s\) is sufficiently
  large, then all vertices of \(X_s\) are at distance at most \(r-1\) from \(s\),
  and the result follows from the induction hypothesis applied with root \(s\).

  Thus, by taking \(F(m)\) large enough, we may assume that there are many
  children \(s\) with \(X_s\ne\emptyset\). Choose one terminal from each of those many different
  sets \(X_s\). After a pigeonhole argument, we may assume that the chosen
  root-to-terminal paths all have the same length \(\lambda\le r\). These tree paths
  meet only in \(o\).

  Write these paths as
  \[
      Q_i \quad = \quad o q_i^1 q_i^2\cdots q_i^\lambda=x_i .
  \]
  For two ordered paths \(Q_i,Q_j\), \(i<j\), color the pair \((i,j)\) by the set
  \[
      C_{ij}:=\{(a,b)\in[\lambda]^2 : q_i^a q_j^b\in E(G)\}.
  \]
  There are only finitely many colors, so Ramsey's theorem lets us pass to a subset of many paths
  on which this color is constant; call the common set \(C\).

  If \(C=\emptyset\), then the paths \(Q_i\) themselves form an independent fan
  with center \(c\coloneq o\).
  Otherwise, \(C\ne\emptyset\) and
  \[
      h \coloneq \max\{a : (a,b)\in C \text{ or } (b,a)\in C \text{ for some } b\}.
  \]
  is well defined.
  Intuitively, $h\in[\lambda]$ is the rightmost position on the Ramseyed paths that is incident to ``cross-path'' edges.
  In particular, there exists $a \in [\lambda]$ such that either \((a,h)\in C\) or \((h,a)\in C\). 
  In the first case, take the first path \(Q_0\) as a root path set
  put \(c \coloneq q_0^a\). For every other chosen path \(Q_i\), use the path
  \[
      c\, q_i^h q_i^{h+1}\cdots q_i^\lambda .
  \]
  Otherwise, \((h,a)\in C\) and we take the last chosen path as the root path and argue symmetrically.

  By maximality of \(h\), there are no edges between the suffixes
  \(q_i^h q_i^{h+1}\cdots q_i^\lambda\) and \(q_j^h q_j^{h+1}\cdots q_j^\lambda\),
  except possibly the edges \(q_i^h q_j^h\). If these edges are absent, we get an
  independent fan. If they are present, we get a linked fan. This
  proves the lemma.
\end{proof}

Our initial idea towards proving \Cref{lem:extract-patterns} was to extract subdivided cliques and webs from induced shallow minor models of subdivided cliques, by applying the fan lemma to the branchsets of the principal vertices together with a set $X$ of outgoing edges towards the subdivision vertices.
This idea does not work in this naive way, because the fan lemma restricts the number of subdivided connections to a smaller set $X'$. 
When applying it to multiple branchsets, we get no guarantees that the resulting $X'$ sets will be consistent, i.e., be able to form the edgeset of a complete graph.
We will circumvent this issue by working in a bipartite setting.

\begin{definition}[Bipartite patterns]
  Fix $r \geq 3$ and $n,m\in\N$. We define:
  \begin{itemize}
    \item $\biclique{r}{n}{m}$ to be the \emph{$r$-subdivided biclique} of order $n$, which is obtained by subdividing each edge of the biclique $K_{n,m}$ exactly $r$ many times.
    \item $\biweb{r}{n}{m}$ to be the \emph{$r$-biweb} of order $n$, which is obtained from $\biclique{r}{n}{m}$ by turning the neighborhoods of the principal vertices on both sides of the bipartition into cliques.
    \item $\mixweb{r}{n}{m}$ to be the \emph{$r$-mixweb} of order $n$, which is obtained from $\biclique{r}{n}{m}$ by only turning the neighborhoods of the principal vertices on the size $n$ side of the bipartition into cliques.
  \end{itemize}
\end{definition}

See \Cref{fig:bipartite-patterns} for an illustration of the bipartite patterns.

\begin{figure}[htbp]
  \centering
  \includegraphics[scale = 0.7]{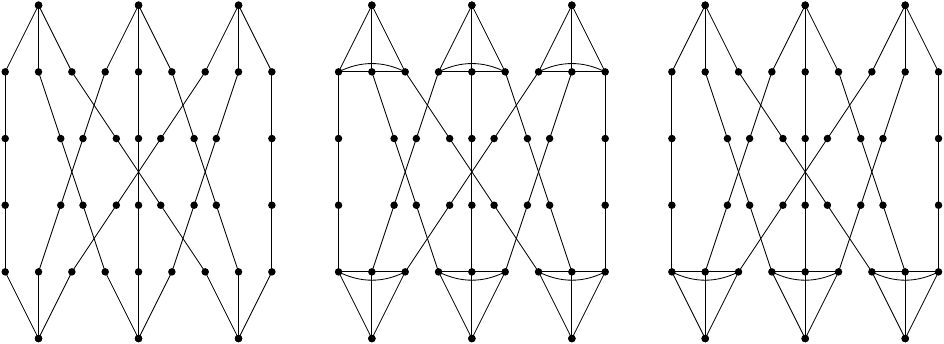}
  \caption{$\biclique{4}{3}{3}$, $\biweb{4}{3}{3}$, and $\mixweb{4}{3}{3}$.}
  \label{fig:bipartite-patterns}
\end{figure}

Observe that asymptotically, it does not matter whether we exhibit patterns or bipartite patterns.

\begin{observation}\label{obs:bipartize}
  Fix $r \geq 3$ and $n,m\in\N$.

  \smallskip\noindent
  Patterns contain bipartite patterns:
  \begin{enumerate}
    \item $\clique{r}{n+m}$ contains $\biclique{r}{n}{m}$ as an induced subgraph.
    \item $\web{r}{n+m}$ contains $\biweb{r}{n}{m}$ as an induced subgraph.
  \end{enumerate}

  \smallskip\noindent
  Bipartite patterns contain patterns:
  \begin{enumerate}
    \setcounter{enumi}{2}
    \item $\biclique{r}{n}{\binom{n}{2}}$ contains $\clique{2r+1}{n}$ as an induced subgraph.
    \item $\biweb{r}{n}{\binom{n}{2}}$ contains $\web{2r}{n}$ as an induced subgraph.
    \item $\mixweb{r}{n}{\binom{n}{2}}$ contains $\clique{2r}{n}$ as an induced subgraph.
  \end{enumerate}
  
\end{observation}

By the above observation, proving \Cref{lem:extract-patterns} is reduced to proving the following lemma.

\begin{lemma}\label{lem:extract-bipatterns}
  For every $d,m\in\N$, there is $M \in \N$ with the following property.
  If a graph $G$ contains $\biclique{3}{M}{M}$ as an $d$-shallow induced minor, then there is $3 \leq r \leq 11d$ such that $G$ contains
  $\biclique{r}{m}{m}$, $\biweb{r}{m}{m}$, or $\mixweb{r}{m}{m}$
  as an induced subgraph.
\end{lemma}
\begin{proof}
  We can assume $G$ contains the $3$-subdivided biclique $H \coloneq \biclique{3}{m_A}{m_B}$ as an induced $d$-shallow minor for $m_A$ and $m_B$ arbitrarily large in comparison to $m$.
  We use $a_i$ and $b_j$ to denote its principal vertices and $s_{ij}^t$ to denote its subdivision vertices for $i \in [m_A]$, $j \in [m_B]$, $t\in [3]$, i.e., $a_i s_{ij}^1 s_{ij}^2 s_{ij}^3 b_j$ is a path.
  Let $\phi$ be the witnessing model, assigning to each vertex of $H$ its branch set in $G$, and denote

  \[
    A_i \coloneq \phi(a_i), \quad
    B_j \coloneq \phi(b_j), \quad
    S^t_{ij} \coloneq \phi(s_{ij}^t),\quad
    S_{ij} \coloneq \bigcup_{t\in [3]} S^t_{ij}.
  \]

  For each pair $i,j$ we fix a path $P_{ij}$ with minimal length, such that one endpoint lies in $A_i$, the other in $B_j$, and all the inner vertices lie in $S_{ij}$.
  By minimality, inducedness of $\phi$, and the structure of $H$, one can verify that each $P_{ij}$ is an induced path in $G$ that can be written as a concatenation
  \[
    P_{ij} = a_{ij} x_{ij} R_{ij} y_{ij} b_{ij}
  \]
  satisfying
  \[
    a_{ij} \in A_i, \quad
    x_{ij} \in S^1_{ij}, \quad
    R_{ij} \subseteq S_{ij}, \quad
    y_{ij} \in S^3_{ij}, \quad
    b_{ij} \in B_j.
  \]
  In particular, each $P_{ij}$ contains at least $5$ vertices. 
  Additionally by $d$-shallowness and minimality, $P_{ij}$  contains at most $1 + 3(2d+1) + 1 = 6d+5$ vertices.

  Since only a constant number of path lengths are possible (for a fixed $d$), we can apply Ramsey's theorem for bicliques, and as a result assume that all path $P_{ij}$ have the same length $5\leq \ell_P \leq 6d+5$.
  (Formally, by applying Ramsey's theorem we need to pass to a smaller biclique, but this is no problem as $m_A$ and $m_B$ can be chosen arbitrarily large.)

  Next for an index set $J \subseteq [m_B]$, we define $X_i(J) \coloneq \{x_{ij} \mid j \in J\}$.
  Let $J_0 \coloneq [m_B]$ and consider the set
  \[
    A_1^+ \coloneq A_1 \cup X_1(J_0).
  \]
  The induced subgraph $G[A_1]$ has radius at most $d$ by $d$-shallowness of $\phi$.
  By construction, each $x_{1j}$ is adjacent to $A_1$ and hence $G[A^+_1]$ has radius at most $d+1$.
  Applying the fan lemma (\Cref{lem:fan-ramsey}) to the graph $G[A^+_1]$ and vertices $X_1(J_0)$, yields a large subset $J_1 \subseteq J_0$, such that $A_1^+$ contains an $\ell$-fan with leaves $X_1(J_1)$ for some $\ell \in [d+1]$.
  We can iterate this process on the remaining $A_i$ branchsets, where in each step we further restrict the eligible branchsets $B_j$, to those that all the previously extracted fans can reach through their leaves.
  Formally, the set $J_i \subseteq J_{i-1}$ is always created by applying the fan lemma to the graph $G[A_i^+]$ and vertices $X_i(J_{i-1})$, where $A_{i}^+ \coloneq A_i \cup X_{i}(J_{i-1})$.
  By choosing $m_A$ and $m_B$ sufficiently large, we can assume that after $m_A$ steps, we still have $m$ eligible indices $J \coloneq J_{m_A}$ left.
  Relabeling the indices such that $J = [m]$, we can repeat the argument symmetrically on the $B$-side using $I_0 \coloneq [m_A]$ and for all $j \in J$
  \[
    Y_j(I) \coloneq \{y_{ij} \mid i \in I\},\quad
    B_{j}^+ \coloneq B_j \cup Y_{j}(I_{j-1}).
  \]
  In conclusion, we obtain two size $m$ sets $I \subseteq [m_A]$ and $J \subseteq [m_B]$ such that for all $i\in I$ and $j\in J$
  \begin{itemize}
    \item each set $A_i \cup X_i(J)$ contains a fan $F_i^A$ with leaves $X_i(J)$, and
    \item each set $B_j \cup Y_j(I)$ contains a fan $F_j^B$ with leaves $Y_j(I)$.
  \end{itemize}
  By the pigeonhole principle, we can assume that all the $F_i^A$ are $\ell_A$-fans for the same length $\ell_A \in [d+1]$, and of the same type $\tau_A \in\{\mathrm{independent}, \mathrm{linked}\}$.
  We can assume the same for all the $F_j^B$ fans (for some possibly different length $\ell_B$ and type $\tau_B$).
  We claim that the subgraph of $G$ induced by the set
  \[
    \bigcup_{i \in I, j \in J} F_i^A \cup R_{ij} \cup F_j^B
  \]
  forms either $\biclique{r}{m}{m}$, $\biweb{r}{m}{m}$, or $\mixweb{r}{m}{m}$, depending on $\tau_A$ and $\tau_B$, for $r \coloneq \ell_A + (\ell_P - 4) + \ell_B$, which satisfies 
  \[
    3 \leq \underbrace{\ell_A + (\ell_P - 4) + \ell_B}_{r} \leq (d+1) + (6d+5 - 4) + (d+1) = 8d + 3 \leq 11d.
  \]
  To see this, note first that by minimality of $P_{ij}$, no vertex of $R_{ij}$ has a neighbor in $A_i$ or in $B_j$: an edge from such a vertex $v$ to some $u \in A_i$ would yield the strictly shorter $A_i$-$B_j$ path obtained by following $uv$ and then the suffix of $P_{ij}$ from $v$, still with all inner vertices in $S_{ij}$; the case of $B_j$ is symmetric.
  Since $\phi$ is induced, the only edges of $G$ between distinct branch sets are those between branch sets adjacent in $H$, that is, along $A_i \sim S^1_{ij} \sim S^2_{ij} \sim S^3_{ij} \sim B_j$.
  Hence, writing $c^A_i$ and $c^B_j$ for the centers of $F^A_i$ and $F^B_j$, concatenating 
  \begin{center}
    the $c^A_i$-$x_{ij}$ path in $F^A_i$ \qquad and \qquad
    $R_{ij}$ \qquad and \qquad
    the $y_{ij}$-$c^B_j$ path in $F^B_j$
  \end{center}
  produces an induced $c^A_i$-$c^B_j$ path $Q_{ij}$ with exactly $\ell_A + (\ell_P - 4) + \ell_B$ internal vertices, and for $(i,j) \neq (i',j')$ the paths $Q_{ij}$ and $Q_{i'j'}$ meet only in shared centers, with no edges between them other than those internal to a common fan.
  Such a fan is either independent (no edges between its paths) or linked (its first vertices form a clique, turning that center into a webbed one).
  Therefore the induced subgraph is $\biclique{r}{m}{m}$, $\biweb{r}{m}{m}$, or $\mixweb{r}{m}{m}$ according to whether none, both, or exactly one of $\tau_A, \tau_B$ is linked.
\end{proof}

\section{FSMF implies pattern-free}

\begin{lemma}\label{lem:fsmf-implies-pattern-free}
    For every hereditary FMSF graph class $\CC$ and every $r \geq 3$, there is $k \in \N$ such that $\CC$ excludes $\clique{r}{k}$ and $\web{r}{k}$.
\end{lemma}

We will prove the lemma by contraposition, showing how to extract shallow minors that are arbitrarily large and fat from the forbidden patterns.
Observe in \Cref{fig:patterns} that in $\clique{r}{k}$ and $\web{r}{k}$, we can find minors with fatness proportional to $r$, by choosing the branchsets of the edges to contain roughly the $r/2$ many middle vertices on the subdivided paths.
More precisely, we observe the following. 
(The constant $10$ can be chosen smaller, but its exact value is finicky to compute and does not matter to us.)

\begin{observation}\label{obs:encode-fat-minors}
    $\clique{r}{n}$ and $\web{r}{n}$ both contain every $n$-vertex graph as a $(\lfloor r/4 \rfloor + 10)$-shallow $(\lfloor r/2 \rfloor - 10)$-fat minor. 
\end{observation}

To show that a graph class $\CC$ is not FSMF, we have to show that for every $\delta \in \N$, there exists some $\alpha$ such that $\CC$ contains all graphs as $\delta$-fat $\alpha$-shallow minors.
By \Cref{obs:encode-fat-minors}, it suffices to show that $\CC$ contains $\clique{r}{n}$ or $\web{r}{n}$ for arbitrarily large values of both $n$ and $r$ simultaneously.
The next lemma shows how to ``boost'' from small values of $r$ to arbitrarily large ones.

\newcommand{\rsmall}{r_{\mathrm{small}}}
\newcommand{\rbig}{r_{\mathrm{big}}}
\newcommand{\rbigger}{r_{\mathrm{bigger}}}

\begin{lemma}
    For all naturals $\rbig \geq \rsmall \geq 3$ and every hereditary graph class $\CC$:
    \begin{itemize}
        \item If $\{\clique{\rsmall}{n} \mid n\in\N\} \subseteq \CC$, 
        then also $\{\clique{\rbigger}{n} \mid n\in\N\} \subseteq \CC$, for some $\rbigger \geq \rbig$.

        \item If $\{\web{\rsmall}{n} \mid n\in\N\} \subseteq \CC$, 
        then also $\{\web{\rbigger}{n} \mid n\in\N\} \subseteq \CC$, for some $\rbigger \geq \rbig$.
    \end{itemize}
\end{lemma}
\begin{proof}
    Assume $\CC$ satisfies the premise of the first bullet.
    There exists $\rbigger \in \N$ that satisfies $\rbigger \geq \rbig$ and can be obtained by iterating the function $f(x) \coloneq 2x+1$ a finite number of times on $\rsmall$.
    Combining items 1. and 3. of \Cref{obs:bipartize}, we get that
    \begin{center}
        $\clique{f(\rsmall)}{n}$ is an induced subgraph of $\clique{\rsmall}{2\binom{n}{2}}$.
    \end{center}
    Since $\CC$ contains $\clique{\rsmall}{n}$ for all $n\in\N$, we can iterate the above observation and conclude that $\CC$ also contains as induced subgraphs $\clique{\rbigger}{n}$ for all $n\in\N$.
    As $\CC$ is hereditary, these induced subgraphs are also contained as graphs in $\CC$.
    
    The second bullet follows analogously by combining items 2. and 4. of \Cref{obs:bipartize}, yielding
    \begin{center}
        $\web{2 \rsmall}{n}$ is an induced subgraph of $\web{\rsmall}{2\binom{n}{2}}$.
    \end{center}
\end{proof}

As discussed, this proves \Cref{lem:fsmf-implies-pattern-free}.

\newcommand{\Bb}{{\cal B}}
\newcommand{\rad}{\mathrm{rad}}
\newcommand{\cen}{\mathrm{cen}}

\section{Ball intersection graphs}

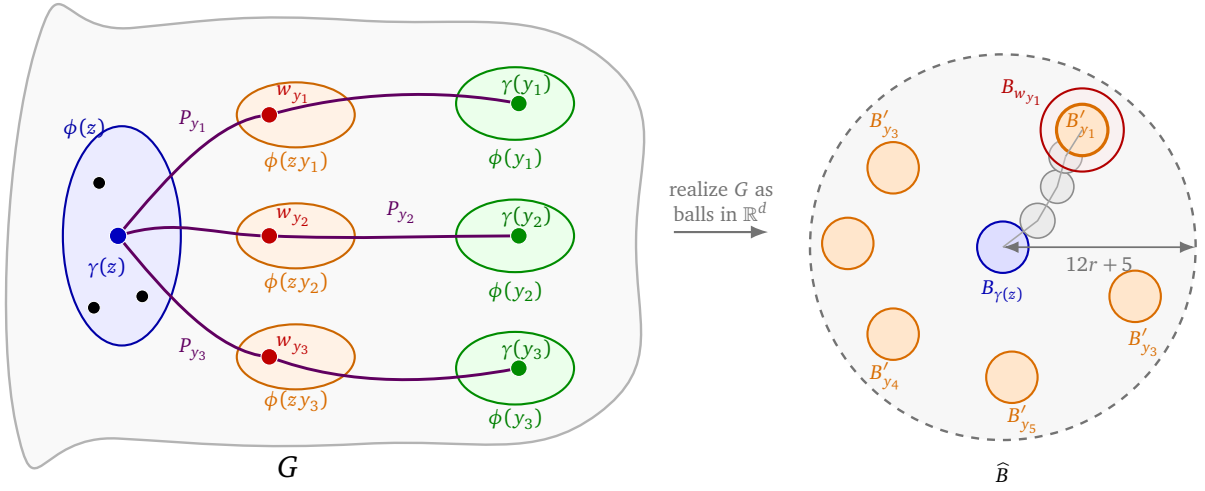
\begin{figure}[htbp]
    \centering
    \begin{tikzpicture}[
    >=Latex,
    every node/.style={font=\small},
    smalllabel/.style={font=\scriptsize},
    tinylabel/.style={font=\tiny},
    vertex/.style={
        circle,
        fill=black,
        draw=white,
        line width=.35pt,
        inner sep=1.7pt
    },
    gammaz/.style={
        circle,
        fill=blue!75!black,
        draw=white,
        line width=.4pt,
        inner sep=2.3pt
    },
    gammay/.style={
        circle,
        fill=green!55!black,
        draw=white,
        line width=.4pt,
        inner sep=2.2pt
    },
    wy/.style={
        circle,
        fill=red!75!black,
        draw=white,
        line width=.4pt,
        inner sep=2.2pt
    },
    phiz/.style={
        draw=blue!65!black,
        fill=blue!8,
        line width=.8pt
    },
    phiy/.style={
        draw=green!55!black,
        fill=green!8,
        line width=.8pt
    },
    phie/.style={
        draw=orange!80!black,
        fill=orange!10,
        line width=.8pt
    },
    modelpath/.style={
        violet!75!black,
        line width=1.1pt,
        line cap=round
    },
    ambient/.style={
        fill=gray!5,
        draw=black!30,
        line width=.9pt
    },
    widehatball/.style={
        draw=black!55,
        fill=gray!6,
        line width=.9pt,
        dashed
    },
    unitball/.style={
        draw=orange!85!black,
        fill=orange!20,
        line width=.8pt
    },
    centralball/.style={
        draw=blue!70!black,
        fill=blue!13,
        line width=.8pt
    },
    bigwyball/.style={
        draw=red!70!black,
        fill=red!10,
        fill opacity=.32,
        line width=.8pt
    },
    smallball/.style={
        draw=black!45,
        fill=black!8,
        line width=.6pt
    },
    arrow/.style={
        ->,
        line width=.8pt,
        black!55
    }
]

\begin{scope}


\path[ambient]
    plot[smooth cycle,tension=.85] coordinates {
        (-.35,.00) (-.10,4.95) (1.20,5.25)
        (4.15,5.35) (7.40,4.85) (7.85,2.45)
        (7.35,.15) (4.25,-.25) (1.00,-.18)
    };

\node[font=\large] at (3.25,-0.55) {$G$};

\draw[phiz] (1.05,2.50) ellipse (.78 and 1.45);
\node[smalllabel,blue!65!black] at (.55,3.92) {$\phi(z)$};

\draw[phie] (3.35,4.10) ellipse (.78 and .43);
\draw[phie] (3.35,2.50) ellipse (.78 and .43);
\draw[phie] (3.35,.90) ellipse (.78 and .43);

\node[smalllabel,orange!85!black] at (3.35,3.48) {$\phi(zy_1)$};
\node[smalllabel,orange!85!black] at (3.35,1.88) {$\phi(zy_2)$};
\node[smalllabel,orange!85!black] at (3.35,.32) {$\phi(zy_3)$};

\draw[phiy] (6.25,4.25) ellipse (.78 and .48);
\draw[phiy] (6.25,2.50) ellipse (.78 and .48);
\draw[phiy] (6.25,.75) ellipse (.78 and .48);

\node[smalllabel,green!50!black] at (6.25,3.52) {$\phi(y_1)$};
\node[smalllabel,green!50!black] at (6.25,1.77) {$\phi(y_2)$};
\node[smalllabel,green!50!black] at (6.25,.08) {$\phi(y_3)$};

\coordinate (gz)  at (1.00,2.50);

\coordinate (w1)  at (3.00,4.10);
\coordinate (w2)  at (3.00,2.50);
\coordinate (w3)  at (3.00,.90);

\coordinate (gy1) at (6.30,4.25);
\coordinate (gy2) at (6.30,2.50);
\coordinate (gy3) at (6.30,.75);

\node[vertex] at (.75,3.20) {};
\node[vertex] at (1.32,1.70) {};
\node[vertex] at (.68,1.55) {};

\draw[modelpath]
    (gz)
    .. controls (1.70,3.35) and (2.35,4.05) .. (w1)
    .. controls (4.05,4.40) and (5.25,4.45) .. (gy1);

\draw[modelpath]
    (gz)
    .. controls (1.80,2.75) and (2.28,2.50) .. (w2)
    .. controls (4.10,2.45) and (5.25,2.50) .. (gy2);

\draw[modelpath]
    (gz)
    .. controls (1.65,1.70) and (2.32,.95) .. (w3)
    .. controls (4.05,.50) and (5.25,.55) .. (gy3);

\node[gammaz] at (gz) {};
\node[wy] at (w1) {};
\node[wy] at (w2) {};
\node[wy] at (w3) {};

\node[gammay] at (gy1) {};
\node[gammay] at (gy2) {};
\node[gammay] at (gy3) {};

\node[smalllabel,blue!70!black] at ($(gz)+(-.15,-.42)$) {$\gamma(z)$};

\node[smalllabel,red!75!black] at ($(w1)+(.32,.27)$) {$w_{y_1}$};
\node[smalllabel,red!75!black] at ($(w2)+(.32,.25)$) {$w_{y_2}$};
\node[smalllabel,red!75!black] at ($(w3)+(0.32,.15)$) {$w_{y_3}$};

\node[smalllabel,green!45!black] at ($(gy1)+(.10,.25)$) {$\gamma(y_1)$};
\node[smalllabel,green!45!black] at ($(gy2)+(.10,.25)$) {$\gamma(y_2)$};
\node[smalllabel,green!45!black] at ($(gy3)+(.10,.25)$) {$\gamma(y_3)$};

\node[smalllabel,violet!75!black] at (2.0,4.02)
    {$P_{y_1}$};
\node[smalllabel,violet!75!black] at (4.75,2.82)
    {$P_{y_2}$};
\node[smalllabel,violet!75!black] at (2.0,1.0)
    {$P_{y_3}$};


\end{scope}

\draw[arrow] (8.35,2.55) -- (9.60,2.55)
    node[midway,above,smalllabel,align=center]
    {realize \(G\) as\\ balls in \(\mathbb R^d\)};

\begin{scope}[shift={(9.45,0)}]


\def\Rhat{2.55}
\def\Unit{.34}

\coordinate (C) at (3.25,2.35);

\draw[widehatball] (C) circle (\Rhat);
\node[smalllabel] at ($(C)+(0,\Rhat-5.53)$)
    {$\widehat B$};

\draw[centralball] (C) circle (\Unit);
\node[smalllabel,blue!70!black] at ($(C)+(0,-.58)$)
    {$B_{\gamma(z)}$};

\coordinate (q1) at ($(C)+(.46,.35)$);
\coordinate (q2) at ($(C)+(.72,.80)$);
\coordinate (q3) at ($(C)+(.83,1.20)$);

\draw[smallball] (q1) circle (.23);
\draw[smallball] (q2) circle (.22);
\draw[smallball] (q3) circle (.22);

\draw[black!35,line width=.6pt]
    (C) -- (q1) -- (q2) -- (q3);


\coordinate (u2) at ($(C)+(1.05,1.55)$);

\draw[bigwyball] (u2) circle (.55);
\node[smalllabel,red!70!black] at ($(u2)+(-.78,.50)$)
    {$B_{w_{y_1}}$};

\draw[unitball,very thick] (u2) circle (\Unit);
\node[smalllabel,orange!85!black] at ($(u2)+(0,.08)$)
    {$B'_{y_1}$};

\draw[black!35,line width=.6pt] (q3) -- (u2);

\coordinate (u1) at ($(C)+(-1.45,1.05)$);
\coordinate (u3) at ($(C)+(1.75,-.65)$);
\coordinate (u4) at ($(C)+(-1.45,-1.15)$);
\coordinate (u5) at ($(C)+(.12,-1.72)$);
\coordinate (u6) at ($(C)+(-2.05,.05)$);

\draw[unitball] (u1) circle (\Unit);
\draw[unitball] (u3) circle (\Unit);
\draw[unitball] (u4) circle (\Unit);
\draw[unitball] (u5) circle (\Unit);
\draw[unitball] (u6) circle (\Unit);

\node[smalllabel,orange!85!black] at ($(u1)+(-.10,.58)$) {$B'_{y_3}$};
\node[smalllabel,orange!85!black] at ($(u3)+(.12,-.60)$) {$B'_{y_3}$};
\node[smalllabel,orange!85!black] at ($(u4)+(-.10,-.55)$) {$B'_{y_4}$};
\node[smalllabel,orange!85!black] at ($(u5)+(.15,-.55)$) {$B'_{y_5}$};

\draw[<->,black!55,line width=.65pt]
    (C) -- ($(C)+(\Rhat,0)$)
    node[midway,below,smalllabel] {$12r+5$};


\end{scope}

\end{tikzpicture}
    \caption{(Left) paths from $\gamma(z)$ to $\{\gamma(y_i)\}_{i=1}^3$ in $G$. (Right) the ball $\widehat{B}$ contains  $\{B'_{y_i}\}_{i=1}^3$.}
    \label{fig:balls}
\end{figure}

Finally, in this section we prove that the intersection graphs of balls in Euclidean spaces are shallow-induced-minor-free.


\ballsismf*
\begin{proof}
    We prove that for every $r\in \N$, if $G$ is an intersection graph of a family of balls in $\R^d$ that contains $K_h^{(1)}$ -- the $1$-subdivision of $K_h$ --- as an $r$-shallow induced minor, then $h\leq 1+(12r+5)^d$. Let then  $\phi$ be an $r$-shallow induced model of $K_h^{(1)}$ in $G$. By a slight abuse of notation, for an edge $e\in E(K_h)$ by $\phi(e)$ we denote the branch set of the vertex subdividing $e$ in $K_h^{(1)}$.

    Let $\{B_u\colon u\in V(G)\}$ be a family of balls in $\R^d$ such that $B_u$ and $B_v$ intersect if and only if $u$ and $v$ are adjacent in $G$. By enlarging the balls slightly, we may assume that the balls $B_u$ have pairwise different radii. 

    For every $x\in V(K_h)$, we define $\gamma(x)\in \phi(x)$ as the vertex of $\phi(x)$ such that the ball $B_{\gamma(x)}$ has the largest radius among the balls $B_u$ with $u\in \phi(x)$. Further, let $z$ be the vertex of $K_h$ for which the ball $B_{\gamma(z)}$ has the smallest radius among the balls $\{B_{\gamma(x)}\colon x\in V(K_h)\}$.

    Since $\phi$ is $r$-shallow,
    for every vertex $y\in V(K_h)$, $y\neq z$, we may fix a $\gamma(z)$-$\gamma(y)$-path $P_{y}$ in $G$ of length at most $6r+2$ whose all vertices are contained in $\phi(z)\cup \phi(zy)\cup \phi(y)$. See \Cref{fig:balls}. Let $w_y$ be the first vertex encountered on $P_{y}$ when traversed it from $z$ to $y$ such that the radius of $B_{w_y}$ is larger than that of $B_z$; such a vertex exists, because the radius of $B_{\gamma(y)}$ is larger than the radius of $B_{\gamma(z)}$. Note that since $B_{\gamma(z)}$ has the largest radius among the balls assigned to the vertices of $\phi(z)$, we have $w_y\notin \phi(z)$. Consequently, $w_y\in \phi(yz)\cup \phi(y)$. Note that since $\phi$ is an induced model, this implies that the vertices $\{w_y\colon y\in V(K_h)\setminus \{z\}\}$ are pairwise different and non-adjacent, so the balls $\{B_{w_y}\colon y\in V(K_h)\setminus \{z\}\}$ are pairwise disjoint.

    By rescaling, we may assume that $B_{\gamma(z)}$ has radius $1$. Then for every $y\in V(K_h)\setminus \{z\}$, the ball $B_{w_y}$ has radius larger than $1$, hence we can find a ball $B'_y\subseteq B_{w_y}$ of radius $1$ that also intersects the ball assigned to the predecessor of $w_y$ on $P_{y}$. By the choice of $w_y$, all the balls assigned to the vertices of the prefix of $P_y$ from $z$ to $w_y$ (exclusive) have radii smaller than $1$. Since the length of $P_y$ is at most $6r+2$, it follows that $B'_y\subseteq \widehat{B}$, where $\widehat{B}$ is the ball with the same center as $B_{\gamma(z)}$, but with radius~$12r+5$. 

    We conclude that the balls $\{B_y'\colon y\in V(K_h)\setminus \{z\}\}$ are pairwise disjoint and contained in $\widehat{B}$. Since $\widehat{B}$ has volume $(12r+5)^d$ times larger than a ball of radius $1$, it follows that there can be at most $(12r+5)^d$ such balls. So $h\leq 1+(12r+5)^d$, as required.
\end{proof}

\bibliographystyle{plain}
\bibliography{ref}



\end{document}